\documentclass[aop,preprint]{imsart}

\RequirePackage{amsthm,amsmath,amsfonts,amssymb}
\RequirePackage{mathtools}

\RequirePackage[authoryear]{natbib}

\RequirePackage{graphicx}
\graphicspath{{figures/}}
\RequirePackage{subcaption}
\RequirePackage{tikz}
\usetikzlibrary{positioning,arrows.meta}

\RequirePackage{booktabs}
\RequirePackage{array}
\RequirePackage{xltabular}
\RequirePackage{ragged2e}

\RequirePackage{placeins}

\RequirePackage[colorlinks,citecolor=blue,urlcolor=blue,linkcolor=blue]{hyperref}

\startlocaldefs

\DeclareMathOperator{\Var}{Var}

\newcolumntype{Y}{>{\RaggedRight\hspace{0pt}\arraybackslash}X}
\newcolumntype{P}[1]{>{\RaggedRight\hspace{0pt}\arraybackslash}p{#1}}
\newcolumntype{L}[1]{>{\RaggedRight\hspace{0pt}\arraybackslash}p{#1}}

\theoremstyle{plain}
\newtheorem{theorem}{Theorem}[section]
\newtheorem{proposition}[theorem]{Proposition}
\newtheorem{lemma}[theorem]{Lemma}
\newtheorem{corollary}[theorem]{Corollary}

\theoremstyle{definition}
\newtheorem{remark}[theorem]{Remark}
\newtheorem{definition}[theorem]{Definition}

\newcounter{assumption}
\renewcommand{\theassumption}{A\arabic{assumption}}

\endlocaldefs
\begin{document}

\begin{frontmatter}

\title{Law of the Iterated Logarithm for Weighted Sums of Functionals of
Long-Memory Gaussian Sequences}

\runtitle{LIL for Weighted Long-Memory Gaussian Functionals}

\begin{aug}

\author[A]{\fnms{Elina}~\snm{Moldavskaya}
\ead[label=e1]{elina.mol@technion.ac.il}}

\address[A]{Technion--Israel Institute of Technology
\printead[presep={,\ }]{e1}}

\end{aug}

\begin{abstract}
We establish laws of the iterated logarithm for weighted sums of linear and
nonlinear functionals of long-memory stationary Gaussian sequences. The
leading Wiener chaos determines the scale of the almost-sure
fluctuations; all higher chaoses are negligible on it and the exponent
of the iterated-logarithm factor is half the Hermite rank. In the linear
case, we obtain the exact LIL constant, and it is the classical one. In
the nonlinear case, we identify a functional cluster set and give a
variational formula for the LIL constant. We also present explicit
upper and lower bounds for this constant. The lower bound is expressed
through beta functions and depends on the memory parameter, the weight
exponent, and the Hermite rank; numerically it matches the constant to
within a fraction of one percent, and can therefore be used in its
place.  Thus, the weights affect the
normalization, the geometry of the limiting cluster set, and the
nonlinear LIL constant itself.
\end{abstract}



\end{frontmatter}

\tableofcontents

\section{Introduction}

The \emph{law of the iterated logarithm} (LIL) is a fundamental
result in probability theory that describes the almost-sure size of
the largest fluctuations of partial sums of random variables.
It has been developed extensively in many directions,
including stationary Gaussian sequences, weighted sums of independent
or weakly dependent variables, and unweighted functionals of long-memory Gaussian
sequences. An LIL combining these features---that is, an LIL for weighted sums of functionals of long-memory Gaussian sequences---has not yet been studied.

Such models arise naturally in both theoretical and applied problems. On the theoretical side, weighted sums of strongly dependent random variables (or their transformations) appear in tapered spectral estimation, where deterministic weights reduce boundary effects, and in kernel-weighted constructions, where the weights localize the sum around a selected time point. 
In many applications, weighted rather than straight sums arise when different observations may have different levels of importance. For example, older observations may be discounted in forecasting, while more recent ones are assigned greater weight.

The need for a law of the iterated logarithm for weighted sums of this
kind---with strong dependence, nonlinearity, and deterministic
weighting---goes back to the author's study of constrained estimation
in long-memory regression models. In the approximate representation of the estimators
stated in \cite{Moldavskaya2009Lviv}, equation~(2), an LIL of this
type was needed to establish the iterated-logarithm order of a
weighted sum whose deterministic weights were generated by the second
derivatives of the regression function.

This raises the natural question: \emph{does a weighted LIL hold in the long-memory setting?} If so, \emph{what is its exact almost-sure scale, and what form should its functional cluster set take?}

This paper studies the LIL for weighted sums
\[
S_n=\sum_{t=1}^n a_tG(X_t),
\]
where $\{X_t\}_{t\ge1}$ is a stationary centered Gaussian sequence whose covariance function is regularly varying at infinity with index $-\alpha$, $\alpha\in(0,1)$, $G$ has Hermite rank $m\ge1$, and the deterministic weights are regularly varying,
\(
a_t=t^{-d}\ell_a(t),
\)
with $\ell_a$ slowly varying. 
The exponent $d$ is allowed to be of either sign: $d>0$ corresponds to downweighting later observations, while $d<0$ corresponds to upweighting them, as in the recency-weighted examples above. 
We work in the regime
\(
\alpha m<1,
2d+\alpha m<1,
\)
in which the leading Hermite component of \(G(X_t)\) retains long memory and the variance of the weighted partial sums \(S_n\) regularly varies with index \(2-2d-\alpha m>1\).

Among the sharpest statements in asymptotic probability theory, the LIL is highly sensitive to the mechanism that produces dependence. When the dependence is weak, the classical variance-normalized picture remains essentially stable. Under long-range dependence, however, this picture breaks down at a structural level: correlations decay so slowly that summability is lost, decorrelation is dramatically slower, and the leading asymptotic behavior is governed by the interaction between memory and nonlinearity.

For nonlinear functionals, the Hermite rank \(m\) determines the
first nonzero Wiener chaos and therefore the effective strength of dependence: the rate at which
the covariance of the leading Hermite component decays.
For \(m\ge2\), this component lies in a fixed higher Wiener
chaos and is non-Gaussian. As a result, the Gaussian
comparison and decorrelation techniques usually used to
control dependence between sparse exceedance events in LIL lower-bound arguments
no longer suffice: in a higher Wiener
chaos, covariance alone does not determine the joint behavior
of such events.

The deterministic weights introduce a further obstruction:
even though the Gaussian input \(\{X_t\}_{t\ge1}\) \emph{is stationary},
the weighted sequence \(\{a_tG(X_t)\}_{t\ge1}\) \emph{is
not} (unless the weights are constant). This loss of stationarity is structural: shifted blocks no longer have the same probabilistic
structure, and the usual stationarity-based LIL arguments
cannot be applied directly.

None of these three features---long memory, nonlinearity,
and deterministic weighting---is new in isolation in LIL theory. This trio is.

In the unweighted long-memory setting, the leading Wiener chaos is
governed by self-similar Hermite-type limits. With regularly varying
deterministic weights, the limiting spectral kernels take the form
\[
Q_t(x_1,\ldots,x_m)
=
C_{m,\alpha,d}
\left(
\int_0^t
y^{-d}
e^{iy(x_1+\cdots+x_m)}
\,dy
\right)
\prod_{j=1}^m |x_j|^{(\alpha-1)/2},
\qquad t\ge0,
\]
where \(C_{m,\alpha,d}\) is the signed normalization constant defined
below.

In the unweighted case \(a_t\equiv1\), one has \(d=0\), and the factor
\(y^{-d}\) disappears. When \(d\ne0\) (weighted settings), this factor produces a 
different \emph{new limiting object}. 
The associated multiple Wiener--It\^o integral process satisfies the
self-similarity scaling relation with index
\(H=1-d-\alpha m/2\), but for \(d\ne0\) it does not have stationary
increments; indeed, the variance of an increment over \([t,t+h]\)
depends on \(t\).  Consequently, the factor $y^{-d}$ deforms the multilinear limiting form, and hence both the functional cluster
set and the nonlinear LIL constant. 
Thus, the role of the weights is not merely to modify the
normalization; they enter the \emph{geometry} of the limit theorem
itself.

The laws of the iterated logarithm were established by \cite{LaiStout1978} for stationary
Gaussian sequences, and by \cite{Arcones1999} for functions of
stationary Gaussian vectors. For nonlinear functionals of
long-memory Gaussian sequences, the pioneering work of \cite{Taqqu1977} established the
functional LIL for \(m=1\), leaving the cases \(m\ge2\) open. The
higher-rank cases were subsequently treated by \cite{MoriOodaira1986,MoriOodaira1987} for self-similar
processes with stationary increments represented by multiple Wiener
integrals. Related functional LILs were obtained for long-memory
empirical processes by
\cite{DehlingTaqqu1988}, and exact LILs for subordinated Gaussian
sequences by
\cite{AzmoodehPeccatiPoly2016}. None of these results treats
deterministic weights. Deterministic weighted LILs for independent random variables
were obtained by \cite{LiTomkins1996}.

Our contributions are as follows. First, we identify the correct
weighted almost-sure scale
$
A_n(2\log\log n)^{m/2},
A_n^2=\operatorname{Var}(S_n),
$
where $A_n$ is regularly varying with index
$
H=1-d-\frac{\alpha m}{2}>\frac12 .
$
We prove that all Wiener-chaos components above order $m$ are
negligible on this scale (Theorem~\ref{thm:reduction}) and establish a
general upper bound at the same normalization
(Theorem~\ref{thm:upper}). Second, in the linear case $m=1$, we obtain
the sharp variance-normalized weighted LIL with the exact constant one
(Theorem~\ref{thm:linear-sharp-lil}).

Third, and mainly, for $m\ge2$ we prove a weighted functional LIL:
the normalized interpolated processes are almost surely relatively
compact in $C[0,T]$, and their cluster set is
$$
\mathcal K_T
=
\Bigl\{
t\mapsto
\bigl\langle Q_t,\xi^{\otimes m}
\bigr\rangle_{\mathfrak H_{\mathbb R}^{\otimes m}}
:
\|\xi\|_{\mathfrak H_{\mathbb R}}\le1
\Bigr\}
$$
(Theorem~\ref{thm:functional-nonlinear-lil}). 
As a consequence, evaluating the cluster
functions at \(t=1\) yields the law of the iterated logarithm itself,
\[
    \limsup_{n\to\infty}
    \frac{|S_n|}{A_n(2\log\log n)^{m/2}}
    =\Lambda_m
    \quad\text{a.s.},
    \quad
    \Lambda_m=\Lambda_m(\alpha,d)
:=
\sup_{\|\xi\|_{\mathfrak H_{\mathbb R}}\le1}
\bigl|
\bigl\langle Q_1,\xi^{\otimes m}
\bigr\rangle_{\mathfrak H_{\mathbb R}^{\otimes m}}
\bigr|
\]
with the
sharp nonlinear LIL constant $\Lambda_m,$ given by the explicit variational formula
(Corollary~\ref{cor:sharp-nonlinear-constant}). At
\(m=1\) the same formula returns \(\Lambda_1=1\), in agreement with
Theorem~\ref{thm:linear-sharp-lil}.

Moreover, for $m\ge2$,
$$
0<L_{m,\alpha,d}\le\Lambda_m<(m!)^{-1/2},
$$
where the lower bound is given in closed form through beta functions
(Lemma~\ref{lem:Lambda-bounds}). The
dependence on the model is real: $\Lambda_m$ varies with both the
memory parameter $\alpha$ and the weight exponent $d$, whereas the
ceiling $(m!)^{-1/2}$ is fixed by the normalization alone and retains
no information about either.
Remark~\ref{rem:rarity-of-constants} discusses how unusual such an
explicit description is.

The method is pathwise and functional, not distributional. A direct
second Borel--Cantelli argument for the nonlinear lower bound would
require dependence estimates for sparse exceedance events in a fixed
higher Wiener chaos---precisely where covariance alone is
insufficient. We bypass this by returning to the underlying Gaussian
spectral coordinates and exploiting their recurrence under dilation.
The rescaled weighted discrete kernels converge to the limiting
kernels \(Q_t\), uniformly in time and with separate control of their
global \(L^2\)-tails. Recurrence of finitely many dilated Gaussian
coordinates, established through a sharp Kochen--Stone argument,
yields the lower cluster functions by finite-rank approximation;
Borell--Tsirelson bounds, fixed-chaos maximal estimates, and a uniform
small-kernel bound yield the upper inclusion and relative
compactness. In the linear case the constant one is obtained
independently by a purely probabilistic argument, which provides a
nontrivial consistency check of the variational formula at \(m=1\).

\begin{remark}[On explicit LIL constants]
\label{rem:rarity-of-constants}
Closed-form expressions for nonlinear LIL constants are rarely
available: in most settings the theory determines the almost-sure
scale while the constant remains represented abstractly by a norm or a
variational problem. After variance normalization the linear case is
exceptional---classically for independent summands, and here for
\(m=1\): one has \(\Lambda_1=\|Q_1\|_{\mathfrak H_{\mathbb R}}=1\),
the supremum being attained at
\(\xi=Q_1/\|Q_1\|_{\mathfrak H_{\mathbb R}}\). For \(m\ge2\) no such
simplification is available: \(\Lambda_m\) is the spectral norm of the
symmetric \(m\)-linear form induced by \(Q_1\), whereas
\((m!)^{-1/2}\) is its Hilbert--Schmidt norm, and the inequality
\(\Lambda_m<(m!)^{-1/2}\) is strict because \(Q_1\) has rank greater
than one. We are not aware of a setting in which such a constant has
been evaluated in closed form. 

The present paper gives two results that are seldom available together
for \(m\ge2\): an exact variational characterization, together with the
explicit estimate 
\[0<L_{m,\alpha,d}\le\Lambda_m<(m!)^{-1/2}\]
whose
lower bound reproduces the computed values of \(\Lambda_m\) to within a
fraction of one percent over the parameter range examined in
Remark~\ref{rem:Lambda-numerics}, and is therefore usable as a
\emph{closed-form working value for the constant}.
The strict gap is a nonlinear
phenomenon. Its magnitude reflects the geometry of \(Q_1\) and hence
the memory and the weights.
\end{remark}

To clarify the position of the present paper within the existing LIL
literature, Table~\ref{tab:comparison-lil} summarizes and compares the present
weighted long-memory setting and several closely related earlier LIL works.
It explains why existing theories do not cover our settings and the weighted functional cluster theorem proved in this
paper. It also shows the novelty of the model,
assumptions, methods, results, and limiting objects.

\begingroup
\scriptsize
\setlength{\tabcolsep}{3pt}
\renewcommand{\arraystretch}{1.06}
\setlength{\emergencystretch}{3em}

\begin{xltabular}{\textwidth}{@{}P{2.55cm}Y Y@{}}
\caption{Comparison with related LIL results.}
\label{tab:comparison-lil}\\

\toprule
Reference
&
Existing setting and result
&
Relation to the present paper
\\
\midrule
\endfirsthead

\caption[]{Comparison with related LIL results continued.}\\
\toprule
Reference
&
Existing setting and result
&
Relation to the present paper
\\
\midrule
\endhead

\midrule
\multicolumn{3}{r}{\emph{Continued on the next page}}\\
\endfoot

\bottomrule
\endlastfoot

\cite{LaiStout1978}
&
Stationary Gaussian sequences, including strongly dependent ones; LIL
and upper--lower class results under general covariance conditions.
&
For \(G(x)=x\) and \(a_t\equiv1\), the present model belongs to this
framework. Nonconstant weights destroy stationarity, while for
\(m\ge2\) the leading chaos is non-Gaussian. Their theory therefore
does not yield the weighted higher-chaos cluster set.
\\[0.45em]
\cite{Arcones1999}
&
Compact LILs for unweighted Gaussian functionals on the classical
\((2n\log\log n)^{1/2}\) scale.
&
The Gaussian input may have nonsummable covariances, but the
conditions there place the transformed sums in the classical-scale regime.
Here \(\alpha m<1\) keeps the leading Hermite component long-range
dependent, while nonconstant weights destroy stationarity and produce
the \(Q_t\)-based cluster set.
\\[0.45em]
\cite{Taqqu1977}
&
Unweighted nonlinear functionals of long-memory Gaussian sequences;
a strong reduction to the leading Hermite chaos and a functional LIL
for \(m=1\), with the cases \(m\ge2\) left open.
&
For \(a_t\equiv1\) the present model reduces to this setting, and the
constant 1 at \(m=1\) agrees for both papers. Theorem~\ref{thm:reduction} extends
the reduction to regularly varying weights. For \(m\ge2\) a
weighted functional LIL is proved. For \(d\ne0\) the factor
\(y^{-d}\) changes the limiting kernel, and hence the cluster set and
the constant.
\\[0.45em]

\cite{MoriOodaira1986,MoriOodaira1987}
&
Functional LILs for self-similar processes with stationary increments
represented by multiple Wiener integrals, and for certain processes
converging to such limits.
&
For \(a_t\equiv1\), the present functional LIL for \(m\ge2\) recovers
their result. For \(d\ne0\), the limiting process has no stationary
increments and falls outside their class, so their theory does not
determine the weighted cluster set or the nonlinear LIL constant.
\\[0.45em]
\cite{DehlingTaqqu1988}
&
A functional LIL for the two-parameter empirical process of a
long-memory subordinated Gaussian sequence. Its limit factorizes into
a deterministic function of the level and a Hermite process in time.
&
The result concerns an empirical process rather than deterministically
weighted partial sums, and its limiting time component is generated
by the unweighted Hermite kernel. The weighted kernel \(Q_t\), and
hence the cluster set and LIL constant it determines, do not arise
there.
\\[0.45em]
\cite{AzmoodehPeccatiPoly2016}
&
Exact LILs obtained through uniform multidimensional Wasserstein
approximation by Gaussian vectors.
&
For \(m\ge2\) in the present long-memory regime, the leading term has
a non-Gaussian fixed-chaos limit, so this approximation method is
unavailable. Their results do not treat deterministic weights or yield
a weighted functional cluster set.
\\[0.45em]
\cite{LiTomkins1996}
&
Weighted sums of independent identically distributed random
variables. In the scalar case, the LIL constant is \(\|f\|_2\);
compact LILs are also obtained in a Banach-space setting.
&
After variance normalization, their scalar LIL has constant 1, as
does the present linear result for \(m=1\). For \(m\ge2\) and
\(d\ne0\), however, the weights enter the limiting kernel \(Q_t\) and
the nonlinear LIL constant, a phenomenon absent from the independent
setting.
\\[0.45em]
\textbf{Present paper}
&
Regularly varying weighted sums
\(S_n=\sum_{t\le n}a_tG(X_t)\) of functionals of a long-memory
Gaussian sequence; scale \(A_n(2\log\log n)^{m/2}\), reduction to the
leading chaos, a sharp linear LIL for \(m=1\), and a functional LIL
for \(m\ge2\).
&
Combines deterministic weighting, long memory, and a fixed higher
Wiener chaos. For \(d=0\) and \(\ell_a\equiv1\), \(Q_t\) reduces to
the unweighted Hermite kernel; for \(d\ne0\), the factor \(y^{-d}\)
changes the kernel and hence the cluster set and nonlinear LIL
constant.
\end{xltabular}

\endgroup

The paper is organized as follows. Section~\ref{sec:assumptions}
introduces the assumptions and notation, and
Section~\ref{sec:main-results} states the main results.
Section~\ref{sec:reduction} proves the reduction to the leading Hermite
chaos, and Section~\ref{sec:upper-bound} establishes the general upper
bound. Section~\ref{sec:linear} proves the sharp weighted LIL in the
linear case. Section~\ref{sec:spectral-realization} proves the nonlinear weighted
functional LIL and derives the variational formula for the nonlinear
LIL constant. Section~\ref{sec:numerical-illustrations}  presents numerical
illustrations and Section~\ref{sec:conclusion} concluding remarks. The
appendix collects the auxiliary variance, maximal, and spectral
estimates.

The three main theorems are logically independent: the general upper
bound is proved under the weakest assumptions and is not used in the
sharp results, and the linear case is treated by a route that does not
involve the cluster-set construction.
The proof has three independent components; see
Figure~\ref{fig:architecture}. First, the weighted variance and 
higher-chaos estimates reduce the problem to the leading Wiener chaos. Second, 
the spectral dilation identifies the limiting deterministic kernels. Third, the 
functional cluster theorem is proved by finite-dimensional recurrence, finite-rank 
approximation, and a small-kernel fixed-chaos estimate.
\section{Assumptions and notation}\label{sec:assumptions}

Throughout the paper, \(\{X_t\}_{t\ge1}\) denotes a centered stationary Gaussian sequence with
\(\mathbb E X_t=0\) and \(\mathbb E X_t^2=1\). Its covariance function is
\(
    \rho(h):=\mathbb E[X_tX_{t+h}], \quad h\ge0.
\)
When negative lags occur, we use the even extension
\(\rho(-h):=\rho(h)\).

\medskip
\refstepcounter{assumption}\label{A1}
\noindent\textbf{(\theassumption)}\; (Time-domain long-memory condition).
We assume that \(\{X_t\}\) has \emph{long memory}, also called  \emph{long-range dependence} or \emph{strong dependence} in the sense that
\(
    \rho(h)\sim c\,h^{-\alpha}L_0(h),
    h\to\infty,
\)
where \(c>0\), \(0<\alpha<1\), and \(L_0\) is a measurable eventually positive slowly varying function at infinity, that is,
 \(L_0(rx)/L_0(x)\to1\) as \(x\to\infty\) for every \(r>0\). In particular, the covariances are not summable.

\medskip
\refstepcounter{assumption}\label{A2}
\noindent\textbf{(\theassumption)}\; (Spectral-domain long-memory condition). 
We also assume that \(\{X_t\}\) admits a spectral density \(f\) on \((-\pi,\pi]\) such that
\(
    \rho(h)=\int_{-\pi}^{\pi}e^{ih\lambda}f(\lambda)\,d\lambda,
\)
and
\(
    f(\lambda)\sim c_f|\lambda|^{\alpha-1}L_0(1/|\lambda|),
     \lambda\to0,
\)
with \(c_f>0\), while \(f\) is bounded on every compact subset of \((-\pi,\pi]\setminus\{0\}\).

\medskip
\refstepcounter{assumption}\label{A3}
\noindent\textbf{(\theassumption)}\; (Hermite rank). Let \(G\in L^2(\mathbb R,\phi)\), where \(\phi\) is the standard Gaussian measure.
We assume that
\(G\) has \emph{Hermite rank \(m\ge1\)}, that is,
\(
    G(x)=\sum_{q=m}^\infty \frac{J_q}{q!}H_q(x),
    \quad J_m\neq0, 
\)
    where \(
J_q:=\int_{\mathbb R}G(x)H_q(x)\phi(dx)\)
 and \(H_q\) denotes the \(q\)-th Hermite polynomial,
\(
H_q(x)=(-1)^q e^{x^2/2}\frac{d^q}{dx^q}e^{-x^2/2}.
\)

\medskip
\refstepcounter{assumption}\label{A4}
\noindent\textbf{(\theassumption)}\; (Regular variation of the weights).The weights are assumed to be \emph{regularly varying}: for \(n\ge1\),
\(
    a_n=n^{-d}\ell_a(n),
\)
where \(d\in\mathbb R\) and \(\ell_a\) is \emph{slowly varying} at infinity and eventually positive. Equivalently,
\(a_n\) is regularly varying with index \(-d\).

\medskip
For the spectral kernel estimates we impose one additional regularity condition on the weights:

\refstepcounter{assumption}\label{A5}
\noindent\textbf{(\theassumption)}\; (Smoothness of the weight factor).
The function \(\ell_a\) is eventually continuously differentiable and satisfies
\(
\frac{x\ell_a'(x)}{\ell_a(x)}\longrightarrow0,\quad x\to\infty.
\)

\medskip
\refstepcounter{assumption}\label{A6}
\noindent\textbf{(\theassumption)}\; (High-frequency spectral regularity). For the spectral proof of the nonlinear functional theorem we assume that
\(
f\in W^{1,1}\bigl([-\pi,-\delta]\cup[\delta,\pi]\bigr)
\quad\text{for every }\delta\in(0,\pi).
\)

\begin{remark}[Relation between the time- and spectral-domain formulations of long memory]
Assumptions \((\ref{A1})\) and \((\ref{A2})\) are not equivalent. If both
hold with the same slowly varying function \(L_0\), then their constants
satisfy
\(
c=c_\alpha c_f,
\quad
c_\alpha:=2\Gamma(\alpha)\cos\left(\frac{\pi\alpha}{2}\right).
\)
This is the standard relation between the covariance and spectral-density
asymptotics; see
\cite{BFGK2013}.
Under additional Tauberian regularity conditions, the covariance and
spectral formulations are equivalent.
\end{remark}

\begin{remark}[Role of the additional regularity assumptions]
Assumptions \((\ref{A5})\) and \((\ref{A6})\) are used only in the spectral part
of the proof of the nonlinear functional theorem.

Assumption \((\ref{A5})\) is a smoothness condition on the slowly varying factor
in the weights. Its role is to exclude strongly oscillating slowly varying
factors and to ensure that the weights vary regularly enough for the Abel
summation estimates for weighted exponential sums. It is not used in the
variance asymptotics, in the reduction to the leading Wiener chaos, or in the
general upper bound.

Assumption \((\ref{A6})\) is a regularity condition away from the spectral
singularity at the origin. The long-memory behavior is determined by the
low-frequency singularity of \(f\) near zero; \((\ref{A6})\) is needed only to
control the high-frequency part of the prelimit spectral kernels. Thus it is a
technical high-frequency regularity assumption, not an additional long-memory
condition.

Assumption (\ref{A6}) is satisfied by the standard long-memory models.
For FARIMA$(p,\widetilde d,q)$ processes whose autoregressive and
moving-average polynomials have no zeros on the unit circle, and for
fractional Gaussian noise, the spectral density is infinitely
differentiable on $[-\pi,\pi]\setminus\{0\}$, hence belongs to
$W^{1,1}$ away from the origin.
\end{remark}

\medskip
We define the weighted partial sums by
\(
    S_n:=\sum_{u=1}^n a_uG(X_u),  n\ge1,
\)
their \emph{leading chaos} component by
\[
    S_{n,m}:=\frac{J_m}{m!}\sum_{u=1}^n a_uH_m(X_u),
\]
and the remainder by
\(
    R_n:=S_n-S_{n,m}.
\)
Further, let
\(
    A_n^2:=\operatorname{Var}(S_n)
\)
and
\(
    v_n:=n^H\ell_a(n)L_0(n)^{m/2}.
\)
Thus \(v_n\) is the deterministic regularly varying scale of the leading chaos, while \(A_n\) is the exact variance normalization.

\medskip
Throughout the paper we work in \emph{the long-memory} regime
\(
    \alpha m<1 \quad\text{and}\quad 2d+\alpha m<1.
\)
Denote:
\(
    H:=1-d-\frac{\alpha m}{2}.
\)

\begin{remark}\label{rem:boundary}
The restrictions $\alpha m<1$ and $2d+\alpha m<1$ delimit the parameter
range treated in this paper. The two boundary relations are of a
different nature and, except at $d=0$, lie in disjoint parts of the
boundary of this range: imposing $\alpha m=1$ while retaining
$2d+\alpha m<1$ forces $d<0$, whereas imposing $2d+\alpha m=1$ while
retaining $\alpha m<1$ forces $d>0$. The two relations coincide at
$d=0$.

At $\alpha m=1$, which since $0<\alpha<1$ can occur only for $m\ge2$,
the covariance sequence of the Hermite term $H_m(X_t)$ is critical:
\(
    \operatorname{Cov}\bigl(H_m(X_t),H_m(X_{t+h})\bigr)
    \sim
    m!\,c^m h^{-1}L_0(h)^m .
\)
The behavior of its critical tail is therefore measured by
\(
    B_m(n):=\sum_{h=1}^{n}h^{-1}L_0(h)^m,
\)
which is not determined by \((\ref{A1})\). For example, if
$L_0\equiv1$, then $B_m(n)\sim\log n$, whereas if
$L_0(h)=(\log(e+h))^{-2/m}$, then
$B_m(n)\to B_m(\infty)<\infty$, so that
$\sum_{h\ge1}|\rho(h)|^m<\infty$ even though
$\sum_{h\ge1}|\rho(h)|=\infty$. Accordingly, neither a universal
logarithmic correction at the variance level nor a particular limiting
regime for the leading chaos follows from the present assumptions.
This boundary requires a finer condition on $L_0$ and a separate
analysis.

At $2d+\alpha m=1$ with $\alpha m<1$, the situation is different.
Here $H=1/2$, but the covariance sequence of $H_m(X_t)$ remains
nonsummable: the equality $H=1/2$ is produced by the decay of the
weights, not by a transition to weak dependence. The kernel formula
itself remains nondegenerate at this boundary. Under $\alpha m<1$, the
integral $I_{m,\alpha,d}$ is finite whenever $2d+\alpha m<2$, and the
identity in Lemma~\ref{lem:limit-kernel-L2} gives
\(
    \|Q_t\|_{L^2(\mathbb R^m)}^2=\frac{t}{m!}
\)
when $2d+\alpha m=1$. For $m\ge2$, these facts provide no basis for
inferring a Gaussian limiting regime or for replacing the
iterated-logarithm exponent $m/2$ by $1/2$. They do not, however,
establish the functional LIL at the boundary.

We state our results away from both boundaries and make no claim about
the corresponding LILs there.
\end{remark}


\begin{proposition}[Exact variance asymptotics]\label{prop:exact-variance}
Assume \((\ref{A1})\), \((\ref{A3})\), and \((\ref{A4})\), \(\alpha m<1\), and \(2d+\alpha m<1\). Then
\[
    \operatorname{Var}(S_{n,m})\sim\sigma_m^2v_n^2.
\]
where
\[
\sigma_m^2
=
\frac{J_m^2c^m}{m!}
\int_0^1\int_0^1
x^{-d}y^{-d}|x-y|^{-\alpha m}\,dx\,dy.
\]
The integral is finite. Moreover,
\(
A_n^2=\operatorname{Var}(S_n)\sim \operatorname{Var}(S_{n,m}).
\)

Consequently, \(A_n\sim\sigma_m v_n\), and \(A_n\) is regularly varying with index \(H\).
\end{proposition}
The restriction \(2d+\alpha m<1\) is the regime needed for the subsequent results; the variance asymptotic itself can be extended to a wider range, but this extension is not used here.

Under \(\alpha m<1\) and \(2d+\alpha m<1\), the natural normalization is
\(
    A_n(2\log\log n)^{m/2}.
\)

The proof is given in the Appendix.

We use Potter's bound for slowly varying functions in the standard form
\cite[Theorem 1.5.6]{BGT1987}. 
\begin{remark}[Use of Potter bounds]
Throughout the paper, whenever Potter's bound is applied to a regularly varying
sequence, we use the following convention. The bound is first applied for all
arguments larger than a fixed threshold depending on the Potter exponent. The
finitely many remaining bounded arguments are then absorbed into the constant.
In particular, after increasing the constant \(C_\eta\), we may use estimates of
the form
\(
    |\rho(h)|^m
    \le
    C_\eta(1+h)^{-\alpha m+\eta},
     h\ge0,
\)
and analogous Potter-type estimates for the weight sequence \(a_n=n^{-d}\ell_a(n)\),
uniformly over the ranges considered in the sequel.
\end{remark}

We also use the standard Mills
bounds for the normal tail.

Finally, for nonnegative functions \(r\) and \(s\), \(r(x)\asymp s(x)\) means that there exist constants \(0<C_1<C_2<\infty\) such that \(C_1s(x)\le r(x)\le C_2s(x)\) for all sufficiently large \(x\).

\section{Main results}\label{sec:main-results}

We now state the main results. The first two results reduce the problem to the leading
Wiener chaos and give the general upper bound. The last two results distinguish the linear
and nonlinear regimes. Throughout this section we assume regime
\(\alpha m<1\), \(2d+\alpha m<1\). The reduction Theorem \ref{thm:reduction}, the upper- bound Theorem \ref{thm:upper}, and the linear LIL Theorem \ref{thm:linear-sharp-lil} use only
\((\ref{A1})\), \((\ref{A3})\), and \((\ref{A4})\). The spectral assumption
\((\ref{A2})\), the smoothness assumption \((\ref{A5})\), and the high-frequency
regularity assumption \((\ref{A6})\) are imposed for the nonlinear functional
Theorem \ref{thm:functional-nonlinear-lil}.

We first reduce the weighted sums to the leading Wiener chaos.

\begin{theorem}[Reduction to the leading chaos]\label{thm:reduction}
We have
\[
\limsup_{n\to\infty}
\frac{|R_n|}{A_n(2\log\log n)^{m/2}}
=0
\qquad\text{\rm a.s.}
\]
Consequently, the LIL on this scale is determined by the leading chaos term \(S_{n,m}\).
\end{theorem}

Our next result gives the upper bound on the natural scale.

\begin{theorem}[Upper bound]\label{thm:upper} 
There exists \(C_{\mathrm{up}}<\infty\) such that
\[
\limsup_{n\to\infty}
\frac{|S_n|}{A_n(2\log\log n)^{m/2}}
\le C_{\mathrm{up}}
\qquad\text{\rm a.s.}
\]
\end{theorem}

\begin{theorem}[Sharp weighted LIL in the linear case]\label{thm:linear-sharp-lil}
Let \(m=1\). Then
\[
\limsup_{n\to\infty}
\frac{S_n}{A_n(2\log\log n)^{1/2}}
=1,
\qquad
\liminf_{n\to\infty}
\frac{S_n}{A_n(2\log\log n)^{1/2}}
=-1
\quad\text{a.s.}
\]
Consequently,
\[
\limsup_{n\to\infty}
\frac{|S_n|}{A_n(2\log\log n)^{1/2}}
=1
\quad\text{a.s.}
\]
\end{theorem}


\begin{remark}[The unweighted fractional-Gaussian-noise benchmark]
\label{rem:unweighted-fgn-constant}
In the special unweighted linear case \(a_t\equiv1\), \(G(x)=x\),
\(X_t=B_{H_0}(t)-B_{H_0}(t-1)\), where \(B_{H_0}\) is fractional Brownian
motion with \(\operatorname{Var}(B_{H_0}(1))=1\), we have
\(S_n=B_{H_0}(n)\), so \(A_n=n^{H_0}\). Here
\(
    \alpha=2-2H_0,\ c=H_0(2H_0-1),\ L_0\equiv1,
\)
so the general exponent \(H=1-d-\alpha m/2\) reduces to \(H_0\), and since
\(
\int_0^1\int_0^1|x-y|^{-\alpha}\,dx\,dy=\tfrac1{H_0(2H_0-1)}=1/c,
\)
Proposition~\ref{prop:exact-variance} gives \(\sigma_1^2=1\) exactly. This is
a consistency check: the general normalization correctly reduces to the
\emph{classical Hurst} exponent and unit constant in this benchmark case.

The classical LIL for fractional Brownian motion, equivalently the linear case of
Taqqu's long-memory LIL \citep{Taqqu1977}, gives
\[
    \limsup_{n\to\infty}
    \frac{|B_{H_0}(n)|}
    {n^{H_0}(2\log\log n)^{1/2}}
    =
    1
    \qquad\text{\rm a.s.}
\]
Therefore, in this benchmark case, the deterministic constant in
Theorem~\ref{thm:linear-sharp-lil} is \(1\).
\end{remark}

\begin{definition}[Cluster set in \(C\lbrack 0,T\rbrack\)]
For a sequence \((x_n)_{n\ge3}\) in \(C[0,T]\), we denote by
\(
    \operatorname{Cl}_{C[0,T]}(x_n)
\)
the set of all subsequential limits of \((x_n)\) in the uniform norm on
\(C[0,T]\).
\end{definition}

For the nonlinear theorem we use the real spectral Hilbert space
\(
\mathfrak H_{\mathbb R}
:=
\{h\in L^2(\mathbb R;\mathbb C): h(-x)=\overline{h(x)}\},
\)
endowed with the real inner product
\[
\langle h,g\rangle_{\mathfrak H_{\mathbb R}}
:=
\operatorname{Re}\int_{\mathbb R} h(x)\overline{g(x)}\,dx.
\]
The corresponding symmetric tensor products are taken over this real Hilbert
space.

Let
\[
I_{m,\alpha,d}
:=
\int_0^1\int_0^1
y^{-d}z^{-d}|y-z|^{-\alpha m}\,dy\,dz= \frac{
        2\,\mathrm B(1-d,\,1-\alpha m)
    }
    {
        2-2d-\alpha m
    }.
\]
Here \(\mathrm B(\cdot,\cdot)\) denotes the beta function. .

By Proposition \ref{prop:exact-variance},
\(
\sigma_m^2
=
\frac{J_m^2c^m}{m!}I_{m,\alpha,d},
\quad
\sigma_m>0,
\)
where \(\sigma_m\) denotes the positive square root. We define the signed
normalization constant
\(
C_{m,\alpha,d}
:=
\frac{J_m c_f^{m/2}}{m!\sigma_m}.
\)
Equivalently, using \(c=c_\alpha c_f\), where
\(
c_\alpha=2\Gamma(\alpha)\cos\left(\frac{\pi\alpha}{2}\right),
\)
one may write
\[
C_{m,\alpha,d}
=
\frac{\operatorname{sgn}(J_m)}
{\sqrt{m!}\,c_\alpha^{m/2}\,I_{m,\alpha,d}^{1/2}}.
\]
In particular, the sign of \(J_m\) is included in \(C_{m,\alpha,d}\).

The limiting kernels are
\[
Q_t(x_1,\ldots,x_m)
=
C_{m,\alpha,d}
\left(
\int_0^t y^{-d}e^{iy(x_1+\cdots+x_m)}\,dy
\right)
\prod_{j=1}^m |x_j|^{(\alpha-1)/2},
\qquad t\ge0.
\]
The kernels are understood as elements of the real symmetric tensor space
\(\mathfrak H_{\mathbb R}^{\odot m}\).

For \(T>0\), define
\[
    \mathcal K_T
    :=
    \left\{
        t\mapsto
        \big\langle Q_t,\xi^{\otimes m}\big\rangle_{\mathfrak H_{\mathbb R}^{\otimes m}}
        \,:\,
        \xi\in\mathfrak H_{\mathbb R},\ \|\xi\|_{\mathfrak H_{\mathbb R}}\le1
    \right\}
    \subset C[0,T].
\]

\begin{remark}[Compactness of the limiting cluster set]
\label{rem:KT-compact}
The set \(\mathcal K_T\) is compact in \(C[0,T]\); see
Lemma~\ref{lem:KT-compact} in the Appendix.
\end{remark}

We also set \(S_0:=0\) and use the linearly interpolated process
\[
    \mathcal S_n(t)
    :=
    \bigl(1-\{nt\}\bigr)S_{\lfloor nt\rfloor}
    +
    \{nt\}S_{\lfloor nt\rfloor+1},
    \qquad 0\le t\le T.
\]

\begin{theorem}[Weighted functional LIL in the nonlinear case]
\label{thm:functional-nonlinear-lil}
Let \(m\ge2\). Assume \((\ref{A1})\)--\((\ref{A6})\), \(\alpha m<1\), and
\(2d+\alpha m<1\). Then, for every fixed \(T>0\), the sequence
\[
    \left\{
        \frac{\mathcal S_n(\cdot)}
        {A_n(2\log\log n)^{m/2}}
        : n\ge3
    \right\}
\]
is almost surely relatively compact in \(C[0,T]\), and
\[
    \operatorname{Cl}_{C[0,T]}
    \left(
        \frac{\mathcal S_n(\cdot)}
        {A_n(2\log\log n)^{m/2}}
    \right)
    =
    \mathcal K_T
    \qquad\text{\rm a.s.}
\]
\end{theorem}

\begin{remark}[Unweighted higher-chaos benchmark]
When \(d=0\) and \(\ell_a\equiv1\), we have \(a_t\equiv1\),
and \(Q_t\) reduces to the variance-normalized spectral kernel
of the order-\(m\) Hermite process. Moreover,
\[
H=1-\frac{\alpha m}{2},
\qquad
I_{m,\alpha,0}
=
\frac{2}{(1-\alpha m)(2-\alpha m)}
=
\frac{1}{H(2H-1)},
\]
and hence
\[
A_n
\sim
\sigma_m n^H L_0(n)^{m/2}.
\]
Under the Fourier identification of the time- and spectral-domain
representations, the cluster set \(\mathcal K_T\) agrees, after
variance normalization, with the cluster set obtained in
\cite{MoriOodaira1986,MoriOodaira1987}. Thus
Theorem~\ref{thm:functional-nonlinear-lil} recovers their
unweighted functional LIL for \(m\ge2\) on the common range
of assumptions.
\end{remark}

\begin{corollary}[Sharp nonlinear weighted LIL constant]
\label{cor:sharp-nonlinear-constant}
Under the assumptions of Theorem~\ref{thm:functional-nonlinear-lil},
\[
    \limsup_{n\to\infty}
    \frac{|S_n|}
    {A_n(2\log\log n)^{m/2}}
    =
    \Lambda_m
    \qquad\text{\rm a.s.},
\]
where the deterministic constant is given by the variational formula
\[
    \Lambda_m
    =
    \sup_{\|\xi\|_{\mathfrak H_{\mathbb R}}\le1}
    \left|
        \big\langle Q_1,\xi^{\otimes m}
        \big\rangle_{\mathfrak H_{\mathbb R}^{\otimes m}}
    \right|.
\text{      Moreover, }
    0<L_{m,\alpha,d} 
    \le
    \Lambda_m
    \le
    \frac{1}{\sqrt{m!}},
\]
where
\[
    L_{m,\alpha,d}
    :=
    \left(\frac{2-\alpha}{2(1-\alpha)}\right)^{m/2}
    \frac{
        \displaystyle
        \sum_{k=0}^{m}
        \binom{m}{k}
        \mathrm B\!\left(
            1-d+(1-\alpha)k,\,
            1+(1-\alpha)(m-k)
        \right)
    }
    {
        \sqrt{m!\,I_{m,\alpha,d}}
    },
\]
and
\[
    I_{m,\alpha,d}
    =
    \frac{
        2\,\mathrm B(1-d,\,1-\alpha m)
    }
    {
        2-2d-\alpha m
    }.
\]
Here \(\mathrm B(\cdot,\cdot)\) denotes the beta function.
\end{corollary}

\begin{remark}\label{rem:Lambda}
For fixed $m\ge1$, $\Lambda_m=\sup_{\|\xi\|_{\mathfrak H_{\mathbb R}}\le1}
|\langle Q_1,\xi^{\otimes m}\rangle_{\mathfrak H_{\mathbb R}^{\otimes m}}|$
is the spectral norm of the symmetric $m$-linear form induced by $Q_1$.
In particular, $\mathbb E[I_m(Q_1;W)^2]=m!\|Q_1\|_2^2=1$, so the
limiting chaos has unit variance, consistently with
Proposition~\ref{prop:exact-variance}. For $m=1$, $\xi^{\otimes1}=\xi$,
so the variational formula reduces exactly to the Hilbert-space norm,
$\Lambda_1=\|Q_1\|_{L^2(\mathbb R)}$, which equals $1$ after the
normalization built into $C_{1,\alpha,d}$.

 For $m\ge2$, $\Lambda_m$ is the spectral norm, whereas
$\|Q_1\|_{L^2}$ is the Hilbert--Schmidt norm of the same symmetric
tensor. By Lemma~\ref{lem:KT-compact}, applied with \(T=1\), the
supremum defining $\Lambda_m$ is attained. Hence equality in the
Cauchy--Schwarz inequality shows that
$\Lambda_m=(m!)^{-1/2}$ would hold if and only if $Q_1$ were a
rank-one symmetric tensor,
$Q_1=\lambda\,\xi^{\otimes m}$ for some unit vector $\xi$.

Set \(\Psi(u):=\int_0^1 y^{-d}e^{iyu}\,dy\).
If \(Q_1\) were rank one, then, after dividing by
\(\prod_j |x_j|^{(\alpha-1)/2}\) and fixing \(m-2\) variables, one
would obtain \(\Psi(u+v)=a(u)b(v)\) a.e. Hence
\(
\Psi(u+v)\Psi(u'+v')
=
\Psi(u+v')\Psi(u'+v)
\)
a.e. Since \(d<1\), the function \(\Psi\) is continuous, so this
identity holds everywhere. Taking \(u'=v'=0\) and using
\(\Psi(0)=\int_0^1 y^{-d}\,dy=(1-d)^{-1}\ne0\), we obtain the
continuous exponential Cauchy equation, and therefore
\(\Psi(u)=\Psi(0)e^{\lambda u}\) for some \(\lambda\in\mathbb C\).
This is impossible, since the Riemann--Lebesgue lemma gives
\(\Psi(u)\to0\) as \(u\to\pm\infty\), whereas a nonzero exponential
cannot tend to zero in both directions.
\end{remark}

\begin{remark}[Where the memory and the weights enter the constants]
\label{rem:upper-bound-universal}
It is instructive to record precisely where the memory parameter
$\alpha$, the weight exponent $d$, and the slowly varying factor
$\ell_a$ enter the almost-sure asymptotics, and where they do not.

At first order, the entire dependence is carried by the scale
$A_n(2\log\log n)^{m/2}$: both the regular-variation index
$H=1-d-\alpha m/2$ and the constant $\sigma_m$ of
Proposition~\ref{prop:exact-variance} depend on all model ingredients.
The constant $\Lambda_m$ describes the second-order, purely geometric
effect that remains after this exact normalization.

For $m=1$ this residual constant is universal, $\Lambda_1=1$, as in
the classical Gaussian theory. The nonlinear case is different:
for $m\ge2$ the constant $\Lambda_m$ is no longer universal, but varies
with the model.
Numerically, on the parameter grid examined, with the effective memory
ranging up to \(\alpha m=0.9\): \(\Lambda_2\) decreases from \(0.702\) at
\((\alpha,d)=(0.10,0)\) to \(0.391\) at \((\alpha,d)=(0.45,0.02)\), and
\(\Lambda_3\) from \(0.399\) at \((\alpha,d)=(0.10,0)\) to \(0.210\) at
\((\alpha,d)=(0.30,0.02)\);
and the explicit lower bound of
Lemma~\ref{lem:Lambda-bounds} reproduces these values
to within a fraction of one percent. Analytically, the lemma shows only that,
for fixed $d\le0$,
\(
    L_{m,\alpha,d}\asymp(1-\alpha m)^{1/2}
    \qquad\text{as }\alpha m\uparrow1;
\)
it does not establish the corresponding asymptotic rate for
$\Lambda_m$. Thus the weighted long-memory
structure enters the LIL twice: through the scale, and -- for
$m\ge2$ -- through the value of the constant itself.

Only the ceiling $(m!)^{-1/2}$ is model-free, and necessarily so: by
Lemma~\ref{lem:Lambda-bounds} the normalization fixes
$\|Q_1\|_{L^2(\mathbb R^m)}=(m!)^{-1/2}$ for every admissible
$(\alpha,d,\ell_a)$, and the bound
$\Lambda_m\le\|Q_1\|_{L^2}$ is the general domination of the spectral
norm of a symmetric tensor by its Hilbert--Schmidt norm. The
distance of $\Lambda_m$ to this ceiling measures how far $Q_1$ is
from a rank-one tensor, i.e., how strongly the factor $y^{-d}$ and the
low-frequency singularity mix the extremal directions; by
Remark~\ref{rem:Lambda}, for $m\ge2$ this distance is strictly
positive. In this sense the pair
$\bigl(L_{m,\alpha,d},(m!)^{-1/2}\bigr)$ separates the two roles: the
lower bound carries the model, the upper bound carries the
normalization.
\end{remark}

\begin{remark}[Two independent routes to the linear constant]\label{2 routes}
The value $\Lambda_1=1$ predicted here is reproduced independently, by an
entirely different route, in the proof of Theorem~\ref{thm:linear-sharp-lil}:
hypercontractive tail bounds, Borel--Cantelli on a geometric subsequence,
and a Kochen--Stone argument, with no reference to the Hilbert space
$\mathfrak H_{\mathbb R}$ or to the cluster-set construction
built for \(m\ge2\). It is therefore a nontrivial check that the same value
is predicted by the general variational formula of
Corollary~\ref{cor:sharp-nonlinear-constant}.
The agreement of two structurally unrelated
arguments -- one purely probabilistic, the other variational --
at \(m=1\) is a consistency check on the
normalization $A_n$, complementing the fractional-Gaussian-noise benchmark
of Remark~\ref{rem:unweighted-fgn-constant}.
\end{remark}

\FloatBarrier
\begin{figure}[!htbp]
\centering
\begin{tikzpicture}[
  x=1mm, y=1mm,
  every node/.style={font=\scriptsize},
  box/.style={draw, rounded corners=2pt, align=center,
              inner sep=2.6pt, text width=#1mm, anchor=center},
  sidebox/.style={draw, dashed, rounded corners=2pt, align=center,
                  inner sep=2.6pt, text width=40mm, anchor=center},
  main/.style={draw, thick, rounded corners=2pt, align=center,
               inner sep=3pt, text width=112mm, anchor=center},
  flow/.style={-{Stealth[length=1.8mm]}, thick},
  weak/.style={-{Stealth[length=1.8mm]}, dashed}
]

\node[box=92] (A) at (-10,0)
  {Weighted sums \(S_n=\sum_{t\le n}a_t G(X_t)\)\\[1pt]
   long memory (A1), Hermite rank \(m\) (A3), regularly varying weights (A4)};

\node[box=68] (B) at (-10,-19)
  {Reduction to the leading chaos [Thm~\ref{thm:reduction}]\\[1pt]
   \(R_n=o\bigl(A_n(2\log\log n)^{m/2}\bigr)\) a.s.\\[1pt]
   Lemmas~\ref{lem:two-dimensional-rv-riemann}--\ref{lem:dyadic-maximal}};

\node[sidebox] (U) at (52,-19)
  {\emph{Side result, not used below:}\\[1pt]
   universal upper bound [Thm~\ref{thm:upper}]\\[1pt]
   Lemmas~\ref{lem:leading-chaos-local-variance}, \ref{lem:fixed-chaos-maximal-tail}};

\node[box=68] (C) at (-10,-37)
  {Leading chaos \(S_{n,m}\) on the scale \(A_n(2\log\log n)^{m/2}\)};

\node[box=42] (L) at (-50,-58)
  {\(m=1\): sharp linear LIL, constant~\(1\) [Thm~\ref{thm:linear-sharp-lil}]\\[1pt]
   Lem~\ref{lem:linear-var-corr}; Lemmas~\ref{lem:fixed-chaos-maximal-tail}---
   \ref{lem:finite-dimensional-sparse-recurrence} (Kochen--Stone)};

\node[box=62] (R) at (14,-56)
  {\(m\ge2\): spectral realization \(X_t=W(h_t)\),\\
   dilations \(W_n=W\circ U_n\)};

\node[box=62] (K) at (14,-77)
  {Uniform kernel convergence [Prop~\ref{prop:full-kernel-convergence}]\\[1pt]
   \(\sup_t\|\widetilde F^{\mathrm{full}}_{n,t}-Q_t\|_2\to0\)\\[1pt]
   Lem./Cor.~\ref{lem:spectral-factor-rescaled}--\ref{lem:uniform-abel-weighted-sum},
   Prop~\ref{prop:full-prelimit-tail};\\
   uses (A2), (A5), (A6)};

\node[box=54] (LI) at (-18,-102)
  {Lower inclusion \(\mathcal K_T\subset\mathrm{Cl}\)
   [Prop~\ref{prop:functional-lower-inclusion}]\\[1pt]
   sparse recurrence (Prop~\ref{cor:dilation-coordinate-recurrence},
   Cor~\ref{cor:finite-rank-chaos-recurrence})\\[1pt]
   finite rank (Lem~\ref{lem:finite-rank-lower-cluster},
   \ref{lem:uniform-kernel-recurrence})\\[1pt]
   small kernels (Lem~\ref{lem:uniform-small-kernel-chaos})};

\node[box=54] (UI) at (42,-102)
  {Upper inclusion and compactness
   [Prop~\ref{prop:functional-upper-inclusion}]\\[1pt]
  Borell–Tsirelson (Lem~\ref{lem:finite-dimensional-upper-cluster},
   \ref{lem:finite-rank-chaos-upper})\\[1pt]
   gap control (Lem~\ref{lem:functional-grid-gap},
   \ref{lem:fixed-chaos-maximal-tail}, \ref{lem:full-leading-chaos-increment})};

\node[main] (T) at (12,-127)
  {Functional cluster theorem [Thm~\ref{thm:functional-nonlinear-lil}]:\ \
   \(\mathrm{Cl}_{C[0,T]}=\mathcal K_T\)
   \ (with Prop~\ref{prop:functional-remainder})\\[2pt]
   \(\Rightarrow\) sharp constant
   \(\Lambda_m=\sup_{\|\xi\|\le1}\bigl|\langle Q_1,\xi^{\otimes m}\rangle\bigr|\)
   [Cor~\ref{cor:sharp-nonlinear-constant}],\ \ bounds
   \(0<L_{m,\alpha,d}\le\Lambda_m\le(m!)^{-1/2}\) [Lem~\ref{lem:Lambda-bounds}]};

\draw[flow] (A) -- (B);
\draw[flow] (B) -- (C);
\draw[flow] (B.east) -- (U.west);
\draw[flow] (C.south) -- ++(0,-6) -| (L.north)
      node[pos=0.92,left=0.5mm,font=\tiny] {\(m=1\)};
\draw[flow] (C.south) -- ++(0,-6) -| (R.north)
      node[pos=0.92,right=0.5mm,font=\tiny] {\(m\ge2\)};
\draw[flow] (R) -- (K);
\draw[flow] (K.south) -- ++(0,-7) -| (LI.north);
\draw[flow] (K.south) -- ++(0,-7) -| (UI.north);
\draw[flow] (LI.south) -- ++(0,-7) -| (T.north);
\draw[flow] (UI.south) -- ++(0,-7) -| (T.north);
\draw[weak] (L.south) |- (T.west)
      node[pos=0.35,left=0.5mm,font=\tiny,align=center]
      {consistency\\ check\\ \(\Lambda_1=1\)};
\end{tikzpicture}
\caption{Architecture of the proof. Solid arrows indicate logical
dependence. The universal upper bound
(Theorem~\ref{thm:upper}, dashed frame) is a side result, proved under
the weaker assumptions (A1), (A3), (A4) alone, and is not used in the
proofs of Theorems~\ref{thm:linear-sharp-lil}
and~\ref{thm:functional-nonlinear-lil}. The dashed path on the left
records the consistency check \(\Lambda_1=1\)
(Remarks~\ref{rem:Lambda} and~\ref{2 routes}). The
bottom block records the two-sided bounds of
Lemma~\ref{lem:Lambda-bounds}, evaluated numerically in
Remark~\ref{rem:Lambda-numerics}.}
\label{fig:architecture}
\end{figure}
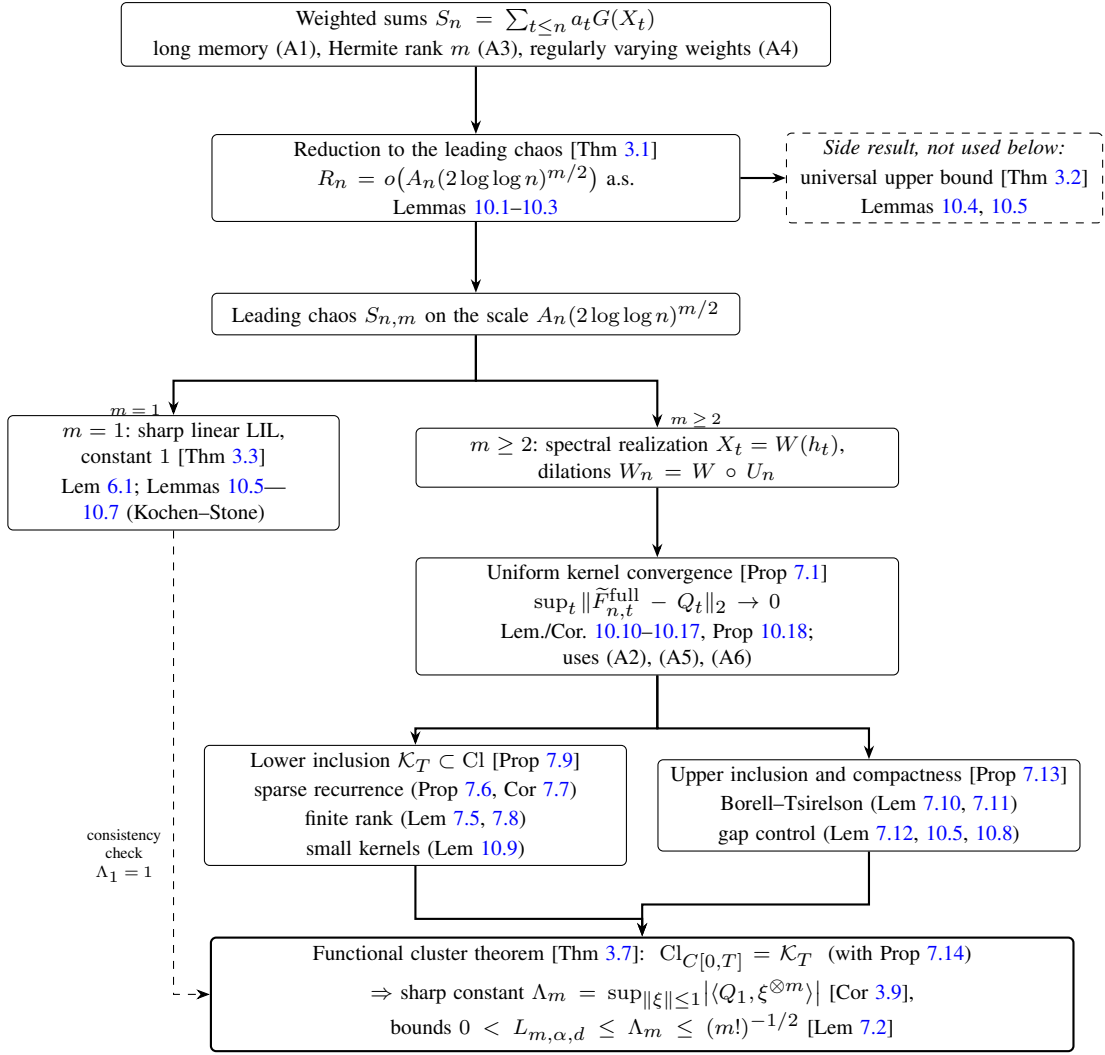

\section{Reduction to the leading chaos}\label{sec:reduction}

In this section we prove Theorem~\ref{thm:reduction}.

\begin{proof}[Proof of Theorem~\ref{thm:reduction}]
Let \(G_m(x):=\frac{J_m}{m!}H_m(x)\), \(\widetilde G(x):=G(x)-G_m(x)\), and \(n_k:=2^k\).
Then \(\widetilde G\in L^2(\mathbb R,\phi)\) has Hermite rank at least \(m+1\), and
\(
    R_n=\sum_{t=1}^n a_t\widetilde G(X_t).
\)

By Lemma~\ref{lem:higher-rank-variance-gain} and
Proposition~\ref{prop:exact-variance}, \(\Var(R_n)=o(A_n^2)\).
 More precisely, there
exist constants \(C<\infty\) , \(\delta>0\) and $\gamma\ge1$ such that
\(
    \Var(R_{n_k})\le C A_{n_k}^2 2^{-2k\delta}.
\)
\newline Hence, for every \(\varepsilon>0\),
\(
    \sum_{k=2}^\infty
    \mathbb P\!\left(
        |R_{n_k}|>\varepsilon A_{n_k}(2\log\log n_k)^{m/2}
    \right)
    <\infty,
\)
and therefore, by Borel--Cantelli,
\[
    \frac{R_{n_k}}{A_{n_k}(2\log\log n_k)^{m/2}}\longrightarrow0
    \qquad\text{\rm a.s.}
\]

It remains to fill the dyadic gaps. For \(n_k\le n\le n_{k+1}\),
\(
    R_n-R_{n_k}=\sum_{t=n_k+1}^n a_t\widetilde G(X_t).
\)
Again by Lemma~\ref{lem:higher-rank-variance-gain}, the interval variance bound required in
Lemma~\ref{lem:dyadic-maximal} holds: for every interval
\(I\subset\{n_k+1,\ldots,n_{k+1}\}\),
\[
    \Var\!\left(\sum_{t\in I} a_t\widetilde G(X_t)\right)
    \le
    C A_{n_k}^2 2^{-2k\delta}\left(\frac{|I|}{n_k}\right)^\gamma .
\]
Applying Lemma~\ref{lem:dyadic-maximal} with \(N=n_k\),
\(Z_t:=a_t\widetilde G(X_t)\), \(n_k+1\le t\le n_{k+1}\),
and \(B_N^2=C A_{n_k}^2 2^{-2k\delta}\), we obtain
\[
    \mathbb E\max_{n_k\le n\le n_{k+1}} |R_n-R_{n_k}|^2
    \le
    C k^2 A_{n_k}^2 2^{-2k\delta}.
\]
Therefore, for every \(\varepsilon>0\),
\[
    \sum_{k=2}^\infty
    \mathbb P\!\left(
        \max_{n_k\le n\le n_{k+1}} |R_n-R_{n_k}|
        >
        \varepsilon A_{n_k}(2\log\log n_k)^{m/2}
    \right)
    <\infty.
\]
Another application of Borel--Cantelli gives
\[
    \max_{n_k\le n\le n_{k+1}}
    \frac{|R_n-R_{n_k}|}{A_{n_k}(2\log\log n_k)^{m/2}}
    \longrightarrow0
    \qquad\text{\rm a.s.}
\]

Since \(A_n\) is regularly varying, \(A_n\asymp A_{n_k}\) uniformly for
\(n_k\le n\le n_{k+1}\), and \(\log\log n\sim\log\log n_k\) uniformly on the same blocks. The
dyadic convergence and the gap estimate therefore imply
\[
    \limsup_{n\to\infty}
    \frac{|R_n|}{A_n(2\log\log n)^{m/2}}
    =0
    \qquad\text{\rm a.s.}
\]
This proves the theorem.
\end{proof}

\section{Upper bound}\label{sec:upper-bound}

In this section we prove Theorem~\ref{thm:upper}.

Before turning to the sharp linear theorem and to the nonlinear functional
cluster theorem, we first record a general upper bound on the natural weighted
LIL scale. This result is not used later in the proof of Theorem~\ref{thm:functional-nonlinear-lil}, but it
shows that the normalization $A_n(2\log\log n)^{m/2}$ is the correct universal
upper almost-sure scale in the whole weighted long-memory regime.

We use two auxiliary estimates: a local
variance bound for leading-chaos increments and a maximal tail estimate in a fixed Wiener
chaos. Their proofs are given in the appendix.

\begin{proof}[Proof of Theorem~\ref{thm:upper}]
By Theorem~\ref{thm:reduction}, it suffices to prove the estimate for the leading chaos term
\(S_{n,m}\).

Let \(n_k:=\lfloor\theta^k\rfloor\), where \(1<\theta<2\). Since \(S_{n_k,m}\) belongs to the fixed \(m\)-th Wiener chaos and
\(\|S_{n_k,m}\|_2\le C A_{n_k}\), the hypercontractive fixed-chaos tail bound
\[
\mathbb P(|F|>x)\le C_m
\exp\left\{-c_m\left(\frac{x}{\|F\|_2}\right)^{2/m}\right\},
\qquad F\in\mathcal H_m,
\]
see \cite{Janson1997},
gives constants \(c_1,c_2>0\) such that
\[
    \mathbb P\!\left(
        |S_{n_k,m}|>\lambda A_{n_k}(2\log\log n_k)^{m/2}
    \right)
    \le
    c_1(\log n_k)^{-2c_2\lambda^{2/m}} .
\]
Since \(\log n_k\asymp k\), fix \(\lambda_0>0\) so large that
\(
    2c_2\lambda_0^{2/m}>1 .
\)
Then the above bound with \(\lambda=\lambda_0\) is summable in \(k\). Hence, by
Borel--Cantelli,
\[
    \limsup_{k\to\infty}
    \frac{|S_{n_k,m}|}
    {A_{n_k}(2\log\log n_k)^{m/2}}
    <\infty
    \qquad\text{\rm a.s.}
\]

We now fill the gaps between \(n_k\) and \(n_{k+1}\). Since \(1<\theta<2\), for all large \(k\)
we have \(n_{k+1}\le2n_k\), and hence all intervals in the gap lie inside
\(\{n_k+1,\ldots,2n_k\}\). 
Since \(\alpha m<1\), choose \(\eta\in(0,1-\alpha m)\) and set
\(q:=2(1-\alpha m/2)-\eta>1\). By Lemma~\ref{lem:leading-chaos-local-variance}, for every
interval \(I\subset\{n_k+1,\ldots,n_{k+1}\}\),
\[
    \operatorname{Var}\left(
        \frac{J_m}{m!}\sum_{t\in I}a_tH_m(X_t)
    \right)
    \le
    C A_{n_k}^2\left(\frac{|I|}{n_k}\right)^q .
\]
Let \(L_k:=n_{k+1}-n_k\).  Since \(L_k\le C(\theta-1)n_k\), we have
\(
    \frac{|I|}{n_k}
    =
    \frac{L_k}{n_k}\frac{|I|}{L_k}
    \le
    C(\theta-1)\frac{|I|}{L_k}.
\)
Hence
\[
    \operatorname{Var}\left(
        \frac{J_m}{m!}\sum_{t\in I}a_tH_m(X_t)
    \right)
    \le
    C A_{n_k}^2(\theta-1)^q
    \left(\frac{|I|}{L_k}\right)^q .
\]

Apply Lemma~\ref{lem:fixed-chaos-maximal-tail} on the block
\(\{n_k+1,\ldots,n_{k+1}\}\), with
\(
    B_k:=C_0 A_{n_k}(\theta-1)^{q/2},
\)
where \(C_0\) is chosen so that the variance assumption of the maximal lemma is satisfied.
Then there exist constants \(C_1,C_2>0\), independent of \(k\), such that, for every
\(\lambda>0\),
\[
    \mathbb P\left(
        \max_{n_k\le n\le n_{k+1}}
        |S_{n,m}-S_{n_k,m}|
        >
        \lambda A_{n_k}(2\log\log n_k)^{m/2}
    \right) 
 \]
 \[                                      
\le
    C_1\exp\left\{
        -C_2\lambda^{2/m}
        (\theta-1)^{-q/m}
        2\log\log n_k
    \right\}.
\]
Equivalently, the right-hand side is bounded by
\(
    C_1(\log n_k)^{-2C_2\lambda^{2/m}(\theta-1)^{-q/m}} .
\)
Fix \(\lambda_1>0\) so large that
\(
    2C_2\lambda_1^{2/m}(\theta-1)^{-q/m}>1 .
\)
Then the last bound with \(\lambda=\lambda_1\) is summable in \(k\). Therefore,
by Borel--Cantelli,
\[
    \limsup_{k\to\infty}
    \max_{n_k\le n\le n_{k+1}}
    \frac{|S_{n,m}-S_{n_k,m}|}
    {A_{n_k}(2\log\log n_k)^{m/2}}
    <\infty
    \qquad\text{\rm a.s.}
\]

Finally, regular variation gives \(A_n\asymp A_{n_k}\) uniformly for
\(n_k\le n\le n_{k+1}\), while \(\log\log n\sim\log\log n_k\) uniformly on the same blocks.
Combining the subsequence estimate with the gap estimate yields
\[
    \limsup_{n\to\infty}
    \frac{|S_{n,m}|}
    {A_n(2\log\log n)^{m/2}}
    \le C_{\rm up}
    \qquad\text{\rm a.s.}
\]
The same upper bound holds for \(S_n\) by Theorem~\ref{thm:reduction}.
\end{proof}

\section{The linear case \(m=1\)}\label{sec:linear} 
Throughout this section \(m=1\). Write
\(
G(x)=J_1x+\widetilde G(x),\qquad \operatorname{rank}(\widetilde G)\ge2,
\)
and hence \(S_n=S_{n,1}+R_n\), where \(S_{n,1}=J_1\sum_{t\le n}a_tX_t\). Put
\(\sigma_n^2:=\operatorname{Var}(S_{n,1})\) and
\(
H_1:=1-d-\frac{\alpha}{2}>0.
\)

\begin{lemma}\label{lem:linear-var-corr}
Assume \ref{A1} and \ref{A4}. Let \(m=1\), \(J_1\ne0\), and suppose that
\(
    2d+\alpha<1 .
\)
As \(n\to\infty\),
\[
\sigma_n^2\sim J_1^2c\,I_{\alpha,d}\,
n^{2-2d-\alpha}\ell_a(n)^2L_0(n),
\qquad
I_{\alpha,d}:=\int_0^1\int_0^1x^{-d}y^{-d}|x-y|^{-\alpha}\,dx\,dy.
\]
Moreover, \(A_n^2\sim\sigma_n^2\). Finally, for every \(\eta\in(0,\alpha/2)\) there is
\(C_\eta<\infty\) such that, whenever \(3\le n\le N/4\),
\[
\left|\operatorname{Corr}(S_{n,1},S_{N,1})\right|
\le C_\eta\left(\frac nN\right)^{\alpha/2-\eta}.
\]
\end{lemma}

\begin{proof}
The variance asymptotic and the relation \(A_n^2\sim\sigma_n^2\) follow from
Proposition~\ref{prop:exact-variance} applied with \(m=1\). It remains to prove the correlation bound. Write
\[
    \operatorname{Cov}(S_{n,1},S_{N,1})=\sigma_n^2+B_{n,N},
    \qquad
    B_{n,N}:=J_1^2\sum_{s\le n<t\le N}a_sa_t\rho(t-s).
\]
Since \(2d+\alpha<1\), we have \(d+\alpha<1\), and hence
\(
    H_1=1-d-\frac{\alpha}{2}>\frac{\alpha}{2}.
\)
Therefore Potter's bound and the variance asymptotic give:
\(
    \frac{\sigma_n}{\sigma_N}
    \le C_\eta\left(\frac nN\right)^{\alpha/2-\eta},
    \qquad 3\le n\le N/4.
\)

Split \(B_{n,N}=B_{n,N}^{\rm near}+B_{n,N}^{\rm far}\), where the near part corresponds to
\(n<t\le2n\), and the far part to \(2n<t\le N\). By Cauchy's inequality,
\[
    |B_{n,N}^{\rm near}|
    \le
    \|S_{n,1}\|_2
    \left\|
        J_1\sum_{n<t\le2n}a_tX_t
    \right\|_2
    \le C\sigma_n^2.
\]
The last bound follows from Lemma~\ref{lem:leading-chaos-local-variance} with
\(m=1\), applied to \(I=\{n+1,\ldots,2n\}\), together with
\(A_n\sim\sigma_n\).
Therefore
\[
    \frac{|B_{n,N}^{\rm near}|}{\sigma_n\sigma_N}
    \le
    C\frac{\sigma_n}{\sigma_N}
    \le
    C_\eta\left(\frac nN\right)^{\alpha/2-\eta}.
\]

For the far part, \(t-s\asymp t\) whenever \(s\le n\) and \(t>2n\). Hence, using regular
variation of \(\rho\), Karamata's theorem and \(d<1\), \(d+\alpha<1\),
\[
    |B_{n,N}^{\rm far}|
    \le
    C
    \sum_{s\le n}|a_s|
    \sum_{2n<t\le N}|a_t|\,|\rho(t-s)|
\]
\[
   \le
    C
    \sum_{s\le n}s^{-d}\ell_a(s)
    \sum_{2n<t\le N}t^{-d-\alpha}\ell_a(t)L_0(t)  
    \le
    C n^{1-d}\ell_a(n)\,
      N^{1-d-\alpha}\ell_a(N)L_0(N).
\]
Dividing by
\[
    \sigma_n\sigma_N
    \asymp
    n^{1-d-\alpha/2}N^{1-d-\alpha/2}
    \ell_a(n)\ell_a(N)L_0(n)^{1/2}L_0(N)^{1/2},
\]
we obtain
\[
    \frac{|B_{n,N}^{\rm far}|}{\sigma_n\sigma_N}
    \le
    C
    \left(\frac nN\right)^{\alpha/2}
    \left(\frac{L_0(N)}{L_0(n)}\right)^{1/2}.
\]
By Potter's bound, uniformly for \(3\le n\le N/4\),
\[
    \left(\frac{L_0(N)}{L_0(n)}\right)^{1/2}
    \le C_\eta\left(\frac Nn\right)^\eta.
\]
Thus
\[
    \frac{|B_{n,N}^{\rm far}|}{\sigma_n\sigma_N}
    \le
    C_\eta\left(\frac nN\right)^{\alpha/2-\eta}.
\]
Combining the estimates for \(\sigma_n^2\), \(B_{n,N}^{\rm near}\), and \(B_{n,N}^{\rm far}\)
proves
\[
    |\operatorname{Corr}(S_{n,1},S_{N,1})|
    \le
    C_\eta\left(\frac nN\right)^{\alpha/2-\eta},
    \qquad 3\le n\le N/4.
\]
\end{proof}

\begin{proof}[Proof of Theorem~\ref{thm:linear-sharp-lil}]
It is enough to prove the result for \(S_{n,1}\), since Theorem~\ref{thm:reduction}
gives \(R_n=o(A_n(2\log\log n)^{1/2})\) a.s. and Lemma~\ref{lem:linear-var-corr} gives
\(A_n\sim\sigma_n\).

For the upper bound, let \(n_k=\lfloor\theta^k\rfloor\), \(1<\theta<2\). 
Since \(S_{n_k,1}/\sigma_{n_k}\) is standard normal, for every
\(\varepsilon>0\), Gaussian tails and Borel--Cantelli give
\[
\limsup_k\frac{|S_{n_k,1}|}{\sigma_{n_k}(2\log\log n_k)^{1/2}}\le1+\varepsilon
\quad\text{a.s.}
\]
The gaps are controlled by Lemma~\ref{lem:fixed-chaos-maximal-tail} with \(m=1\), using
the local variance estimate of Lemma~\ref{lem:leading-chaos-local-variance}. Since
\(\operatorname{Var}\sum_{t\in I}a_tX_t
\le C\sigma_{n_k}^2(|I|/n_k)^q\) for some \(q>1\), we get, for every \(\delta>0\),
\[
\sum_k
\mathbb P\left(
\max_{n_k\le n\le n_{k+1}}
|S_{n,1}-S_{n_k,1}|
>\delta\sigma_{n_k}(2\log\log n_k)^{1/2}
\right)<\infty
\]
by choosing \(\theta>1\) close enough to \(1\). Regular variation of \(\sigma_n\), followed by
\(\theta\downarrow1\), \(\delta\downarrow0\), and \(\varepsilon\downarrow0\), yields
\[
\limsup_n\frac{|S_{n,1}|}{\sigma_n(2\log\log n)^{1/2}}\le1
\quad\text{a.s.}
\]

For the lower bound we apply
Lemma~\ref{lem:finite-dimensional-sparse-recurrence} with $N=1$ and
$Z_n:=S_{n,1}/\sigma_n$, so that $\operatorname{Var}(Z_n)=1$. By
Lemma~\ref{lem:linear-var-corr}, for every $\eta\in(0,\alpha/2)$ there
is $C_\eta<\infty$ with
$|\operatorname{Cov}(Z_n,Z_m)|\le C_\eta(n/m)^{\alpha/2-\eta}$ for
$3\le n\le m/4$; for the remaining pairs $m/4<n\le m$ the trivial bound
$|\operatorname{Cov}(Z_n,Z_m)|\le1
\le4^{\alpha/2-\eta}(n/m)^{\alpha/2-\eta}$ applies after enlarging
$C_\eta$. Thus the covariance hypothesis of
Lemma~\ref{lem:finite-dimensional-sparse-recurrence} holds with
$\theta=\alpha/2-\eta$. Fix $\varepsilon\in(0,1)$ and apply the lemma
with $a:=1-\varepsilon$: there is a sparse sequence
$n_k=\lfloor\exp(k^\gamma)\rfloor$ such that, almost surely,
$|Z_{n_k}-(1-\varepsilon)(2\log\log n_k)^{1/2}|<\varepsilon$ for
infinitely many $k$. In particular,
$Z_{n_k}\ge(1-\varepsilon)(2\log\log n_k)^{1/2}-\varepsilon$ infinitely
often, whence
\[
    \limsup_{n\to\infty}
    \frac{S_{n,1}}{\sigma_n(2\log\log n)^{1/2}}\ge1-\varepsilon
    \qquad\text{a.s.}
\]
Letting \(\varepsilon\downarrow0\) gives the lower bound \(1\);
applying the same argument with \(a=-(1-\varepsilon)\) gives
\[
\liminf_{n\to\infty}
\frac{S_{n,1}}{\sigma_n(2\log\log n)^{1/2}}
\le -1
\quad\text{a.s.}
\]
 Combining with the upper bound and using
\(S_n=S_{n,1}+R_n\), \(A_n\sim\sigma_n\), proves the theorem.
\end{proof}
\begin{remark}
This value $\Lambda_1=1$ matches the prediction of the general variational
formula of Corollary~\ref{cor:sharp-nonlinear-constant}; see
Remark~\ref{rem:Lambda}.
\end{remark}
\section{The nonlinear case \(m\ge2\)}
\label{sec:spectral-realization}

Throughout this section \(m\ge2\). We first fix the spectral realization and the
normalization convention for multiple Wiener integrals.

The Gaussian comparison estimate and the finite-dimensional sparse
recurrence used below
(Lemmas~\ref{lem:gaussian-shifted-ball-comparison}
and~\ref{lem:finite-dimensional-sparse-recurrence}) are stated and
proved in the Appendix; they are shared with the proof of the linear
Theorem~\ref{thm:linear-sharp-lil}.

We use the real Hilbert space \(\mathfrak H_{\mathbb R}\) introduced in
Section~3. The corresponding tensor products and symmetric tensor products are always
taken over this real Hilbert space. We write
\(I_m(\cdot;W)\) for the \(m\)-th multiple Wiener integral with respect to an
isonormal Gaussian process \(W\) over \(\mathfrak H_{\mathbb R}\), normalized by
\[
    \mathbb E\bigl[I_m(F;W)I_m(G;W)\bigr]
    =
    m!\,
    \langle F,G\rangle_{\mathfrak H_{\mathbb R}^{\otimes m}} .
\]

Since \(\rho\) is real and even, we choose the even version of the spectral density \(f\);
equivalently, one may replace \(f\) by \((f(\lambda)+f(-\lambda))/2\). This preserves the
assumptions on \(f\). We extend \(f\) by zero outside \([-\pi,\pi]\), and put \(q:=f^{1/2}\).
Then \(q\) is real and even.

For \(t\ge1\), define \(h_t(\lambda):=e^{it\lambda}q(\lambda)\). Since
\(q\) is real and even, \(h_t(-\lambda)=\overline{h_t(\lambda)}\), hence
\(h_t\in\mathfrak H_{\mathbb R}\). We realize the stationary Gaussian sequence by
\(X_t=W(h_t)\). Moreover,
\[
    \langle h_s,h_t\rangle_{\mathfrak H_{\mathbb R}}
    =
    \operatorname{Re}\int_{\mathbb R}e^{i(s-t)\lambda}f(\lambda)\,d\lambda
    =
    \rho(s-t).
\]
Consequently, \(H_m(X_t)=I_m(h_t^{\otimes m};W)\).

Then
\(
    S_{n,m}
    =
    I_m(F_n;W),
\)
where the symmetric kernel \(F_n\in\mathfrak H_{\mathbb R}^{\odot m}\) is given by
\[
    F_n(\lambda_1,\ldots,\lambda_m)
=
\frac{J_m}{m!}
\sum_{u=1}^n
a_u e^{iu(\lambda_1+\cdots+\lambda_m)}
\prod_{j=1}^m q(\lambda_j).
\]
The Hermitian symmetry of \(F_n\) ensures that \(I_m(F_n;W)\) is real. This
realization is fixed once and for all; hence all random variables
\((S_{n,m})_{n\ge1}\) are constructed on the same isonormal Gaussian space.

For \(r>0\), define
\(
    (U_rh)(x):=r^{1/2}h(rx),
    \qquad h\in L^2(\mathbb R;\mathbb C).
\)
Then \(U_rU_s=U_{rs}\), and each \(U_r\) is an isometry of
\(L^2(\mathbb R;\mathbb C)\) preserving Hermitian symmetry, hence acting
isometrically on \(\mathfrak H_{\mathbb R}\). For \(n\ge3\), write
\(
    W_n(h):=W(U_nh),
\)
so that \(W_n\) is again an isonormal Gaussian process over
\(\mathfrak H_{\mathbb R}\), and
\(
    I_m(K;W_n)=I_m\bigl(U_n^{\otimes m}K;W\bigr)
\)
for every \(K\in\mathfrak H_{\mathbb R}^{\odot m}\), as one checks on
\(K=h^{\otimes m}\) and extends by linearity and density.

Throughout this section, we use the test class
\newline
\(
    D_{\mathbb R}
    :=
    \left\{
        h\in C_c^\infty(\mathbb R;\mathbb C):
        h(-x)=\overline{h(x)}
        \ \text{for all }x\in\mathbb R
    \right\}.
\)
This is a real linear subspace of \(\mathfrak H_{\mathbb R}\), contained in
\(L^1(\mathbb R)\cap L^\infty(\mathbb R)\), and dense in
\(\mathfrak H_{\mathbb R}\). Indeed, if
\(g_j\in C_c^\infty(\mathbb R;\mathbb C)\) converges to
\(h\in\mathfrak H_{\mathbb R}\) in \(L^2\), then
\(
    \widetilde g_j(x)
    :=
    \frac12\bigl(g_j(x)+\overline{g_j(-x)}\bigr)
\)
belongs to \(D_{\mathbb R}\) and converges to \(h\) in \(L^2\).

For \(h\in D_{\mathbb R}\), put
\(
    (Ah)(x):=\frac12h(x)+xh'(x).
\)
Then \(Ah\in D_{\mathbb R}\), and
\begin{equation}\label{eq:dilation-lipschitz}
    \|U_rh-h\|_{L^2(\mathbb R)}
    \le
    |\log r|\,\|Ah\|_{L^2(\mathbb R)},
    \qquad r>0.
\end{equation}
Indeed,
\(
    \frac{d}{ds}(U_sh)
    =
    s^{-1}U_s(Ah)
    \qquad\text{in }L^2(\mathbb R),
\)
and hence
\(
    U_rh-h=\int_1^r s^{-1}U_s(Ah)\,ds.
\)
Minkowski's inequality and the isometry of \(U_s\) now give
\eqref{eq:dilation-lipschitz}.

For \(t\ge0\), define the linearly interpolated leading-chaos process
\[
    \mathcal S_{n,m}(t)
    :=
    \bigl(1-\{nt\}\bigr)S_{\lfloor nt\rfloor,m}
    +
    \{nt\}S_{\lfloor nt\rfloor+1,m}.
\]
Let \(F_0:=0\), and define the corresponding interpolated spectral kernel
\[
    F_{n,t}^{\mathrm{lin}}
    :=
    \bigl(1-\{nt\}\bigr)F_{\lfloor nt\rfloor}
    +
    \{nt\}F_{\lfloor nt\rfloor+1}.
\]
The full rescaled kernel is
\[
    \widetilde F^{\mathrm{full}}_{n,t}(x)
    :=
    \frac{n^{-m/2}}{A_n}
    F_{n,t}^{\mathrm{lin}}
    \left(
        \frac{x_1}{n},\ldots,\frac{x_m}{n}
    \right)
    \mathbf 1_{[-\pi n,\pi n]^m}(x),
    \qquad x=(x_1,\ldots,x_m).
\]

Since \(q\) vanishes outside \([-\pi,\pi]\), the kernel
\(F^{\mathrm{lin}}_{n,t}\) is supported in \([-\pi,\pi]^m\), whence
\(U_n^{\otimes m}\widetilde F^{\mathrm{full}}_{n,t}
=A_n^{-1}F^{\mathrm{lin}}_{n,t}\) and therefore
\[
    I_m\bigl(\widetilde F^{\mathrm{full}}_{n,t};W_n\bigr)
    =
    I_m\bigl(U_n^{\otimes m}\widetilde F^{\mathrm{full}}_{n,t};W\bigr)
    =
    \frac{I_m\bigl(F^{\mathrm{lin}}_{n,t};W\bigr)}{A_n}
    =
    \frac{\mathcal S_{n,m}(t)}{A_n}.
\]

The deterministic limiting kernel is
\[
    Q_t(x)
    =
    C_{m,\alpha,d}
    \left(
        \int_0^t y^{-d}e^{iy(x_1+\cdots+x_m)}\,dy
    \right)
    \prod_{j=1}^m |x_j|^{(\alpha-1)/2},
    \qquad t\ge0.
\]

\begin{proposition}[Uniform full-kernel convergence]
\label{prop:full-kernel-convergence}
Assume \((\ref{A1})\)--\((\ref{A6})\), \(m\ge2\),
\(\alpha m<1\), and \(2d+\alpha m<1\). For every fixed \(T<\infty\),
\[
    \sup_{0\le t\le T}
    \left\|
        \widetilde F^{\mathrm{full}}_{n,t}
        -
        Q_t
    \right\|_{L^2(\mathbb R^m)}
    \longrightarrow0 .
\]
\end{proposition}

\begin{proof}
Put \(\theta(x):=x_1+\cdots+x_m\). Let \(\widehat F_{n,t}\) denote the
non-interpolated rescaled kernel
\[
    \widehat F_{n,t}(x)
    :=
    \frac{n^{-m/2}}{A_n}
    F_{\lfloor nt\rfloor}
    \left(\frac{x_1}{n},\ldots,\frac{x_m}{n}\right)
    \mathbf 1_{[-\pi n,\pi n]^m}(x).
\]
Then
\[
    \widehat F_{n,t}(x)=\gamma_n R_n^{(0)}(t,\theta(x))\Psi_n(x),
\]
where
\[
    R_n^{(0)}(t,\theta)
    :=
    \frac{1}{n^{1-d}\ell_a(n)}
    \sum_{u=1}^{\lfloor nt\rfloor}a_u e^{iu\theta/n},
\]
\[
    \Psi_n(x)
    :=
    \prod_{j=1}^m
    \left(
        n^{-(1-\alpha)/2}L_0(n)^{-1/2}
        q(x_j/n)\mathbf 1_{\{|x_j|\le\pi n\}}
    \right),
\]
and
\[
    \gamma_n
    :=
    \frac{J_m}{m!}
    \frac{n^{1-d-\alpha m/2}\ell_a(n)L_0(n)^{m/2}}{A_n}.
\]
By Proposition~\ref{prop:exact-variance}, \(\gamma_n\to\gamma_m:=J_m/(m!\sigma_m)\).

Fix \(R>0\). By Lemma~\ref{lem:spectral-factor-rescaled} and
Corollary~\ref{cor:product-compact-box},
\[
    \Psi_n\to c_f^{m/2}\prod_{j=1}^m |x_j|^{(\alpha-1)/2}
    \quad\text{in }L^2([-R,R]^m).
\]
Moreover, \(\gamma_m c_f^{m/2}=C_{m,\alpha,d}\). By
Corollary~\ref{cor:compact-frequency-uniformity-full},
\[
    \sup_{0\le t\le T}\sup_{x\in[-R,R]^m}
    \left|
        R_n^{(0)}(t,\theta(x))
        -
        \int_0^t y^{-d}e^{iy\theta(x)}\,dy
    \right|\to0.
\]
Therefore
\[
    \sup_{0\le t\le T}
    \|\widehat F_{n,t}-Q_t\|_{L^2([-R,R]^m)}\to0.
\]

We now pass from compact boxes to \(\mathbb R^m\). By Lemma~\ref{lem:limit-kernel-L2},
\[
    \lim_{R\to\infty}\sup_{0\le t\le T}
    \int_{([-R,R]^m)^c}|Q_t(x)|^2\,dx=0.
\]
The estimate of Proposition~\ref{prop:full-prelimit-tail}, applied at the grid times
\(t=\ell/n\), gives
\[
    \lim_{R\to\infty}\limsup_{n\to\infty}\sup_{0\le t\le T}
    \int_{([-R,R]^m)^c}|\widehat F_{n,t}(x)|^2\,dx=0.
\]
Combining the compact-box convergence with these two tail estimates yields
\[
    \sup_{0\le t\le T}
    \|\widehat F_{n,t}-Q_t\|_{L^2(\mathbb R^m)}\to0.
\]

It remains to control the interpolation error. Since
\(F_{\lfloor nt\rfloor+1}-F_{\lfloor nt\rfloor}\) is one summand,
\[
    \widetilde F^{\mathrm{full}}_{n,t}(x)-\widehat F_{n,t}(x)
    =
    \{nt\}\frac{J_m}{m!}\frac{n^{-m/2}}{A_n}
    a_{\lfloor nt\rfloor+1}
    e^{i(\lfloor nt\rfloor+1)\theta(x)/n}        
   \times
    \prod_{j=1}^m q(x_j/n)\mathbf 1_{[-\pi n,\pi n]^m}(x).
\]
Since \(\int_{-\pi}^{\pi}f(\lambda)\,d\lambda=\rho(0)=1\), this gives
\[
    \sup_{0\le t\le T}
    \|\widetilde F^{\mathrm{full}}_{n,t}-\widehat F_{n,t}\|_2
    \le
    C_m\frac{\sup_{1\le u\le Tn+1}|a_u|}{A_n}.
\]
We claim that the last ratio tends to zero. By Potter's bound, for every \(\eta>0\),
\[
    \sup_{1\le u\le Tn+1}|a_u|
    \le
    C_{\eta,T} n^{\max\{-d,0\}+\eta}\ell_a(n).
\]
Also \(A_n\sim\sigma_m n^{1-d-\alpha m/2}\ell_a(n)L_0(n)^{m/2}\). Applying Potter's bound
to \(L_0^{-m/2}\), we get
\[
    \frac{\sup_{1\le u\le Tn+1}|a_u|}{A_n}
    \le
    C_{\eta,T}
    n^{\max\{-d,0\}-1+d+\alpha m/2+2\eta}.
\]
The exponent is negative for \(\eta>0\) small enough: if \(d\ge0\), this follows from
\(2d+\alpha m<1\); if \(d<0\), it follows from \(\alpha m<1\). Hence
\(
    \sup_{0\le t\le T}
    \|\widetilde F^{\mathrm{full}}_{n,t}-\widehat F_{n,t}\|_2\to0.
\)
Together with \(\sup_{0\le t\le T}\|\widehat F_{n,t}-Q_t\|_2\to0\), this proves the proposition.

\end{proof}

\begin{lemma}[Two-sided bounds for the variational constant]
\label{lem:Lambda-bounds}
Assume $\alpha m<1$ and $2d+\alpha m<1$. Then
\[
    0
    \;<\;
    L_{m,\alpha,d}
    \;\le\;
    \Lambda_m
    \;\le\;
    (m!)^{-1/2},
\]
where the upper bound is the Hilbert--Schmidt norm of the limiting
kernel $Q_1$, and the lower bound is given in closed form by
\[
     L_{m,\alpha,d}
    :=
    \Bigl(\frac{2-\alpha}{2(1-\alpha)}\Bigr)^{m/2}
    \frac{\displaystyle\sum_{k=0}^{m}\binom{m}{k}
      B\bigl(1-d+(1-\alpha)k,\;1+(1-\alpha)(m-k)\bigr)}
    {\bigl(m!\,I_{m,\alpha,d}\bigr)^{1/2}},
\]
where $B(\cdot,\cdot)$ is the Beta function and
$I_{m,\alpha,d}=2B(1-d,1-\alpha m)/(2-2d-\alpha m)$.
\end{lemma}

\begin{proof}
\emph{Upper bound.} By the Cauchy--Schwarz inequality, for every
$\xi\in\mathfrak H_{\mathbb R}$ with
$\|\xi\|_{\mathfrak H_{\mathbb R}}\le1$,
\[
    \bigl|\langle Q_1,\xi^{\otimes m}\rangle\bigr|
    \le\|Q_1\|_{L^2(\mathbb R^m)}\,\|\xi\|_{L^2(\mathbb R)}^{m}
    \le\|Q_1\|_{L^2(\mathbb R^m)} .
\]
By Lemma~\ref{lem:limit-kernel-L2} 
with $t=1$ and the normalization
$\widetilde C_{m,\alpha,d}=1/(m!\,I_{m,\alpha,d})$,
\[
    \|Q_1\|_{L^2(\mathbb R^m)}^2
    =\widetilde C_{m,\alpha,d}\,I_{m,\alpha,d}
    =\frac{1}{m!},
\]
whence $\Lambda_m\le(m!)^{-1/2}$.

\emph{Lower bound.} Consider the trial function
\[
    \xi_1(x):=|x|^{(\alpha-1)/2}\int_0^1 e^{iyx}\,dy
    =|x|^{(\alpha-1)/2}\,\frac{e^{ix}-1}{ix},
    \qquad x\in\mathbb R .
\]
Then $\xi_1(-x)=\overline{\xi_1(x)}$, and $\xi_1\in L^2(\mathbb R)$:
near the origin $|\xi_1(x)|^2\asymp|x|^{\alpha-1}$, integrable since
$\alpha>0$, while at infinity $|\xi_1(x)|^2\le4|x|^{\alpha-3}$,
integrable since $\alpha<1$; hence $\xi_1\in\mathfrak H_{\mathbb R}$.
We use the identity
\begin{equation}
\label{eq:riesz-pair}
    \int_{\mathbb R}|x|^{\alpha-1}e^{iux}\,dx=c_\alpha|u|^{-\alpha},
    \qquad c_\alpha=2\Gamma(\alpha)\cos(\pi\alpha/2),
\end{equation}
understood, as in the proof of Lemma~\ref{lem:limit-kernel-L2}, as
the limit of the regularized integrals with the factor $e^{-\eta|x|}$,
$\eta\downarrow0$; all interchanges of integrals below are justified by
Fubini's theorem for the regularized expressions and dominated
convergence, exactly as there, using $\alpha m<1$ and $d<1$ for the
integrability of the limiting expressions.

By \eqref{eq:riesz-pair},
\[
    \|\xi_1\|_{\mathfrak H_{\mathbb R}}^2
    =\int_0^1\!\!\int_0^1
      \Bigl(\int_{\mathbb R}|x|^{\alpha-1}e^{i(y-z)x}\,dx\Bigr)dy\,dz
    =c_\alpha\int_0^1\!\!\int_0^1|y-z|^{-\alpha}\,dy\,dz
    =\frac{2c_\alpha}{(1-\alpha)(2-\alpha)} .
\]

Next, put
\[
    \varphi(y):=c_\alpha\int_0^1|y-z|^{-\alpha}\,dz
    =\frac{c_\alpha}{1-\alpha}\bigl(y^{1-\alpha}+(1-y)^{1-\alpha}\bigr)>0,
    \qquad 0<y<1,
\]
the second equality being an elementary computation. Applying
\eqref{eq:riesz-pair} in each of the $m$ coordinates $x_1,\ldots,x_m$,
\[
    \bigl\langle Q_1,\xi_1^{\otimes m}\bigr\rangle
    =C_{m,\alpha,d}\int_0^1 y^{-d}\,\varphi(y)^m\,dy .
\]

Expanding the $m$-th power by the binomial formula and
integrating term by term,
\[
    \int_0^1 y^{-d}\bigl(y^{1-\alpha}+(1-y)^{1-\alpha}\bigr)^m dy
    =\sum_{k=0}^{m}\binom{m}{k}
     B\bigl(1-d+(1-\alpha)k,\;1+(1-\alpha)(m-k)\bigr),
\]
all Beta arguments being positive because $d<1$. Therefore
$$
\bigl\langle Q_{1},\xi_{1}^{\otimes m}\bigr\rangle
=C_{m,\alpha,d}\Bigl(\frac{c_{\alpha}}{1-\alpha}\Bigr)^{m}
\sum_{k=0}^{m}\binom{m}{k}
\mathrm{B}\bigl(1-d+(1-\alpha)k,\;1+(1-\alpha)(m-k)\bigr).
$$
 
     Since $\xi_1/\|\xi_1\|_{\mathfrak H_{\mathbb R}}$ is admissible in the
variational problem, and since $\varphi>0$ on $(0,1)$, so that
$\int_0^1y^{-d}\varphi(y)^m\,dy>0$,
\begin{align*}
    \Lambda_m
    &\ge
    \frac{\bigl|\bigl\langle Q_1,\xi_1^{\otimes m}\bigr\rangle\bigr|}
         {\|\xi_1\|_{\mathfrak H_{\mathbb R}}^{m}}\\
    &=
    \bigl|C_{m,\alpha,d}\bigr|
    \Bigl(\frac{c_\alpha}{1-\alpha}\Bigr)^{m}
    \Bigl(\frac{(1-\alpha)(2-\alpha)}{2c_\alpha}\Bigr)^{m/2}
    \sum_{k=0}^{m}\binom{m}{k}
    B\bigl(1-d+(1-\alpha)k,\;1+(1-\alpha)(m-k)\bigr).
\end{align*}
Substituting
$\bigl|C_{m,\alpha,d}\bigr|
=\bigl(m!\,c_\alpha^{m}I_{m,\alpha,d}\bigr)^{-1/2}$, all
powers of $c_\alpha$ cancel and the right-hand side equals
$L_{m,\alpha,d}$.

 Since every factor is positive, $L_{m,\alpha,d}>0$,
which completes the two-sided bound.
\end{proof}


\begin{remark}[Explicit lower bound and degeneration at the boundary]
\label{rem:Lambda-explicit-lower}
Two features of Lemma~\ref{lem:Lambda-bounds} deserve
mention. First, the bound is numerically sharp: for
$(\alpha,d)=(0.20,0.10)$ it gives $L_{2,\alpha,d}=0.6799$ and
$L_{3,\alpha,d}=0.3535$ against the values
$\Lambda_2\approx0.6809$, $\Lambda_3\approx0.3543$ computed in
Remark~\ref{rem:Lambda-numerics}, and comparable accuracy is observed over the parameter range examined there.
Second, for fixed $d\le0$, as $\alpha m\uparrow1$ through the
admissible region,
\(
    B(1-d,1-\alpha m)
    \sim
    \Gamma(1-\alpha m)
    \longrightarrow\infty,
\)
and hence
\(
    L_{m,\alpha,d}\asymp(1-\alpha m)^{1/2}.
\)
Thus it is the explicit lower bound, rather than $\Lambda_m$ itself,
that is shown here to degenerate at this rate. In particular, the bound
is not uniform as the critical boundary is approached.
\end{remark}

\begin{remark}[Geometry versus scale]
\label{rem:geometry-vs-scale}
The two bounds of Lemma~\ref{lem:Lambda-bounds} play different roles.
The upper bound $(m!)^{-1/2}$ is a statement about scale alone: it
follows from the construction of $A_n$, which fixes
$\|Q_1\|_{L^2(\mathbb R^m)}^2=1/m!$ for every admissible
$(\alpha,d,\ell_a)$, together with the elementary domination of the
spectral norm of a symmetric tensor by its Hilbert--Schmidt norm;
consequently it depends on the rank $m$ alone and carries no trace of
the memory or the weights. The lower bound $L_{m,\alpha,d}$ carries
exactly what the upper bound cannot see: the shape of the limiting
kernel
$Q_1(x)=C_{m,\alpha,d}\,\Psi(x_1+\cdots+x_m)\prod_j|x_j|^{(\alpha-1)/2}$,
and in particular the function $\Psi(u)=\int_0^1y^{-d}e^{iyu}\,dy$,
through which the memory exponent $\alpha$ and the weight exponent $d$
enter explicitly; by Remark~\ref{rem:Lambda-explicit-lower}, this
dependence is strong enough to force $L_{m,\alpha,d}\to0$ at the rate
$(1-\alpha m)^{1/2}$ as the critical boundary $\alpha m\uparrow1$ is
approached. In short: the ceiling $(m!)^{-1/2}$ measures the
normalization; the floor $L_{m,\alpha,d}$ measures the geometry.
\end{remark}

 
\begin{lemma}[Finite-rank lower cluster step]
\label{lem:finite-rank-lower-cluster}
Let \(\mathfrak H\) be a real separable Hilbert space, and let \(W_n\), \(n\ge3\), be
isonormal Gaussian processes over \(\mathfrak H\). Let \(E\subset\mathfrak H\) be a
finite-dimensional subspace with orthonormal basis \(e_1,\ldots,e_N\).

Assume that for some \(\xi\in E\) with \(\|\xi\|_{\mathfrak H}\le1\), there exists a subsequence
\(n_k\to\infty\) such that
\[
    \frac{W_{n_k}(e_i)}{(2\log\log n_k)^{1/2}}
    \longrightarrow
    \langle e_i,\xi\rangle_{\mathfrak H},
    \qquad i=1,\ldots,N .
\]
Fix \(T<\infty\). Let \(Q_t\in E^{\odot m}\), \(0\le t\le T\), be a continuous finite-rank kernel family. Then
\[
    \frac{I_m(Q_t;W_{n_k})}{(2\log\log n_k)^{m/2}}
    \longrightarrow
    \langle Q_t,\xi^{\otimes m}\rangle_{\mathfrak H^{\otimes m}}
\]
uniformly in \(t\in[0,T]\).
\end{lemma}

\begin{proof}
Put
\(
    L_k:=2\log\log n_k,
    \quad
    Z_{k,i}:=W_{n_k}(e_i),
    \quad
    a_i:=\langle e_i,\xi\rangle_{\mathfrak H}.
\)
\newline
By assumption,
\( \frac{Z_{k,i}}{L_k^{1/2}}\longrightarrow a_i,
     i=1,\ldots,N .
\)

Since \(Q_t\in E^{\odot m}\), the random variable \(I_m(Q_t;W_{n_k})\) is a Wick polynomial in
the finitely many Gaussian coordinates
\(
    Z_{k,1},\ldots,Z_{k,N}.
\)
More precisely, there exists a polynomial \(P_t\) of total degree \(m\) such that
\[
    I_m(Q_t;W_{n_k})=P_t(Z_{k,1},\ldots,Z_{k,N}),
\]
and its homogeneous part of degree \(m\) is
\[
    P_t^{(m)}(z_1,\ldots,z_N)
    =
    \left\langle
        Q_t,
        \left(\sum_{i=1}^N z_i e_i\right)^{\otimes m}
    \right\rangle_{\mathfrak H^{\otimes m}}.
\]
All remaining terms have total degree at most \(m-2\), because Wick polynomials contain only
contractions of order at least one.

Therefore,
\[
    \frac{P_t(Z_{k,1},\ldots,Z_{k,N})}{L_k^{m/2}}
    =
    P_t^{(m)}
    \left(
        \frac{Z_{k,1}}{L_k^{1/2}},\ldots,
        \frac{Z_{k,N}}{L_k^{1/2}}
    \right)
    +o(1),
\]
where the \(o(1)\) is uniform in \(t\in[0,T]\). Indeed, the lower-degree coefficients are
uniformly bounded in \(t\), since \(t\mapsto Q_t\) is continuous and \([0,T]\) is compact.

Passing to the limit gives
\[
    \frac{I_m(Q_t; W_{n_k})}{L_k^{m/2}}
    \longrightarrow
    \left\langle
        Q_t,
        \left(\sum_{i=1}^N a_i e_i\right)^{\otimes m}
    \right\rangle_{\mathfrak H^{\otimes m}}
\]
uniformly in \(t\in[0,T]\). Since
\(
    \xi=\sum_{i=1}^N a_i e_i,
\)
the right-hand side is
\(
    \langle Q_t,\xi^{\otimes m}\rangle_{\mathfrak H^{\otimes m}}.
\)
This proves the lemma.
\end{proof}


\begin{proposition}[Dilation-coordinate recurrence on finite-dimensional subspaces]
\label{cor:dilation-coordinate-recurrence}

Let \(W\), \(U_n\), and \(W_n\) be as in the beginning of this section; thus
\((U_nh)(x)=n^{1/2}h(nx)\) for \(h\in\mathfrak H_{\mathbb R}\), and
\(W_n(h)=W(U_nh)\), \(n\ge3\).
Let \(E\subset D_{\mathbb R}\) be finite-dimensional, and let
\(e_1,\ldots,e_N\) be an orthonormal basis of \(E\) in \(\mathfrak H_{\mathbb R}\). Fix
\(\xi\in E\) with \(\|\xi\|_{\mathfrak H_{\mathbb R}}<1\). 

Then there exists
\(\gamma>1\) and a sparse deterministic sequence \(n_k:=\lfloor\exp(k^\gamma)\rfloor\)
such that, with probability one, there is a subsequence \(k_\ell\to\infty\) satisfying
\[
    \frac{W_{n_{k_\ell}}(h)}
    {(2\log\log n_{k_\ell})^{1/2}}
    \longrightarrow
    \langle h,\xi\rangle_{\mathfrak H_{\mathbb R}}
    \qquad
    \text{for every }h\in E .
\]
\end{proposition}

\begin{proof}
Write \(\xi=\sum_{i=1}^N a_ie_i\), where
\(a_i=\langle e_i,\xi\rangle_{\mathfrak H_{\mathbb R}}\). Then
\(|a|^2=\sum_i a_i^2=\|\xi\|_{\mathfrak H_{\mathbb R}}^2<1\). For \(n\ge3\), define
\(Z_n=(W_n(e_1),\ldots,W_n(e_N))\). Since \(U_n\) is an isometry on
\(\mathfrak H_{\mathbb R}\), \(\operatorname{Cov}(Z_n)=I_N\).

It remains to verify the covariance decay required in
Lemma~\ref{lem:finite-dimensional-sparse-recurrence}. If \(n\le m\), then
$$
\left|
    \operatorname{Cov}\bigl(W_n(e_i),W_m(e_j)\bigr)
\right|
=
\left|
    \langle U_ne_i,U_me_j\rangle_{\mathfrak H_{\mathbb R}}
\right|                                                        
\le
\left(\frac nm\right)^{1/2}
\int_{\mathbb R}
\left|e_i\!\left(\frac nm y\right)\right|\,
|e_j(y)|\,dy
$$
$$               
\le
\left(\frac nm\right)^{1/2}\|e_i\|_{L^\infty}\|e_j\|_{L^1}.
$$
Since \(E\) is finite-dimensional and \(e_i\in L^1\cap L^\infty\), this is bounded by
\(C_E(n/m)^{1/2}\), uniformly in \(i,j\). Lemma~\ref{lem:finite-dimensional-sparse-recurrence}
therefore applies with exponent \(1/2\).

Choose \(\gamma>1\) so close to \(1\) that \(\gamma|a|^2<1\), and set
\(n_k=\lfloor\exp(k^\gamma)\rfloor\). For every \(q\ge1\),
Lemma~\ref{lem:finite-dimensional-sparse-recurrence} gives
\[
    \mathbb P\left(
        |Z_{n_k}-(2\log\log n_k)^{1/2}a|<\frac1q
        \ \text{\rm for infinitely many }k
    \right)=1.
\]
Intersecting these probability-one events over \(q\), choose \(k_q\to\infty\) such that
\[
    \left|
        \frac{Z_{n_{k_q}}}{(2\log\log n_{k_q})^{1/2}}-a
    \right|<\frac1q .
\]
Then \(Z_{n_{k_q}}/(2\log\log n_{k_q})^{1/2}\to a\). By linearity, for
\(h=\sum_i b_ie_i\in E\),
\[
    \frac{W_{n_{k_q}}(h)}
    {(2\log\log n_{k_q})^{1/2}}
    \to
    \sum_i b_ia_i
    =
    \langle h,\xi\rangle_{\mathfrak H_{\mathbb R}} .
\]
This proves the proposition.
\end{proof}

\begin{corollary}[Finite-rank chaos recurrence]
\label{cor:finite-rank-chaos-recurrence}
Let \(T<\infty\). Let
\(E\subset D_{\mathbb R}\) be finite-dimensional, and let
\(Q_t\in E^{\odot m}\), \(0\le t\le T\), be a continuous
finite-rank kernel family. Fix \(\xi\in E\) with
\(\|\xi\|_{\mathfrak H_{\mathbb R}}<1\). Then there exists a sparse deterministic sequence
\(n_k=\lfloor\exp(k^\gamma)\rfloor\) and, with probability one, a subsequence
\(k_\ell\to\infty\) such that
\[
    \frac{I_m(Q_t;W_{n_{k_\ell}})}
    {(2\log\log n_{k_\ell})^{m/2}}
    \longrightarrow
    \langle Q_t,\xi^{\otimes m}\rangle_{\mathfrak H_{\mathbb R}^{\otimes m}}
\]
uniformly for \(t\in[0,T]\).
\end{corollary}

\begin{proof}
By Proposition~\ref{cor:dilation-coordinate-recurrence}, there exists a sparse deterministic
sequence \(n_k=\lfloor\exp(k^\gamma)\rfloor\) such that, on an event of probability one, one
can choose a subsequence \(k_\ell\to\infty\) for which
\[
    \frac{W_{n_{k_\ell}}(h)}
    {(2\log\log n_{k_\ell})^{1/2}}
    \longrightarrow
    \langle h,\xi\rangle_{\mathfrak H_{\mathbb R}}
    \qquad\text{for every }h\in E.
\]
Applying Lemma~\ref{lem:finite-rank-lower-cluster} pathwise along this subsequence gives
\[
    \frac{I_m(Q_t;W_{n_{k_\ell}})}
    {(2\log\log n_{k_\ell})^{m/2}}
    \longrightarrow
    \langle Q_t,\xi^{\otimes m}\rangle_{\mathfrak H_{\mathbb R}^{\otimes m}}
\]
uniformly in \(t\in[0,T]\). This proves the corollary.
\end{proof}

\subsection{Lower inclusion for the functional cluster set}
\begin{lemma}[Uniform recurrence under \(L^2\)-kernel approximation]
\label{lem:uniform-kernel-recurrence}

Let \(K_{n,t},K_t\in\mathfrak H_{\mathbb R}^{\odot m}\), \(0\le t\le T\), satisfy
\(\sup_{0\le t\le T}\|K_{n,t}-K_t\|_2\to0\). Assume that \(t\mapsto K_t\) is continuous in
\(L^2\), and that the families \(K_{n,t}\) and \(K_t\) satisfy the uniform increment bound
\[
    \|K_{n,t}-K_{n,s}\|_2^2+\|K_t-K_s\|_2^2\le C_T |t-s|^\beta .
\]
for some \(\beta>0\), all \(s,t\in[0,T]\), and all large \(n\). If
\(\xi\in D_{\mathbb R}\) and \(\|\xi\|_{\mathfrak H_{\mathbb R}}<1\), 
then, on an event of probability one, there exists an increasing sequence
\(N_j=N_j(\omega)\to\infty\) such that
\[
    \sup_{0\le t\le T}
    \left|
        \frac{I_m(K_{N_j,t};W_{N_j})}
        {(2\log\log N_j)^{m/2}}
        -
        \langle K_t,\xi^{\otimes m}\rangle
    \right|
    \longrightarrow0 .
\]
\end{lemma}

\begin{proof}
Let
\(
    L_n:=2\log\log n.
\)
Fix \(1<\gamma<\|\xi\|_{\mathfrak H_{\mathbb R}}^{-2}\) if \(\xi\neq0\), and fix any
\(\gamma>1\) if \(\xi=0\). Put
\(
    n_k:=\lfloor \exp(k^\gamma)\rfloor .
\)
Then \(L_{n_k}\ge c_0\log k\) for all large \(k\), for some \(c_0>0\). By the proof of
Corollary~\ref{cor:finite-rank-chaos-recurrence}, this same sparse sequence may be used for
every finite-dimensional subspace \(E\subset D_{\mathbb R}\) containing \(\xi\).

Let \(\delta_j\downarrow0\). For each \(j\), let
 \(\eta_j^{\mathrm{sk}}>0\) be the smallness threshold from
Lemma~\ref{lem:uniform-small-kernel-chaos} corresponding to \(\delta_j\). Choose
\(
    \varepsilon_j
    :=
    \min\left\{
        \delta_j,\frac{\eta_j^{\mathrm{sk}}}{4}
    \right\}.
\)

Since \(D_{\mathbb R}\) is dense in \(\mathfrak H_{\mathbb R}\), finite-rank symmetric tensors
generated by \(D_{\mathbb R}\) are dense in
\(\mathfrak H_{\mathbb R}^{\odot m}\). Since \(\{K_t:0\le t\le T\}\) is compact in \(L^2\),
for each \(j\) we may choose a finite-dimensional subspace
\(E_j\subset D_{\mathbb R}\), enlarged if necessary so that \(\xi\in E_j\), such that
\(
    K_t^{(j)}:=P_{E_j}^{\odot m}K_t
\)
satisfies
\(
    \sup_{0\le t\le T}\|K_t-K_t^{(j)}\|_2<\varepsilon_j .
\)
Moreover \(t\mapsto K_t^{(j)}\) is continuous and
\(
    \|K_t^{(j)}-K_s^{(j)}\|_2^2
    \le C_T|t-s|^\beta .
\)
because \(P_{E_j}^{\odot m}\) is contractive.

By the assumed uniform convergence \(K_{n,t}\to K_t\), choose a deterministic integer
\(N_j^0\) such that, for all \(n\ge N_j^0\),
\[
    \sup_{0\le t\le T}\|K_{n,t}-K_t\|_2<\varepsilon_j .
\]
For \(n\ge N_j^0\), set
\(
    R_{n,t}^{(j)}:=K_{n,t}-K_t^{(j)} .
\)
Then
\(
    \sup_{0\le t\le T}\|R_{n,t}^{(j)}\|_2
    \le
    2\varepsilon_j
    \le
    \frac{\eta_j^{\mathrm{sk}}}{2}.
\)
Also,
\[
    \|R_{n,t}^{(j)}-R_{n,s}^{(j)}\|_2^2
    \le
    2C_T|t-s|^\beta .
\]
Apply Corollary~\ref{cor:finite-rank-chaos-recurrence} to the finite-rank family
\(K_t^{(j)}\). On an event \(\Omega_j^{\mathrm{fr}}\) of probability one, there are infinitely
many \(k\) such that
\[
    \sup_{0\le t\le T}
    \left|
        \frac{I_m(K_t^{(j)};W_{n_k})}{L_{n_k}^{m/2}}
        -
        \langle K_t^{(j)},\xi^{\otimes m}\rangle
    \right|
    \le
    \delta_j .
\]
Next, applying Lemma~\ref{lem:uniform-small-kernel-chaos} to
\(R_{n,t}^{(j)}\), after discarding finitely many \(k\) for which \(n_k<N_j^0\), gives an event
\(\Omega_j^{\mathrm{sk}}\) of probability one such that
\[
    \limsup_{k\to\infty}
    \sup_{0\le t\le T}
    \frac{|I_m(R_{n_k,t}^{(j)};W_{n_k})|}{L_{n_k}^{m/2}}
    \le
    \delta_j .
\]

Set
\(
    \Omega_0:=\bigcap_{j=1}^\infty
    \left(\Omega_j^{\mathrm{fr}}\cap\Omega_j^{\mathrm{sk}}\right).
\)
Then \(\mathbb P(\Omega_0)=1\). On \(\Omega_0\), choose inductively \(k_j=k_j(\omega)\) so
large that
\(
    n_{k_j}\ge \max\{N_j^0,n_{k_{j-1}}+1\},
\)
and both estimates
\[
    \sup_{0\le t\le T}
    \left|
        \frac{I_m(K_t^{(j)};W_{n_{k_j}})}{L_{n_{k_j}}^{m/2}}
        -
        \langle K_t^{(j)},\xi^{\otimes m}\rangle
    \right|
    \le
    \delta_j
\]
and
\[
    \sup_{0\le t\le T}
    \frac{|I_m(R_{n_{k_j},t}^{(j)};W_{n_{k_j}})|}{L_{n_{k_j}}^{m/2}}
    \le
    2\delta_j
\]
hold. Put
\(
    N_j:=n_{k_j}.
\)

Then
\[
\sup_{0\le t\le T}
\left|
    \frac{I_m(K_{N_j,t};W_{N_j})}{L_{N_j}^{m/2}}
    -
    \langle K_t,\xi^{\otimes m}\rangle
\right|                                                        
\quad\le
\sup_{0\le t\le T}
\left|
    \frac{I_m(K_t^{(j)};W_{N_j})}{L_{N_j}^{m/2}}
    -
    \langle K_t^{(j)},\xi^{\otimes m}\rangle
\right|                                                        
\quad+
\]
\[+
\sup_{0\le t\le T}
\frac{|I_m(R_{N_j,t}^{(j)};W_{N_j})|}{L_{N_j}^{m/2}}             
\quad+
\sup_{0\le t\le T}
\left|
    \langle K_t^{(j)}-K_t,\xi^{\otimes m}\rangle
\right|                                                        
\quad\le
\delta_j+2\delta_j+\varepsilon_j\|\xi\|_{\mathfrak H_{\mathbb R}}^m
\le
4\delta_j .
\]
Since \(\delta_j\downarrow0\), the desired convergence follows.
\end{proof}

\begin{proposition}[Lower inclusion for the functional cluster set]
\label{prop:functional-lower-inclusion}
Let the assumption of
Theorem~\ref{thm:functional-nonlinear-lil} hold.
Then, almost surely,
\[
    \mathcal K_T
    \subset
    \operatorname{Cl}_{C[0,T]}
    \left(
        \frac{\mathcal S_{n,m}(\cdot)}
        {A_n(2\log\log n)^{m/2}}
    \right).
\]
\end{proposition}

\begin{proof}

We first prove the assertion for \(\xi\in D_{\mathbb R}\) with \(\|\xi\|<1\). Set
\(f_\xi(t):=\langle Q_t,\xi^{\otimes m}\rangle\). By
Proposition~\ref{prop:full-kernel-convergence},
\(\sup_{0\le t\le T}\|\widetilde F^{\mathrm{full}}_{n,t}-Q_t\|_2\to0\). Moreover the
families \(\widetilde F^{\mathrm{full}}_{n,t}\) and \(Q_t\) satisfy the Hölder \(L^2\)-increment
bound required in Lemma~\ref{lem:uniform-kernel-recurrence}; for \(Q_t\) this follows from Lemma~\ref{lem:limit-kernel-L2} (d), and for
\(\widetilde F^{\mathrm{full}}_{n,t}\) from
Lemma~\ref{lem:full-leading-chaos-increment}. 
Applying Lemma~\ref{lem:uniform-kernel-recurrence} with
\(K_{n,t}=\widetilde F^{\mathrm{full}}_{n,t}\) and \(K_t=Q_t\), we obtain, on an event of
probability one, an increasing sequence \(N_j=N_j(\omega)\to\infty\) such that
\[
    \sup_{0\le t\le T}
    \left|
        \frac{I_m(\widetilde F^{\mathrm{full}}_{N_j,t};W_{N_j})}
        {(2\log\log N_j)^{m/2}}
        -
        f_\xi(t)
    \right|\to0 .
\]
Since
\[
    I_m(\widetilde F^{\mathrm{full}}_{n,t};W_n)
    =
    \frac{\mathcal S_{n,m}(t)}{A_n},
\]
the preceding argument shows that, for this fixed \(\xi\), there is an event
\(\Omega_\xi\) with \(\mathbb P(\Omega_\xi)=1\) such that \(f_\xi\) belongs to the
cluster set on \(\Omega_\xi\).

We now make the exceptional event independent of \(\xi\). Since
\(D_{\mathbb R}\) is dense in \(\mathfrak H_{\mathbb R}\) and
\(\mathfrak H_{\mathbb R}\) is separable, 
choose a countable set
\(
    \mathcal D_0\subset
    \{\xi\in D_{\mathbb R}:\|\xi\|_{\mathfrak H_{\mathbb R}}<1\}
\)
which is dense in the closed unit ball of \(\mathfrak H_{\mathbb R}\).
Such a set exists: the open unit ball is dense in the closed unit ball
(consider $r\xi$, $r\uparrow1$), and $D_{\mathbb R}$ is dense in
$\mathfrak H_{\mathbb R}$, so
$\{\xi\in D_{\mathbb R}:\|\xi\|<1\}$ is dense in the closed unit ball;
a countable dense subset exists by separability.
For every
\(\xi\in\mathcal D_0\), let \(\Omega_\xi\) be the probability-one event obtained above, and set
\(
    \Omega_0:=\bigcap_{\xi\in\mathcal D_0}\Omega_\xi .
\)
Then \(\mathbb P(\Omega_0)=1\). On \(\Omega_0\), the function
\(
    f_\xi(t):=\langle Q_t,\xi^{\otimes m}\rangle
\)
belongs to the cluster set for every \(\xi\in\mathcal D_0\).


Let now \(\xi\in\mathfrak H_{\mathbb R}\) with
\(\|\xi\|_{\mathfrak H_{\mathbb R}}\le1\). Choose
\(\xi_j\in\mathcal D_0\) such that
\(
    \xi_j\to\xi
    \quad\text{in }\mathfrak H_{\mathbb R}.
\)
Recall the map
\(
    \Phi(\xi)(t)
    :=
    \langle Q_t,\xi^{\otimes m}\rangle_{\mathfrak H_{\mathbb R}^{\otimes m}},
    0\le t\le T,
\)
from Remark~\ref{rem:KT-compact}. We use here only its norm-to-uniform
continuity on the closed unit ball. Indeed, since
\(\sup_{0\le t\le T}\|Q_t\|_2<\infty\), we have
\[
\|\Phi(\xi_j)-\Phi(\xi)\|_\infty
=
\sup_{0\le t\le T}
\left|
    \langle Q_t,\xi_j^{\otimes m}-\xi^{\otimes m}\rangle
\right|                                                
\le
\sup_{0\le t\le T}\|Q_t\|_2\,
\|\xi_j^{\otimes m}-\xi^{\otimes m}\|_2                 
\le
\]
\[
C_m
\sup_{0\le t\le T}\|Q_t\|_2\,
\|\xi_j-\xi\|_{\mathfrak H_{\mathbb R}}
\longrightarrow0 .
\]

Thus \(f_{\xi_j}=\Phi(\xi_j)\to \Phi(\xi)=f_\xi\) in \(C[0,T]\). On the event \(\Omega_0\), each \(f_{\xi_j}\) belongs to the cluster set. Since
the cluster set of a sequence in the metric space \(C[0,T]\) is closed, it
follows that \(f_\xi\) also belongs to the cluster set on \(\Omega_0\). Hence
\[
    \mathcal K_T
    \subset
    \operatorname{Cl}_{C[0,T]}
    \left(
        \frac{\mathcal S_{n,m}(\cdot)}
        {A_n(2\log\log n)^{m/2}}
    \right)
\]
on the single probability-one event \(\Omega_0\).
\end{proof}

\begin{lemma}[Finite-dimensional dilation upper cluster bound]
\label{lem:finite-dimensional-upper-cluster}
Let \(E\subset D_{\mathbb R}\) be finite-dimensional with orthonormal basis
\(e_1,\ldots,e_N\), and put
\(
    Z_n:=(W_n(e_1),\ldots,W_n(e_N)),
    \ n\ge3.
\)
Then, almost surely,
\[
    \limsup_{n\to\infty}
    \frac{|Z_n|}{(2\log\log n)^{1/2}}
    \le1.
\]
In particular, every cluster point of
\(Z_n/(2\log\log n)^{1/2}\) belongs to the closed unit ball of
\(\mathbb R^N\).
\end{lemma}

\begin{proof}
Put \(L_n:=2\log\log n\) and
\(
    \Lambda_E
    :=
    \left(\sum_{i=1}^N\|Ae_i\|_2^2\right)^{1/2}.
\)
If \(\Lambda_E=0\), then \eqref{eq:dilation-lipschitz} gives
\(U_re_i=e_i\) for every \(r>0\) and every \(i\), so the conclusion is
immediate. Hence assume that \(\Lambda_E>0\).

Fix \(\varepsilon\in(0,1)\). For \(1<\bar\theta<2\), let
\(
    \mathcal T
    :=
    [1,\bar\theta]\times\{a\in\mathbb R^N:|a|\le1\},
\)
\mbox{$
      g_{r,a}
    :=
    \sum_{i=1}^N a_i(U_re_i-e_i).
$}
The canonical metric of the centered Gaussian process
\(\{W(g_{r,a}):(r,a)\in\mathcal T\}\) is
\(
    d\bigl((r,a),(s,b)\bigr)
    :=
    \|g_{r,a}-g_{s,b}\|_2.
\)
By \eqref{eq:dilation-lipschitz}, for \(r,s\in[1,\bar\theta]\),
\[
    \|U_re_i-U_se_i\|_2
    =
    \|U_{r/s}e_i-e_i\|_2
    \le
    |\log r-\log s|\,\|Ae_i\|_2
    \le
    |r-s|\,\|Ae_i\|_2,
\]
and
\(
    \|U_se_i-e_i\|_2
    \le
    (\log\bar\theta)\|Ae_i\|_2.
\)
Consequently,
$$
    d\bigl((r,a),(s,b)\bigr)
    \le
    |a|\left(\sum_{i=1}^N\|U_re_i-U_se_i\|_2^2\right)^{1/2}
    +|a-b|\left(\sum_{i=1}^N\|U_se_i-e_i\|_2^2\right)^{1/2}
$$
 $$\le
    \Lambda_E\bigl(|r-s|+(\log\bar\theta)|a-b|\bigr)
  \le
    \sqrt2\,\Lambda_E
    \bigl|(r,a)-(s,b)\bigr|_{\mathbb R^{N+1}}.
$$
Moreover,
\(
    \sup_{(r,a)\in\mathcal T}\|g_{r,a}\|_2
    \le
    \Lambda_E\log\bar\theta.
\)
Thus
\(
    \sigma_E(\bar\theta)^2
    :=
    \sup_{(r,a)\in\mathcal T}\|g_{r,a}\|_2^2
    \le
    \Lambda_E^2(\log\bar\theta)^2,
\)
and
\(
    \operatorname{diam}(\mathcal T,d)
    \le
    2\Lambda_E\log\bar\theta
    =:
    D_E(\bar\theta).
\)
Since \(\mathcal T\) is contained in a Euclidean ball of radius \(3\) in
\(\mathbb R^{N+1}\), its covering numbers satisfy
\[
    N(\mathcal T,d,\eta)
    \le
    \left(1+\frac{C_N\Lambda_E}{\eta}\right)^{N+1},
    \qquad 0<\eta\le D_E(\bar\theta),
\]
for a constant \(C_N\) depending only on \(N\). Hence the entropy bound
\cite[Theorem~1.3.3]{AdlerTaylor2007} gives the finite deterministic bound
\[
    M_E(\bar\theta)
    :=
    C\sqrt{N+1}
    \int_0^{D_E(\bar\theta)}
    \sqrt{\log\left(1+\frac{C_N\Lambda_E}{\eta}\right)}\,d\eta
    <\infty,
\]
where \(C\) is universal.

Choose \(\bar\theta\in(1,2)\) so close to \(1\) that
\(
    p
    :=
    \frac{\varepsilon^2}
    {4\Lambda_E^2(\log\bar\theta)^2}
    >1,
\)
then choose \(\theta\in(1,\bar\theta)\), and put
\(
    n_k:=\lfloor\theta^k\rfloor.
\)
Since \(Z_{n_k}\sim N(0,I_N)\), the standard Gaussian tail bound gives
\[
    \mathbb P\left(
        |Z_{n_k}|>(1+\varepsilon)L_{n_k}^{1/2}
    \right)
    \le
    C_N(1+L_{n_k})^{N/2}
    (\log n_k)^{-(1+\varepsilon)^2}
\]
for all sufficiently large \(k\). Since \(\log n_k\asymp k\), the
right-hand side is summable, and Borel--Cantelli gives
\(
    \limsup_{k\to\infty}
    \frac{|Z_{n_k}|}{L_{n_k}^{1/2}}
    \le1+\varepsilon
    \qquad\text{a.s.}
\)

It remains to control the gaps. For all large \(k\),
\(n_{k+1}/n_k\le\bar\theta\). Let
\[
    \mathcal T_0
    :=
    \bigl([1,\bar\theta]\cap\mathbb Q\bigr)
    \times
    \bigl(\mathbb Q^N\cap\{a\in\mathbb R^N:|a|\le1\}\bigr),
\]
and, for \((r,a)\in\mathcal T_0\), put
\(
    X^{(k)}_{r,a}
    :=
    W\bigl(U_{n_k}g_{r,a}\bigr),
    \qquad
    S_k
    :=
    \sup_{(r,a)\in\mathcal T_0}X^{(k)}_{r,a}.
\)
Because \(\mathcal T_0\) is countable and is symmetric under
\(a\mapsto-a\),
\(
    S_k
    =
    \sup_{(r,a)\in\mathcal T_0}|X^{(k)}_{r,a}|.
\)
For \(n_k\le n\le n_{k+1}\), the number \(r:=n/n_k\) belongs to
\([1,\bar\theta]\cap\mathbb Q\), and
\(
    Z_{n,i}-Z_{n_k,i}
    =
    W\bigl(U_{n_k}(U_re_i-e_i)\bigr).
\)
Since rational points are dense in the Euclidean unit ball,
\(
    \max_{n_k\le n\le n_{k+1}}|Z_n-Z_{n_k}|
    \le
    S_k
    \qquad\text{a.s.}
\)

The law of \((X^{(k)}_{r,a})_{(r,a)\in\mathcal T_0}\) does not depend on
\(k\), and its canonical metric is \(d\), because \(U_{n_k}\) is an
isometry. Choose increasing finite symmetric subsets
\(\mathcal T_{0,j}\uparrow\mathcal T_0\), each containing \((1,0)\), and put
\(
    S_{k,j}
    :=
    \sup_{(r,a)\in\mathcal T_{0,j}}X^{(k)}_{r,a}.
\)
Applying Dudley's entropy bound to each finite set \(\mathcal T_{0,j}\),
using the preceding uniform covering-number estimate, and then applying
monotone convergence gives
\(
    \mathbb E S_k
    =
    \lim_{j\to\infty}\mathbb E S_{k,j}
    \le
    M_E(\bar\theta)
    <\infty.
\)

Choose \(k_0\) so large that, for every \(k\ge k_0\),
\(
    M_E(\bar\theta)
    \le
    \frac{\varepsilon}{2}L_{n_k}^{1/2}
\)
and \(n_{k+1}/n_k\le\bar\theta\). For each finite
\(\mathcal T_{0,j}\), the Borell–Tsirelson inequality
\cite[Theorem~2.1.1]{AdlerTaylor2007} gives
\[
    \mathbb P\left(
        S_{k,j}>M_E(\bar\theta)+u
    \right)
    \le
    \mathbb P\left(
        S_{k,j}>\mathbb E S_{k,j}+u
    \right)
    \le
    \exp\left\{-\frac{u^2}{2\sigma_E(\bar\theta)^2}\right\},
    \qquad u>0.
\]
Letting \(j\to\infty\), using continuity from below, and taking
\(u=\frac{\varepsilon}{2}L_{n_k}^{1/2}\), we obtain
$$
    \mathbb P\left(
        \max_{n_k\le n\le n_{k+1}}|Z_n-Z_{n_k}|
        >\varepsilon L_{n_k}^{1/2}
    \right)
    \le
    \exp\left\{
        -\frac{\varepsilon^2L_{n_k}}
        {8\Lambda_E^2(\log\bar\theta)^2}
    \right\}
    =
    (\log n_k)^{-p}.
$$
Since \(\log n_k\asymp k\) and \(p>1\), the last bound is summable in
\(k\). Therefore, Borel--Cantelli gives
\(
    \limsup_{k\to\infty}
    \max_{n_k\le n\le n_{k+1}}
    \frac{|Z_n-Z_{n_k}|}{L_{n_k}^{1/2}}
    \le\varepsilon
    \qquad\text{a.s.}
\)
Combining this with the grid estimate and using \(L_n\ge L_{n_k}\) for
\(n_k\le n\le n_{k+1}\), we obtain
\(
    \limsup_{n\to\infty}
    \frac{|Z_n|}{L_n^{1/2}}
    \le1+2\varepsilon
    \qquad\text{a.s.}
\)
Applying the preceding argument to \(\varepsilon=1/q\), \(q\ge2\), and
intersecting the resulting probability-one events proves the claim.
\end{proof}

\begin{lemma}[Finite-rank chaos upper bound]
\label{lem:finite-rank-chaos-upper}
Let \(E\subset D_{\mathbb R}\) be finite-dimensional, and let
\(
    Q_t\in E^{\odot m},
    \qquad
    0\le t\le T,
\)
be continuous in \(t\). Then, almost surely, the sequence
\[
    \left\{
        t\mapsto
        \frac{I_m(Q_t;W_n)}{(2\log\log n)^{m/2}}:
        n\ge3
    \right\}
\]
is relatively compact in \(C[0,T]\), and every cluster point belongs to
\(
    \left\{
        t\mapsto
        \langle Q_t,\xi^{\otimes m}\rangle:
        \|\xi\|_{\mathfrak H_{\mathbb R}}\le1
    \right\}.
\)
\end{lemma}

\begin{proof}
Let \(e_1,\dots,e_N\) be an orthonormal basis of \(E\), and put
\(
    L_n:=2\log\log n,
    \qquad
    Z_n:=\bigl(W_n(e_1),\dots,W_n(e_N)\bigr),
    \qquad
    Y_n:=\frac{Z_n}{L_n^{1/2}}.
\)
By Lemma~\ref{lem:finite-dimensional-upper-cluster}, almost surely,
\(
    \limsup_{n\to\infty}|Y_n|\le1.
\)
We work on this event. In particular, \((Y_n)\) is bounded.

Since \(Q_t\in E^{\odot m}\), there exists a polynomial \(P_t\) of total degree
\(m\) such that
\(
    I_m(Q_t;W_n)=P_t(Z_n).
\)
Its homogeneous degree-\(m\) part is
\(
    P_t^{(m)}(z)
    =
    \left\langle
        Q_t,
        \left(\sum_{i=1}^N z_i e_i\right)^{\otimes m}
    \right\rangle .
\)
All lower-order terms have degree at most \(m-2\). Their coefficients are
uniformly bounded in \(t\), since \(t\mapsto Q_t\) is continuous and
\([0,T]\) is compact. Since \((Y_n)\) is bounded,
\(
    \frac{I_m(Q_t;W_n)}{L_n^{m/2}}
    =
    P_t^{(m)}(Y_n)+o(1),
\)
uniformly in \(t\in[0,T]\).

Define \(\Psi:\mathbb R^N\to C[0,T]\) by
\(
    \Psi(a)(t):=P_t^{(m)}(a).
\)
The map \(\Psi\) is continuous. Hence the boundedness of \((Y_n)\) and the
preceding uniform expansion imply relative compactness of the normalized
finite-rank sequence in \(C[0,T]\).

Let \(g\) be one of its cluster points, and choose a subsequence \(n_\ell\)
along which the normalized finite-rank processes converge to \(g\) in
\(C[0,T]\). After passing to a further subsequence, the bounded sequence
\((Y_{n_\ell})\) converges to some \(a=(a_1,\dots,a_N)\). The
finite-dimensional bound gives \(|a|\le1\), and the uniform expansion yields
\(
    g(t)=P_t^{(m)}(a)
    =
    \left\langle
        Q_t,
        \left(\sum_{i=1}^N a_i e_i\right)^{\otimes m}
    \right\rangle .
\)
Setting \(\xi:=\sum_{i=1}^N a_i e_i\), we have
\(\|\xi\|_{\mathfrak H_{\mathbb R}}\le1\), which proves the claim.
\end{proof}

\begin{lemma}[Functional gap control for the leading chaos]
\label{lem:functional-grid-gap}
Let \(L_n:=2\log\log n\), and set
\[
    X_n(t):=\frac{\mathcal S_{n,m}(t)}{A_nL_n^{m/2}},
    \qquad 0\le t\le T.
\]
For \(n_k:=\lfloor\theta^k\rfloor\), \(1<\theta<2\),
\[
    \lim_{\theta\downarrow1}
    \limsup_{k\to\infty}
    \max_{n_k\le n\le n_{k+1}}
    \|X_n-X_{n_k}\|_\infty
    =0
    \qquad\text{\rm a.s.}
\]
\end{lemma}

\begin{proof}
Fix \(1<\theta<2\). Uniformly for \(n_k\le n\le n_{k+1}\), regular variation gives
\(A_n/A_{n_k}=1+o_\theta(1)\) and \(L_n/L_{n_k}=1+o_\theta(1)\) as \(k\to\infty\), followed
by \(\theta\downarrow1\). Hence it is enough to control
\[
    \max_{n_k\le n\le n_{k+1}}
    \sup_{0\le t\le T}
    \frac{|\mathcal S_{n,m}(t)-\mathcal S_{n_k,m}(t)|}
    {A_{n_k}L_{n_k}^{m/2}} .
\]

For \(n_k\le n\le n_{k+1}\) and \(0\le t\le T\), put
\(
p:=\lfloor nt\rfloor,\quad q:=\lfloor n_kt\rfloor,
\quad \lambda:=\{nt\},\qquad \mu:=\{n_kt\}.
\)
Then
\[
\mathcal S_{n,m}(t)-\mathcal S_{n_k,m}(t)
=
S_{p,m}-S_{q,m}
+\lambda\bigl(S_{p+1,m}-S_{p,m}\bigr)
-\mu\bigl(S_{q+1,m}-S_{q,m}\bigr).
\]
Moreover, \(p\ge q\) and
\(p-q\le M_k:=C_T(\theta-1)n_k+2\), uniformly in \(t\le T\).
Thus all three terms on the right are increments of length at most
\(M_k\), and the preceding display is bounded by
\[
3\max_{0\le r\le Tn_{k+1}+2}
\max_{0\le \ell\le M_k}
\frac{|S_{r+\ell,m}-S_{r,m}|}
     {A_{n_k}L_{n_k}^{m/2}} .
\]

Cover \(\{0,\ldots,\lfloor Tn_{k+1}\rfloor+M_k\}\) by overlapping blocks
\(B_{j,k}:=[jM_k,(j+3)M_k]\cap\mathbb N\). The number of such blocks is
\(O_T((\theta-1)^{-1})\), and every interval \([r,r+\ell]\) with \(\ell\le M_k\) is contained
in one of them. By Lemma~\ref{lem:full-leading-chaos-increment}, for every interval
\(I\subset B_{j,k}\),
\[
    \operatorname{Var}\left(
        \frac{J_m}{m!}\sum_{u\in I}a_uH_m(X_u)
    \right)
    \le
    C_T A_{n_k}^2
    \left(\frac{|I|}{n_k}\right)^\beta
    \le
    C_T A_{n_k}^2(\theta-1)^\beta
    \left(\frac{|I|}{|B_{j,k}|}\right)^\beta ,
\]
where \(\beta>1\) is the exponent from Lemma~\ref{lem:full-leading-chaos-increment}. Applying
Lemma~\ref{lem:fixed-chaos-maximal-tail} on each block, with
\(B=C_TA_{n_k}(\theta-1)^{\beta/2}\), gives
\[
\begin{aligned}
&\mathbb P\left(
    \max_{0\le r\le Tn_{k+1}+2}
    \max_{0\le \ell\le M_k}
    |S_{r+\ell,m}-S_{r,m}|
    >
    \eta A_{n_k}L_{n_k}^{m/2}
\right)                                                        \\
&\qquad\le
C_T(\theta-1)^{-1}
\exp\left\{
    -c_T\eta^{2/m}(\theta-1)^{-\beta/m}L_{n_k}
\right\}.
\end{aligned}
\]
Since \(L_{n_k}\asymp\log k\), the right-hand side is summable in \(k\) whenever
\(\theta-1\) is sufficiently small. Borel--Cantelli yields
\[
    \limsup_{k\to\infty}
    \max_{n_k\le n\le n_{k+1}}
    \|X_n-X_{n_k}\|_\infty
    \le \eta
    \qquad\text{\rm a.s.}
\]
Taking \(\theta\downarrow1\) and then \(\eta\downarrow0\) proves the claim.
\end{proof}

\begin{proposition}[Relative compactness and upper inclusion]
\label{prop:functional-upper-inclusion}
Let the assumptions of Theorem~\ref{thm:functional-nonlinear-lil} hold. Then the sequence
\[
    \left\{
        \frac{\mathcal S_{n,m}(\cdot)}
        {A_n(2\log\log n)^{m/2}}
        : n\ge3
    \right\}
\]
is almost surely relatively compact in \(C[0,T]\), and every cluster point belongs to
\(\mathcal K_T\).
\end{proposition}

\begin{proof}

Let
\[
    L_n:=2\log\log n,
    \qquad
    X_n(t):=\frac{\mathcal S_{n,m}(t)}{A_nL_n^{m/2}},
    \qquad 0\le t\le T .
\]
Choose deterministic sequences
\(
    \delta_j:=2^{-j},
    \theta_j:=1+2^{-j},
     j\ge1 .
\)
For each \(j\), put
\(
    n_k^{(j)}:=\lfloor \theta_j^k\rfloor .
\)
For every fixed \(j\), \(L_{n_k^{(j)}}\ge c_0\log k\) for all sufficiently large \(k\), with a
constant \(c_0>0\). Let
\[
    \eta_j^{\mathrm{sk}}
    :=
    \eta_0(\delta_j,T,m,\beta,2C_T,c_0)
\]
be the smallness threshold from Lemma~\ref{lem:uniform-small-kernel-chaos}. Choose
\(
    \varepsilon_j
    :=
    \min\left\{
        2^{-j},
        \frac{\eta_j^{\mathrm{sk}}}{4}
    \right\}.
\)

Since \(D_{\mathbb R}\) is dense in \(\mathfrak H_{\mathbb R}\) and
\(\{Q_t:0\le t\le T\}\) is compact in \(L^2(\mathbb R^m)\), for each \(j\) choose a
finite-dimensional subspace
\(
    E_j\subset D_{\mathbb R}
\)
such that, with
\(
    Q_t^{(j)}:=P_{E_j}^{\odot m}Q_t,
\)
one has
\mbox{\(
    \sup_{0\le t\le T}\|Q_t-Q_t^{(j)}\|_2<\varepsilon_j .
\)}
Define
\[
    \mathcal K_T^{(j)}
    :=
    \left\{
        t\mapsto \langle Q_t^{(j)},\xi^{\otimes m}\rangle:
        \|\xi\|_{\mathfrak H_{\mathbb R}}\le1
    \right\}.
\]
Then
\(
    \sup_{g\in\mathcal K_T^{(j)}}
    \operatorname{dist}_{C[0,T]}(g,\mathcal K_T)
    \le
    \varepsilon_j .
\)

We next compare \(X_{n_k^{(j)}}\) with the finite-rank process generated by \(Q^{(j)}\). By
Proposition~\ref{prop:full-kernel-convergence},
\(
    \sup_{0\le t\le T}
    \left\|
        \widetilde F^{\mathrm{full}}_{n,t}-Q_t
    \right\|_2
    \to0 .
\)
Hence, for
\(
    R_{n,t}^{(j)}
    :=
    \widetilde F^{\mathrm{full}}_{n,t}-Q_t^{(j)},
\)
we have\\
\(
    \limsup_{n\to\infty}
    \sup_{0\le t\le T}
    \|R_{n,t}^{(j)}\|_2
    \le
    \varepsilon_j
    <
    \eta_j^{\mathrm{sk}} .
\)
Moreover, \(R_{n,t}^{(j)}\) satisfies the Hölder \(L^2\)-increment bound required in
Lemma~\ref{lem:uniform-small-kernel-chaos}; this follows from
Lemma~\ref{lem:full-leading-chaos-increment} for the prelimit kernels and from
Lemma~\ref{lem:limit-kernel-L2}(d) for \(Q_t^{(j)}\), since \(P_{E_j}^{\odot m}\) is contractive.

Therefore, by Lemma~\ref{lem:uniform-small-kernel-chaos}, on an event
\(\Omega_j^{\mathrm{sk}}\) of probability one,
\[
    \limsup_{k\to\infty}
    \sup_{0\le t\le T}
    \frac{
        |I_m(R_{n_k^{(j)},t}^{(j)};W_{n_k^{(j)}})|
    }
    {L_{n_k^{(j)}}^{m/2}}
    \le
    \delta_j .
\]

By Lemma~\ref{lem:finite-rank-chaos-upper}, on an event
\(\Omega_j^{\mathrm{fr}}\) of probability one, the sequence
\[
    \left(
        \frac{I_m(Q_t^{(j)};W_{n_k^{(j)}})}
        {L_{n_k^{(j)}}^{m/2}}
    \right)_{k\ge1}
\]
is relatively compact in \(C[0,T]\), and every cluster point belongs to
\(\mathcal K_T^{(j)}\). Consequently, on
\(\Omega_j^{\mathrm{fr}}\cap\Omega_j^{\mathrm{sk}}\),
\(
    \limsup_{k\to\infty}
    \operatorname{dist}_{C[0,T]}
    \left(
        X_{n_k^{(j)}},
        \mathcal K_T
    \right)
    \le
    \delta_j+\varepsilon_j .
\)
Indeed,
\[
    X_{n_k^{(j)}}(t)
    =
    \frac{
        I_m(\widetilde F^{\mathrm{full}}_{n_k^{(j)},t};W_{n_k^{(j)}})
    }
    {L_{n_k^{(j)}}^{m/2}} .
\]

It remains to pass from the geometric grid to the full sequence on a common event. By
Lemma~\ref{lem:functional-grid-gap}, there is an event \(\Omega^{\mathrm{gap}}\) of probability one
such that
\[
    g_j(\omega)
    :=
    \limsup_{k\to\infty}
    \max_{n_k^{(j)}\le n\le n_{k+1}^{(j)}}
    \|X_n-X_{n_k^{(j)}}\|_\infty
    \longrightarrow0
    \qquad (j\to\infty)
\]
on \(\Omega^{\mathrm{gap}}\). Set
\(
    \Omega
    :=
    \Omega^{\mathrm{gap}}
    \cap
    \bigcap_{j=1}^\infty
    \left(
        \Omega_j^{\mathrm{fr}}\cap\Omega_j^{\mathrm{sk}}
    \right).
\)
Then \(\mathbb P(\Omega)=1\). On \(\Omega\), for every \(j\),
\[
    \limsup_{n\to\infty}
    \operatorname{dist}_{C[0,T]}(X_n,\mathcal K_T)
    \le
    \delta_j+\varepsilon_j+g_j(\omega).
\]
Letting \(j\to\infty\), we obtain
\(
    \operatorname{dist}_{C[0,T]}(X_n,\mathcal K_T)\to0
\)
along the tail in the limsup sense.

Now let \(g\) be any cluster point of \((X_n)\) in \(C[0,T]\). Then there exists a subsequence
\(X_{n_\ell}\to g\) uniformly. Since \(\mathcal K_T\) is compact, hence closed,
\[
    \operatorname{dist}_{C[0,T]}(g,\mathcal K_T)
    \le
    \liminf_{\ell\to\infty}
    \operatorname{dist}_{C[0,T]}(X_{n_\ell},\mathcal K_T)
    =
    0.
\]
Thus \(g\in\mathcal K_T\).

Finally, the same estimate proves relative compactness. Indeed, for every \(\eta>0\), choose
\(j\) so large that
\(
    \delta_j+\varepsilon_j+g_j(\omega)<\eta .
\)
Then all sufficiently large \(X_n\) lie in the \(\eta\)-neighbourhood of the compact set
\(\mathcal K_T\). Hence the tail of \((X_n)\) is totally bounded in \(C[0,T]\), and the sequence
is relatively compact.

This proves both relative compactness and the upper inclusion on the single probability-one
event \(\Omega\).
\end{proof}

\begin{proposition}[Uniform negligibility of higher chaoses]
\label{prop:functional-remainder}
For every \(T<\infty\),
\[
    \sup_{0\le t\le T}
    \frac{
        |R_{\lfloor nt\rfloor}|
    }
    {
        A_n(2\log\log n)^{m/2}
    }
    \longrightarrow0
    \qquad\text{\rm a.s.}
\]
The same holds with linear interpolation:
\[
    \sup_{0\le t\le T}
    \frac{
        |\mathcal R_n(t)|
    }
    {
        A_n(2\log\log n)^{m/2}
    }
    \longrightarrow0
    \qquad\text{\rm a.s.},
\]
where
$\mathcal R_n(t):=(1-\{nt\})R_{\lfloor nt\rfloor}+\{nt\}R_{\lfloor nt\rfloor+1}.$
\end{proposition}

\begin{proof}
Put \(B_n:=A_n(2\log\log n)^{m/2}\). By Theorem~\ref{thm:reduction},
\[
    \frac{|R_n|}{B_n}\to0
    \qquad\text{a.s.}
\]
Moreover, \(B_n\) is regularly varying with positive index \(H=1-d-\alpha m/2\). Hence, for
every fixed \(T<\infty\), Potter's bound gives
\[
    \sup_{3\le r\le Tn+2}\frac{B_r}{B_n}\le C_T
\]
for all large \(n\). Therefore, on the probability-one event where \(|R_r|/B_r\to0\), for
every \(\varepsilon>0\) there is \(r_0\) such that \(|R_r|\le\varepsilon B_r\) for all
\(r\ge r_0\). Thus
\[
    \sup_{0\le t\le T}\frac{|R_{\lfloor nt\rfloor}|}{B_n}
    \le
    \frac{\max_{0\le r<r_0}|R_r|}{B_n}
    +
    \varepsilon
    \sup_{r_0\le r\le Tn}\frac{B_r}{B_n}
    \le
    o(1)+C_T\varepsilon .
\]
Letting \(\varepsilon\downarrow0\) proves the first claim.

For the interpolated process,
\[
    |\mathcal R_n(t)|
    \le
    |R_{\lfloor nt\rfloor}|+|R_{\lfloor nt\rfloor+1}|,
    \qquad 0\le t\le T.
\]
Hence
\[
    \sup_{0\le t\le T}\frac{|\mathcal R_n(t)|}{B_n}
    \le
    2\max_{0\le r\le Tn+1}\frac{|R_r|}{B_n}\to0
    \qquad\text{a.s.}
\]
This proves the proposition.
\end{proof}

\begin{proof}[Proof of Theorem~\ref{thm:functional-nonlinear-lil}]
By Proposition~\ref{prop:functional-upper-inclusion}, the leading-chaos sequence is almost
surely relatively compact in \(C[0,T]\), and every leading-chaos cluster point belongs to
\(\mathcal K_T\). By Proposition~\ref{prop:functional-lower-inclusion}, every element of
\(\mathcal K_T\) is attained as a leading-chaos cluster point. Hence
\[
    \operatorname{Cl}_{C[0,T]}
    \left(
        \frac{\mathcal S_{n,m}(\cdot)}
        {A_n(2\log\log n)^{m/2}}
    \right)
    =
    \mathcal K_T
    \qquad\text{\rm a.s.}
\]
Proposition~\ref{prop:functional-remainder} gives
\[
    \sup_{0\le t\le T}
    \frac{|\mathcal R_n(t)|}{A_n(2\log\log n)^{m/2}}\to0
    \qquad\text{\rm a.s.}
\]
Since \(\mathcal S_n=\mathcal S_{n,m}+\mathcal R_n\), the full sequence is also relatively
compact and has the same cluster set. This proves the theorem.
\end{proof}

\begin{proof}[Proof of Corollary~\ref{cor:sharp-nonlinear-constant}]
Applying Theorem~\ref{thm:functional-nonlinear-lil} with \(T=1\), we obtain
\[
    \operatorname{Cl}_{C[0,1]}
    \left(
        \frac{
            \mathcal S_n(\cdot)
        }
        {
            A_n(2\log\log n)^{m/2}
        }
    \right)
    =
    \mathcal K_1
    \qquad\text{a.s.}
\]
Evaluating the cluster functions at \(t=1\) gives
\[
    \operatorname{Cl}_{\mathbb R}
    \left(
        \frac{
            S_n
        }
        {
            A_n(2\log\log n)^{m/2}
        }
    \right)
    =
    \left\{
        \big\langle
            Q_1,\xi^{\otimes m}
        \big\rangle:
        \|\xi\|_{\mathfrak H_{\mathbb R}}\le1
    \right\}.
\]
Therefore
\[
    \limsup_{n\to\infty}
    \frac{|S_n|}
    {A_n(2\log\log n)^{m/2}}
    =
    \sup_{\|\xi\|\le1}
    \left|
        \big\langle
            Q_1,\xi^{\otimes m}
        \big\rangle
    \right|
    =
    \Lambda_m
    \qquad\text{a.s.}
\]
The bounds \(
    0<L_{m,\alpha,d}
    \le
    \Lambda_m
    \le
    \frac{1}{\sqrt{m!}},
\)
were proved in 
Lemma~\ref{lem:Lambda-bounds}.
\end{proof}

\FloatBarrier
\section{Numerical illustrations}\label{sec:numerical-illustrations}

In all weighted experiments the normalization uses the \emph{exact}
deterministic variance scale. Since $G=H_m$ in the experiments, the
statistic \\ $S_n=\sum_{t\le n}a_tH_m(X_t)$ lies in the $m$-th Wiener
chaos, and the covariance identity\\
$\operatorname{Cov}(H_m(X_s),H_m(X_t))=m!\bigl(\rho(|t-s|)\bigr)^m$ gives
\[
    A_n^2
    =\operatorname{Var}(S_n)
    =m!\Bigl[\sum_{t=1}^{n}a_t^2
      +2\sum_{t=2}^{n}a_t\sum_{s=1}^{t-1}a_s(\rho(t-s))^m\Bigr],
\]
where $\rho$ is the exact covariance of fractional Gaussian noise,\\
$\rho(h)=\tfrac12\{(h+1)^{2H_X}-2h^{2H_X}+|h-1|^{2H_X}\}$. The inner
sums are a discrete convolution, so $A_k$ is computed exactly for all
$k\le n$ in $O(n\log n)$ operations by the fast Fourier transform; in
the unweighted benchmark this reduces to $A_n=n^{H_X}$. 
The input sequence $(X_t)$ is fractional Gaussian noise, generated by
the Davies--Harte circulant-embedding algorithm from this exact
covariance.

For every displayed case we consider the absolute normalized statistic
\[
    \frac{|S_n|}
    {A_n(2\log\log n)^{m/2}},
\]
where \(A_n^2=\operatorname{Var}(S_n)\) is computed from the exact
covariance formula above. The Gaussian input has
\(\alpha=0.20\), hence \(H_X=0.90\), and the simulation horizon is
\(N=2^{20}\). Nine sample paths from an ensemble of \(M=2000\)
realizations are displayed on a logarithmic horizontal axis.

Two deterministic reference scales are superimposed. The black dashed
line is the asymptotic temporal LIL level \(y=\Lambda_m\). The gray
dotted curve is the exact finite-\(n\) pointwise RMS scale:
\[
    \left[
        \mathbb E
        \left(
            \frac{|S_n|}
            {A_n(2\log\log n)^{m/2}}
        \right)^2
    \right]^{1/2}
    =
    \frac{1}{(2\log\log n)^{m/2}}.
\]
Thus the dotted curve describes the ordinary finite-sample magnitude,
whereas the horizontal dashed line describes rare temporal
\(\limsup\) excursions.

\begin{figure}[!htbp]
\centering
\includegraphics[width=0.72\textwidth]{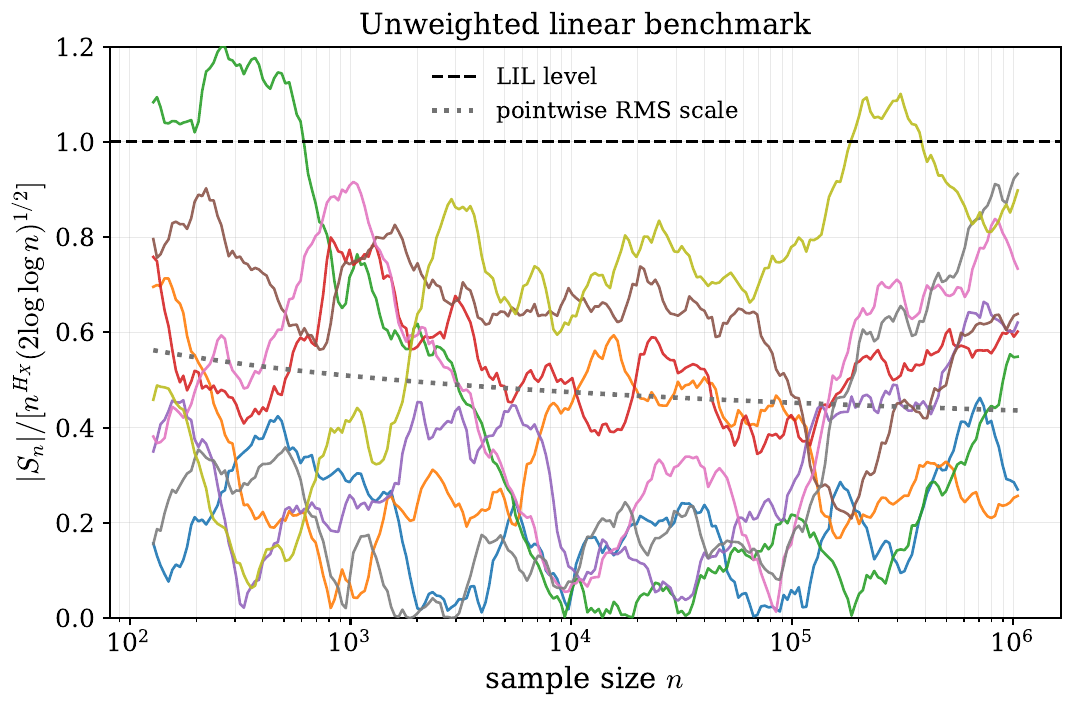}
\caption{
Unweighted linear benchmark with fractional Gaussian noise,
\(\alpha=0.20\), \(H_X=0.90\), \(G=H_1\), and \(a_t\equiv1\).
Here \(A_n=n^{H_X}\), and the displayed trajectories are
\(
    \frac{|S_n|}
    {n^{H_X}(2\log\log n)^{1/2}}.
\)
The black dashed line is the sharp asymptotic LIL level
\(\Lambda_1=1\); the gray dotted curve is the exact finite-\(n\)
pointwise RMS scale $(2\log\log n)^{-1/2}$.
}\label{fig:taqqu-benchmark}
\end{figure}

\begin{figure}[!htbp]
\centering

\begin{subfigure}{0.74\textwidth}
\centering
\includegraphics[
    width=\linewidth
]{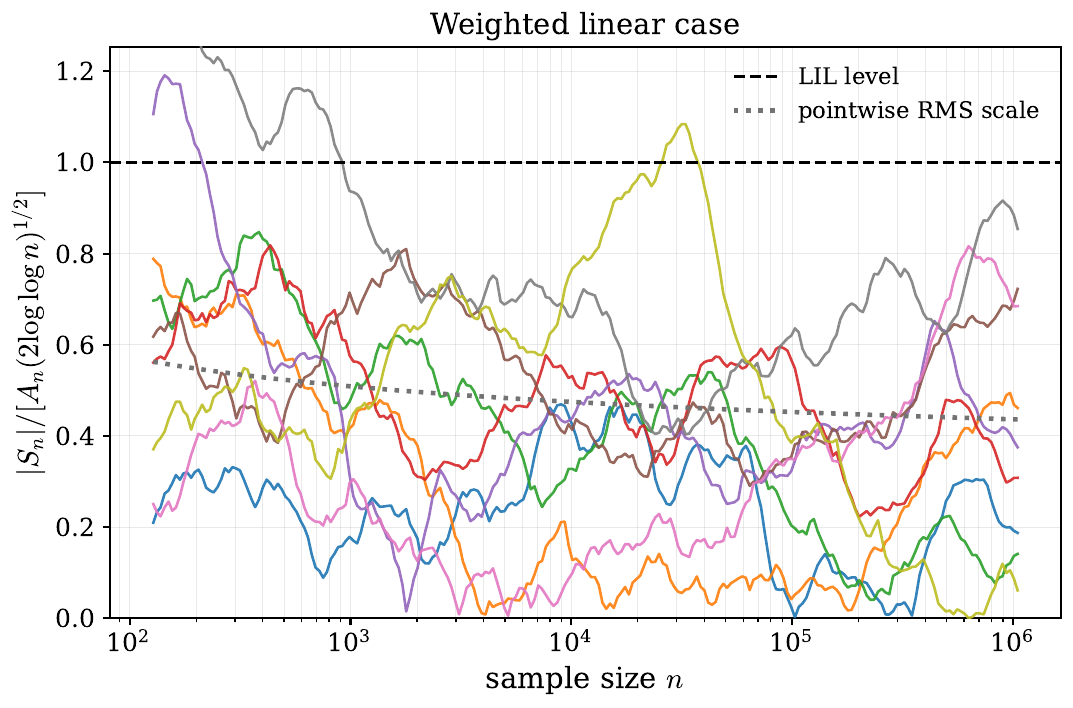}
\caption{\(a_t=t^{-0.10}\).}
\label{fig:weighted-linear-power}
\end{subfigure}

\vspace{0.8em}

\begin{subfigure}{0.74\textwidth}
\centering
\includegraphics[
    width=\linewidth
]{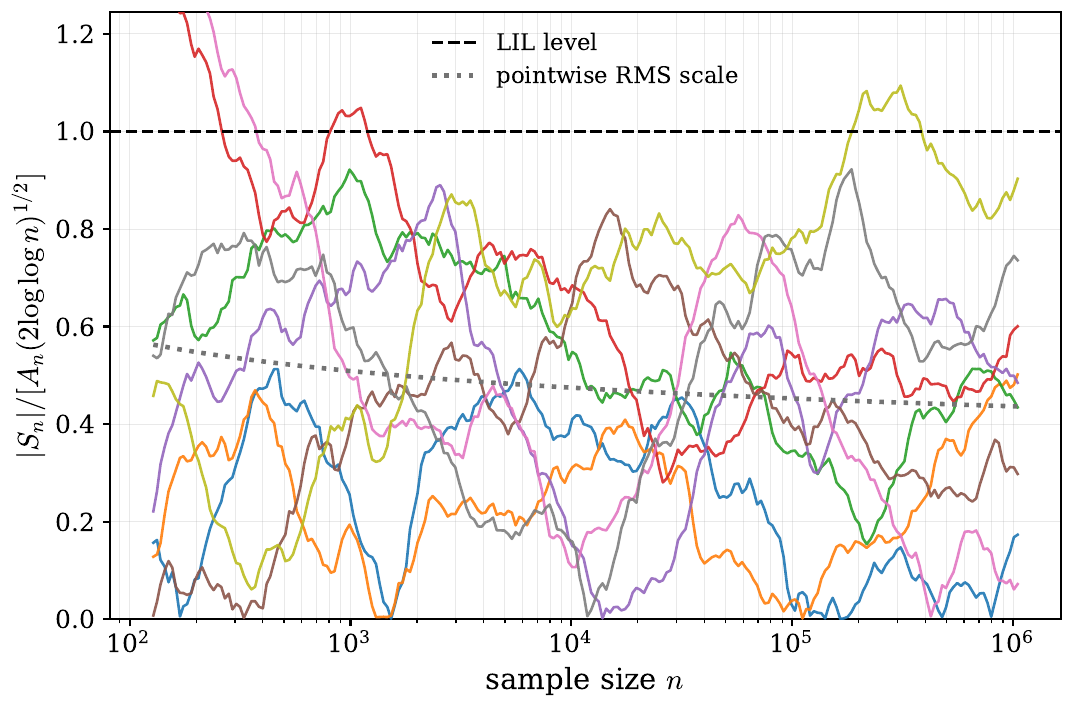}
\caption{\(a_t=t^{-0.10}\{\log(e+t)\}^{1/2}\).}
\label{fig:weighted-linear-log}
\end{subfigure}

\caption{
Weighted linear case with \(\alpha=0.20\), \(G=H_1\), and exact
variance scale \(A_n\). Both panels display
\(
    \frac{|S_n|}
    {A_n(2\log\log n)^{1/2}}.
\)
The black dashed line is the common sharp LIL level
\(\Lambda_1=1\), and the gray dotted curve is the exact finite-\(n\)
pointwise RMS scale. The similar normalized behavior in the two panels
illustrates that the regularly varying weights are absorbed into
\(A_n\), while the linear LIL constant remains unchanged.
}
\label{fig:weighted-linear}
\end{figure}

\begin{figure}[!htbp]
\centering

\begin{subfigure}{0.74\textwidth}
\centering
\includegraphics[
    width=\linewidth
]{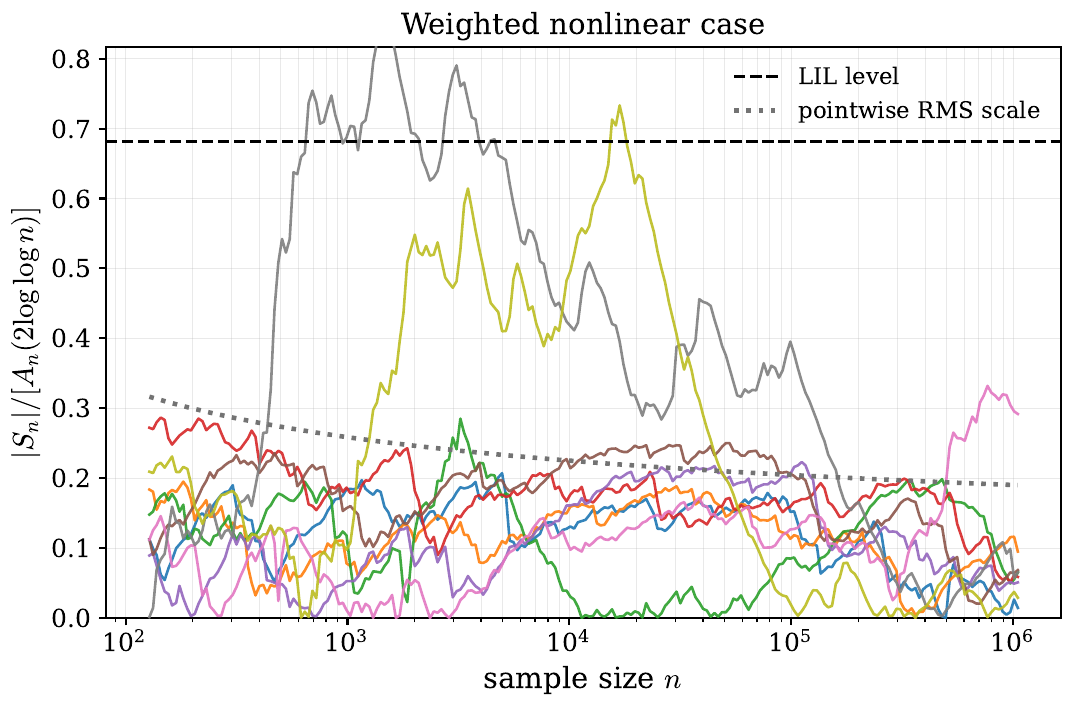}
\caption{\(G=H_2\), Hermite rank \(m=2\).}
\label{fig:weighted-rank2}
\end{subfigure}

\vspace{0.8em}

\begin{subfigure}{0.74\textwidth}
\centering
\includegraphics[
    width=\linewidth
]{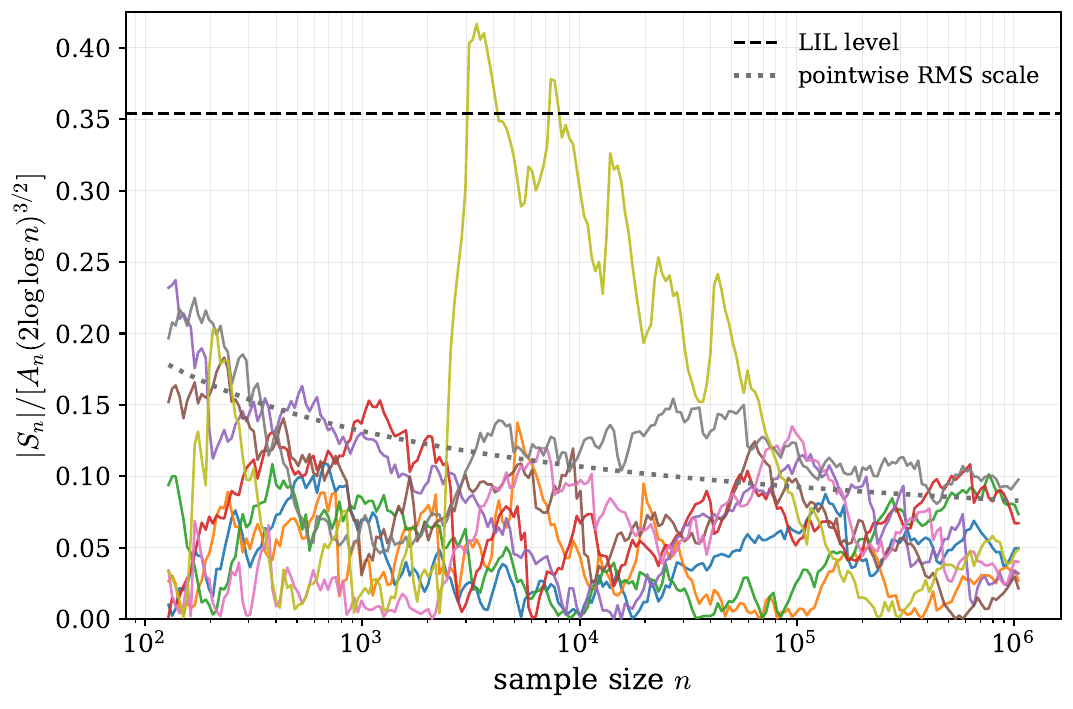}
\caption{\(G=H_3\), Hermite rank \(m=3\).}
\label{fig:weighted-rank3}
\end{subfigure}

\caption{
Weighted nonlinear cases with
\(\alpha=0.20\) and
\(a_t=t^{-0.10}\{\log(e+t)\}^{1/2}\).
Panel~\textup{(a)} displays
\(
    \frac{|S_n|}{A_n(2\log\log n)}
\)
for \(m=2\), and panel~\textup{(b)} displays
\(
    \frac{|S_n|}
    {A_n(2\log\log n)^{3/2}}
\)
for \(m=3\). The black dashed lines are the asymptotic LIL levels
\(\Lambda_2\approx0.6809\) and
\(\Lambda_3\approx0.3543\), respectively. The gray dotted curves are
the exact finite-\(n\) pointwise RMS scales
\((2\log\log n)^{-1}\) and
\((2\log\log n)^{-3/2}\). Their separation from the LIL levels
visualizes the distinction between the ordinary pointwise size of the
trajectories and their rare temporal \(\limsup\) excursions.
}
\label{fig:weighted-nonlinear-comparison}
\end{figure}

\begin{remark}[Finite-sample interpretation of the trajectory plots]
\label{rem:finite-sample-trajectories}
The horizontal LIL level is not a pointwise upper bound and is not the
typical finite-\(n\) location of the trajectory cloud. The sharp LIL
asserts that, for every \(\varepsilon>0\),
\(
    \frac{|S_n|}
    {A_n(2\log\log n)^{m/2}}
    >
    \Lambda_m-\varepsilon
\)
infinitely often almost surely, but it does not prescribe the
frequency of such excursions on a finite simulation horizon. Most
finite-\(n\) values may therefore lie substantially below
\(\Lambda_m\). The gray dotted curve shows their exact
mean-square scale; it is neither an upper envelope nor an additional
LIL boundary. Occasional finite-\(n\) crossings of the black dashed
line are also compatible with the asymptotic \(\limsup\) statement.
\end{remark}

\FloatBarrier

A common ensemble of \(M=2000\) independent fractional-Gaussian-noise
paths was generated and used for all cases. Nine trajectories are
displayed in each figure; they were selected at fixed quantiles of
their late-time normalized excursion size in order to represent a
range of finite-sample behaviours.

\FloatBarrier
\begin{remark}[Numerical evaluation of $\Lambda_m$ and quality of the bounds]
\label{rem:Lambda-numerics}
The variational constant of
Corollary~\ref{cor:sharp-nonlinear-constant} is computable. For real
$h\in L^2[0,1]$, consider the trial functions
\[
    \xi_h(x):=|x|^{(\alpha-1)/2}\int_0^1 h(y)e^{iyx}\,dy
    \in\mathfrak H_{\mathbb R},
\]
for which, with
$(\widetilde Kh)(y):=c_\alpha\int_0^1 h(z)|y-z|^{-\alpha}\,dz$,
\[
    \langle\xi_h,\xi_g\rangle_{\mathfrak H_{\mathbb R}}
    =\langle h,\widetilde Kg\rangle_{L^2[0,1]},
    \qquad
    \bigl\langle Q_1,\xi_h^{\otimes m}\bigr\rangle
    =C_{m,\alpha,d}\int_0^1 y^{-d}\,
     \bigl(\widetilde Kh\bigr)(y)^m\,dy .
\]

The restriction to these functions involves no loss of generality.
Let \(\mathcal V\) be the closed subspace of
\(\mathfrak H_{\mathbb R}\) generated by the functions \(\xi_h\).
If \(\eta\in\mathcal V^\perp\), then
\(\langle\eta,\xi_h\rangle_{\mathfrak H_{\mathbb R}}=0\) for every
\(h\in L^2[0,1]\). Hence the contraction of \(Q_1\) in any one
coordinate against \(\eta\) vanishes, and therefore
\(Q_1\in\mathcal V^{\odot m}\). Projecting \(\xi\) onto
\(\mathcal V\) thus leaves
\(\langle Q_1,\xi^{\otimes m}\rangle\) unchanged and cannot increase
\(\|\xi\|_{\mathfrak H_{\mathbb R}}\). Since the functions \(\xi_h\)
are dense in \(\mathcal V\), the reduced variational problem below has
the same supremum \(\Lambda_m\).


For numerical evaluation, we solve the resulting reduced problem on
\([0,1]\) by a shifted symmetric power iteration on a fine grid. The
Euler--Lagrange stationarity relation motivates the iteration
\(h(y)\propto y^{-d}(\widetilde Kh)(y)^{m-1}\).
The interval \([0,1]\) is divided into \(6000\) equal cells. The cell
integrals of the weakly singular kernel \(|y-z|^{-\alpha}\) and of the
weight \(y^{-d}\) are evaluated in closed form, while the nonlinear
objective is evaluated on the resulting discretization. The iteration
is run with a cap of \(30{,}000\) steps and the stopping rule
\(|\lambda^{(k+1)}-\lambda^{(k)}|<10^{-13}\).
The resulting values are numerical approximations to \(\Lambda_m\),
not certified lower bounds for the continuous variational problem.
Their stability under grid refinement provides a numerical convergence
check: grids of \(1500\), \(3000\), and \(6000\) cells give values of
\(\Lambda_2\) and \(\Lambda_3\) agreeing to six decimal digits.
As a validation of the entire discretization, the same code
run with $m=1$ returns $\Lambda_1=1.000000$, in agreement with
Theorem~\ref{thm:linear-sharp-lil}. For the running example of
this paper, $\alpha=0.20$ and $d=0.10$, it yields the values in
Table~\ref{tab:lambda-main}, together with the closed-form lower bound
$L_{m,\alpha,d}$ of Lemma~\ref{lem:Lambda-bounds} and the upper bound
$(m!)^{-1/2}$.

\begin{table}[!htb]
\centering
\caption{The variational constant $\Lambda_m$ (numerically computed),
the closed-form lower bound $L_{m,\alpha,d}$, and the upper bound
$(m!)^{-1/2}$, for $\alpha=0.20$, $d=0.10$.}
\label{tab:lambda-main}
\begin{tabular}{ccccc}
\hline
$m$ & $L_{m,\alpha,d}$ & $\Lambda_m$ & $(m!)^{-1/2}$
    & accuracy of $L_{m,\alpha,d}$\\
\hline
1 & 0.9993 & 1.0000 & 1.0000 & $99.93\%$\\
2 & 0.6799 & 0.6809 & 0.7071 & $99.85\%$\\
3 & 0.3535 & 0.3543 & 0.4082 & $99.77\%$\\
\hline
\end{tabular}
\end{table}
The strict gap to the Hilbert--Schmidt bound is expected: equality would require $Q_1$ to be a rank-one
symmetric tensor, which fails for $m\ge2$; see Remark~\ref{rem:Lambda}.

The same pattern persists, and sharpens, at higher rank. Table~\ref{tab:lambda-rank} below
reports $\Lambda_m$ for $m=2,\dots,5$ with $\alpha$ chosen so that the
effective memory $\alpha m=0.45$ is held fixed and $d=0.05$, together
with the relative errors
$e_{\mathrm{low}}=(\Lambda_m-L_{m,\alpha,d})/\Lambda_m$ and
$e_{\mathrm{up}}=((m!)^{-1/2}-\Lambda_m)/\Lambda_m$.

\begin{table}[!htb]
\centering
\caption{Rank dependence of the bounds at fixed effective memory
$\alpha m=0.45$, $d=0.05$.}
\label{tab:lambda-rank}
\begin{tabular}{cccccc}
\hline
$m$ & $\alpha$ & $L_{m,\alpha,d}$ & $\Lambda_m$ & $(m!)^{-1/2}$
    & $e_{\mathrm{low}}\ /\ e_{\mathrm{up}}$\\
\hline
2 & 0.2250 & 0.6717 & 0.6720  & 0.7071 & $0.04\%\ /\ 5.2\%$\\
3 & 0.1500 & 0.3828 & 0.3830 & 0.4082 & $0.05\%\ /\ 6.6\%$\\
4 & 0.1125 & 0.1903 & 0.1904 & 0.2041 & $0.06\%\ /\ 7.2\%$\\
5 & 0.0900 & 0.0848 & 0.0849 & 0.0913 & $0.06\%\ /\ 7.6\%$\\
\hline
\end{tabular}
\end{table}

The lower bound is accurate to a fraction of one percent across all
ranks ($e_{\mathrm{low}}<0.25\%$ in Table~\ref{tab:lambda-main} and
$e_{\mathrm{low}}<0.1\%$ in Table~\ref{tab:lambda-rank}), whereas
the upper bound errs by several percent already at $m=2$ and the error
grows with $m$; for larger $\alpha$ the disparity is far more
pronounced, the Hilbert--Schmidt upper bound overestimating $\Lambda_m$
by tens of percent as the boundary $\alpha m\uparrow1$ is approached.
This asymmetry is structural, not numerical: the lower bound is the
value of the objective at an explicit trial function and therefore
reflects the shape of the limiting kernel $Q_1$, while the upper bound
is the Hilbert--Schmidt norm of $Q_1$, which is blind to how far $Q_1$
is from a rank-one tensor and hence cannot improve as that distance
grows.

Consequently, the true constant lies very close to the lower end of
the interval \\
 $[L_{m,\alpha,d},(m!)^{-1/2}]$, not near its midpoint,
and we recommend using the closed-form value $L_{m,\alpha,d}$ itself
as the working approximation to $\Lambda_m$; the upper bound should be
read only as a guaranteed ceiling. The residual underestimate of
$L_{m,\alpha,d}$ is well understood: the trial function $\xi_1$ used
in Lemma~\ref{lem:Lambda-bounds} corresponds to the uniform choice
$h\equiv1$ in the reduced problem on $[0,1]$, whereas the true
maximiser is slightly concentrated toward the endpoints of $[0,1]$; a
single step of the power iteration
$h\mapsto y^{-d}(\widetilde Kh)^{m-1}$ started from $h\equiv1$ removes
almost all of this discrepancy and returns a value indistinguishable
from $\Lambda_m$ at the displayed precision.
\end{remark}

\begin{remark}[Finite-sample interpretation of the normalized trajectories]
\label{rem:normalized-trajectories}
For each displayed realization, define
\(
    Y_n
    :=
    \frac{|S_n|}
    {A_n(2\log\log n)^{m/2}},
    \quad
    T_n:=\frac{S_n}{A_n}.
\)
Theorem~\ref{thm:linear-sharp-lil} for $m=1$ and
Corollary~\ref{cor:sharp-nonlinear-constant} for $m\ge2$ state that
\(
    \limsup_{n\to\infty}Y_n=\Lambda_m
    \qquad\text{a.s.}
\)
Thus the horizontal dashed line at $\Lambda_m$ represents the
asymptotic temporal limsup of a single trajectory. It is neither a
pointwise upper bound at finite $n$ nor the typical location of the
trajectory cloud.

Indeed, since $A_n^2=\operatorname{Var}(S_n)$, one has
$\mathbb E T_n^2=1$, while
\(
    \mathbb E Y_n
    =
    \frac{\mathbb E|T_n|}
    {(2\log\log n)^{m/2}}.
\)
Consequently, at each fixed $n$ the bulk of the normalized trajectories
may lie substantially below $\Lambda_m$, especially for $m=2$ and
$m=3$, where the iterated-logarithm factor is stronger. Occasional
finite-sample crossings of the dashed line do not contradict the LIL:
the theorem concerns the subsequential temporal behavior as
$n\to\infty$, rather than a deterministic bound valid at every $n$.

If $M$ independent trajectories are displayed without selecting them
according to their excursions, a reproducible pointwise reference for
the upper edge of the cloud is obtained as follows. Let
\(
    F_n(x):=\mathbb P(|T_n|\le x),
    \quad
    q_n(p):=F_n^{-1}(p).
\)
Then the median of the pointwise maximum
$\max_{1\le b\le M}Y_n^{(b)}$ is
\[
    u_M(n)
    =
    \frac{q_n(2^{-1/M})}
    {(2\log\log n)^{m/2}},
\]
because
\[
    \mathbb P\!\left(
       \max_{1\le b\le M}|T_n^{(b)}|\le x
    \right)
    =
    F_n(x)^M.
\]
For $m=1$ this curve is explicit:
\[
    u_M(n)
    =
    \frac{
      \Phi^{-1}\!\left((1+2^{-1/M})/2\right)
    }
    {\sqrt{2\log\log n}}.
\]
For $m\ge2$, the finite-$n$ quantile $q_n(2^{-1/M})$ may be estimated
from an independent Monte Carlo ensemble. 
\end{remark}

The LIL level $\Lambda_m$ is an asymptotic temporal limsup. Thus every
trajectory approaches $\Lambda_m$ from below infinitely often, but such
excursions may be extremely sparse and need not be visible on a finite
simulation horizon or on the displayed logarithmic grid.
 \FloatBarrier

\section{Conclusion}\label{sec:conclusion}
The paper establishes the theory of laws of the iterated logarithm for weighted sums of linear and nonlinear
functionals of long-memory stationary Gaussian sequences with regularly varying
deterministic weights.
In the regime
\(\alpha m<1\), \(2d+\alpha m<1\), we identify the natural almost-sure
scale \(A_n(2\log\log n)^{m/2}\) and prove a reduction principle showing
that all Wiener-chaos components above the leading Hermite rank are
negligible on this scale. We also obtain a general almost-sure upper
bound under the time-domain assumptions alone. In the linear case
\(m=1\), the resulting weighted LIL is sharp after variance
normalization, with the classical constant one.

For \(m\ge2\), the main result is a weighted functional LIL. We identify
the almost-sure cluster set in \(C[0,T]\) as
\[
\mathcal K_T
=
\left\{
t\mapsto
\langle Q_t,\xi^{\otimes m}\rangle_{\mathfrak H_{\mathbb R}^{\otimes m}}
:
\|\xi\|_{\mathfrak H_{\mathbb R}}\le1
\right\},
\]
and consequently obtain the sharp nonlinear LIL
\[
\limsup_{n\to\infty}
\frac{|S_n|}
     {A_n(2\log\log n)^{m/2}}
=
\Lambda_m
\qquad\text{a.s.},
\]
with the exact variational characterization
\[
\Lambda_m
=
\sup_{\|\xi\|_{\mathfrak H_{\mathbb R}}\le1}
\left|
\langle Q_1,\xi^{\otimes m}\rangle_{\mathfrak H_{\mathbb R}^{\otimes m}}
\right|.
\]
Moreover, we prove the strict two-sided estimate
\[
0<L_{m,\alpha,d}\le \Lambda_m < (m!)^{-1/2},
\]
where \(L_{m,\alpha,d}\) is given explicitly in closed form through beta
functions. Numerical evaluation of the variational problem shows that
this closed-form lower bound reproduces the computed values of
\(\Lambda_m\) to within a fraction of one percent over the parameter
range examined. In contrast with the linear constant, the nonlinear
constant depends on both the memory parameter \(\alpha\) and the weight
exponent \(d\).

Thus the deterministic weights affect the LIL at two distinct levels.
They enter the normalization through the variance scale and the exponent
\(H=1-d-\alpha m/2\), but for \(m\ge2\) they also remain in the limiting
kernel and change the
geometry of the functional cluster set and the value of the sharp
nonlinear LIL constant itself. 

A further contribution is methodological. For \(m\ge2\) the leading term
belongs to a fixed higher Wiener chaos, where Gaussian comparison
arguments no longer apply. The nonlinear lower cluster
inclusion is obtained pathwise by recurrence of dilated Gaussian
spectral coordinates, avoiding direct decorrelation estimates for sparse
exceedance events in a fixed higher Wiener chaos. Combined with
finite-rank approximation, uniform kernel convergence, Borell--Tsirelson
bounds, and fixed-chaos maximal estimates, this yields the full
functional cluster theorem.

The results also suggest further applications to weighted statistics and
learning procedures under long-range dependence. Whenever an accumulated
quantity can be represented in the form
\(S_n=\sum_{t=1}^n a_tG(X_t)\) with a long-memory Gaussian input,
regularly varying deterministic weights, and a nonlinear functional
\(G\), the theory developed here gives its exact almost-sure fluctuation
scale and, in the nonlinear regime, its sharp pathwise LIL constant.
Possible directions include weighted and online procedures under
long-range dependence, calibration of nonlinear long-memory statistics,
and power-law weighted aggregation. These directions are left for future
work.

\section{Appendix}
\begin{lemma}[Two-dimensional regularly varying Riemann sum]
\label{lem:two-dimensional-rv-riemann}
Let \(\beta\in(0,1)\), \(2d+\beta<1\), and let \(\ell\) and \(L\) be eventually positive
slowly varying functions. Then
\[
\sum_{\substack{1\le s,t\le n\\s\ne t}}
s^{-d}\ell(s)t^{-d}\ell(t)|t-s|^{-\beta}L(|t-s|)
\sim
n^{2-2d-\beta}\ell(n)^2L(n)I_{\beta,d},
\]
where
\(
I_{\beta,d}:=
\int_0^1\int_0^1x^{-d}y^{-d}|x-y|^{-\beta}\,dx\,dy<\infty .
\)
\end{lemma}

\begin{proof}
The finiteness of \(I_{\beta,d}\) follows from \(\beta<1\), \(d<1\), and
\(2d+\beta<1\). Write the normalized sum as
\[
\frac1{n^2}
\sum_{\substack{1\le s,t\le n\\s\ne t}}
\left(\frac sn\right)^{-d}
\left(\frac tn\right)^{-d}
\left|\frac{t-s}{n}\right|^{-\beta}
\frac{\ell(s)}{\ell(n)}
\frac{\ell(t)}{\ell(n)}
\frac{L(|t-s|)}{L(n)} .
\]
On every set
\(
D_\varepsilon:=\{(x,y)\in[0,1]^2:x,y\ge\varepsilon,\ |x-y|\ge\varepsilon\},
\)
the uniform convergence theorem for slowly varying functions gives uniform convergence of the
slowly varying ratios to \(1\). Hence the contribution from \(D_\varepsilon\) converges to
\(\int_{D_\varepsilon}x^{-d}y^{-d}|x-y|^{-\beta}\,dx\,dy\).

It remains to control the complement of \(D_\varepsilon\). Choose \(\eta>0\) so small that
\(\beta+\eta<1\) and \(2d+\beta+4\eta<1\). Potter's bound, with \(C_\eta\) enlarged to cover the cases where
\(s\), \(t\), or \(|t-s|\) is below the Potter threshold, gives, uniformly for
\(1\le s,t\le n\), \(s\ne t\),
\[
\frac{\ell(s)}{\ell(n)}\frac{\ell(t)}{\ell(n)}\frac{L(|t-s|)}{L(n)}
\le
C_\eta
\left(\frac sn\right)^{-\eta}
\left(\frac tn\right)^{-\eta}
\left|\frac{t-s}{n}\right|^{-\eta}.
\]
Put
\[
a:=(d+\eta)_+,\qquad b:=\beta+\eta,\qquad
g(x,y):=x^{-a}y^{-a}|x-y|^{-b}.
\]
Then \(b<1\) and \(2a+b<1\), so \(g\in L^1((0,1)^2)\). Moreover,
\(x^{-d-\eta}\le x^{-a}\) for \(0<x\le1\). Hence the preceding
Potter bound reduces the required estimate to the corresponding
Riemann sums of \(g\).

For \(1\le s\le n\), let
\(C_{s,n}:=((s-1)/n,s/n]\). If \(s\ne t\) and
\((x,y)\in C_{s,n}\times C_{t,n}\), then
\[
x^{-a}\ge(s/n)^{-a},\qquad
y^{-a}\ge(t/n)^{-a},\qquad
|x-y|\le\frac{|t-s|+1}{n}\le2\frac{|t-s|}{n}.
\]
Consequently,
\[
\frac1{n^2}g(s/n,t/n)
\le
2^b\int_{C_{s,n}\times C_{t,n}}g(x,y)\,dx\,dy.
\]
For \(n\ge\varepsilon^{-1}\), the union of the cells corresponding
to \((s/n,t/n)\in D_\varepsilon^c\) is contained in
\(D_{2\varepsilon}^c\). Therefore
\[
\limsup_{n\to\infty}
\frac1{n^2}
\sum_{\substack{s\ne t\\(s/n,t/n)\in D_\varepsilon^c}}
g(s/n,t/n)
\le
2^b\int_{D_{2\varepsilon}^c}g(x,y)\,dx\,dy
\longrightarrow0
\]
as \(\varepsilon\downarrow0\). Combining this estimate with the
convergence on \(D_\varepsilon\) proves the lemma.\end{proof}

\begin{lemma}[Variance gain for the higher-rank remainder]
\label{lem:higher-rank-variance-gain}
Assume \((\ref{A1})\), \((\ref{A3})\), and \((\ref{A4})\), \(\alpha m<1\), and \(2d+\alpha m<1\). Let
\(
    \widetilde G(x):=G(x)-\frac{J_m}{m!}H_m(x).
\)
Then \(\widetilde G\) has Hermite rank at least \(m+1\). Moreover, there exist
\(C<\infty\), \(\delta>0\), and \(\gamma\ge1\) such that, for every sufficiently
large \(N\) and every interval \(I\subset\{N+1,\ldots,2N\}\),
\[
    \operatorname{Var}\left(\sum_{t\in I}a_t\widetilde G(X_t)\right)
    \le
    C v_N^2N^{-2\delta}\left(\frac{|I|}{N}\right)^\gamma .
\]
In particular, after decreasing \(\delta\) if necessary,
\[
    \left\|\sum_{t=1}^n a_t\widetilde G(X_t)\right\|_2
    \le C v_n n^{-\delta},
    \qquad n\ge3,
\]
and hence
\(
    \operatorname{Var}\bigl(\sum_{t=1}^n a_t\widetilde G(X_t)\bigr)
    =o(v_n^2).
\)
\end{lemma}
\begin{proof}
Since
\(
    \widetilde G(x)=\sum_{q=m+1}^{\infty}\frac{J_q}{q!}H_q(x),
\)
we have, for \(h\ge0\),
\[
    \operatorname{Cov}\bigl(\widetilde G(X_0),\widetilde G(X_h)\bigr)
    =
    \sum_{q=m+1}^{\infty}\frac{J_q^2}{q!}\rho(h)^q .
\]
Because \(\rho(h)\to0\), for all sufficiently large \(h\),
\[
    \left|\operatorname{Cov}\bigl(\widetilde G(X_0),\widetilde G(X_h)\bigr)\right|
    \le C|\rho(h)|^{m+1}.
\]
Thus, by regular variation of \(\rho\) and Potter's bound, for every \(\eta>0\), after
increasing \(C_\eta\) to cover bounded lags,
\[
    \left|\operatorname{Cov}\bigl(\widetilde G(X_0),\widetilde G(X_h)\bigr)\right|
    \le C_\eta(1+h)^{-\alpha(m+1)+\eta},
    \qquad h\ge0.
\]

Let \(I\subset\{N+1,\ldots,2N\}\) and put \(\ell:=|I|\). Regular variation of the weights
gives \(|a_t|\le C N^{-d}\ell_a(N)\) for \(t\in I\). Hence
\[
    \operatorname{Var}\left(\sum_{t\in I}a_t\widetilde G(X_t)\right)
    \le C N^{-2d}\ell_a(N)^2
    \sum_{u,v\in I}(1+|u-v|)^{-\alpha(m+1)+\eta}.
\]
Choose \(\eta>0\) so small that \(2\eta<\min\{\alpha,1-\alpha m\}\). If
\(\alpha(m+1)-\eta<1\), the double sum is at most
\(C\ell^{2-\alpha(m+1)+\eta}\). If \(\alpha(m+1)-\eta\ge1\), it is at most
\(C\ell\log(2\ell)\le C_\eta\ell^{1+\eta}\).

By definition, \(v_N^2=N^{2-2d-\alpha m}\ell_a(N)^2L_0(N)^m\). In the first case set
\(\gamma:=2-\alpha(m+1)+\eta\) and choose
\(0<\delta<(\alpha-2\eta)/2\). Since \(L_0(N)^{-m}\le C_\eta N^\eta\),
\[
    N^{-2d}\ell_a(N)^2\ell^\gamma
    \le C v_N^2N^{-\alpha+2\eta}\left(\frac{\ell}{N}\right)^\gamma
    \le C v_N^2N^{-2\delta}\left(\frac{\ell}{N}\right)^\gamma.
\]
In the second case set \(\gamma:=1+\eta\) and choose
\(0<\delta<(1-\alpha m-2\eta)/2\). Then
\[
    N^{-2d}\ell_a(N)^2\ell^\gamma
    \le C v_N^2N^{-(1-\alpha m)+2\eta}\left(\frac{\ell}{N}\right)^\gamma
    \le C v_N^2N^{-2\delta}\left(\frac{\ell}{N}\right)^\gamma.
\]
This proves the interval estimate.

It remains to pass to the global estimate. We use the triangle inequality in \(L^2\), not a
sum of variances, because the dyadic blocks are dependent. Decrease \(\delta\), if necessary,
so that \(0<\delta<H\). Choose \(j_0\) such that the interval estimate holds for
\(N=2^j\), \(j\ge j_0\). If \(2^K\le n<2^{K+1}\), then
\[
    \|R_n\|_2
    \le \|R_{2^{j_0}}\|_2
      +\sum_{j=j_0}^{K-1}
       \left\|\sum_{t=2^j+1}^{2^{j+1}}a_t\widetilde G(X_t)\right\|_2
      +\left\|\sum_{t=2^K+1}^{n}a_t\widetilde G(X_t)\right\|_2 
    \le C_0+C\sum_{j=j_0}^{K}v_{2^j}2^{-j\delta}.
\]
Since \(v_n\) is regularly varying with index \(H>0\), Potter's bound gives, for every
\(\varepsilon>0\),
\[
    \frac{v_{2^j}}{v_{2^K}}
    \le C_\varepsilon 2^{-(K-j)(H-\varepsilon)},
    \qquad j_0\le j\le K.
\]
Choose \(\varepsilon>0\) with \(H-\varepsilon-\delta>0\). Then
\[
    \sum_{j=j_0}^{K}v_{2^j}2^{-j\delta}
    \le C v_{2^K}2^{-K\delta}.
\]
The fixed term \(C_0\) is absorbed because \(v_{2^K}2^{-K\delta}\to\infty\). Since
\(v_{2^K}\asymp v_n\) and \(2^{-K\delta}\asymp n^{-\delta}\),
\[
    \|R_n\|_2\le C v_n n^{-\delta}.
\]
Consequently, \(\operatorname{Var}(R_n)\le C v_n^2n^{-2\delta}=o(v_n^2)\).
\end{proof}

\begin{proof}[Proof of Proposition~\ref{prop:exact-variance}]
We first treat the leading chaos. Since \\
\(
    \operatorname{Cov}(H_m(X_s),H_m(X_t))
=
m!\bigl(\rho(|t-s|)\bigr)^m,
\)
we have
\(
    \operatorname{Var}(S_{n,m})
=
\frac{J_m^2}{m!}\sum_{s,t\le n}a_sa_t
\bigl(\rho(|t-s|)\bigr)^m.
\)
The diagonal and finitely many fixed-lag terms are
\(
    O(\sum_{t\le n}a_t^2)=O(n^{1-2d}\ell_a(n)^2),
\)
which is \(o(v_n^2)\) because \(\alpha m<1\). 
For all sufficiently large \(h\), put
\(
    L_\rho(h):=c^{-m}h^{\alpha m}\rho(h)^m.
\)
Then \(L_\rho\) is eventually positive and slowly varying, and
\(L_\rho(h)\sim L_0(h)^m\). The finitely many remaining lags are negligible.
Applying Lemma~\ref{lem:two-dimensional-rv-riemann} with
\(\beta=\alpha m\), \(\ell=\ell_a\), and \(L=L_\rho\), we obtain
\[
    \operatorname{Var}(S_{n,m})
    \sim
    \frac{J_m^2c^m}{m!}
    n^{2-2d-\alpha m}\ell_a(n)^2L_0(n)^m 
  \times
    \int_0^1\int_0^1x^{-d}y^{-d}|x-y|^{-\alpha m}\,dx\,dy
    =\sigma_m^2v_n^2.
\]

Write \(S_n=S_{n,m}+R_n\). Different Wiener chaoses are orthogonal, so
\(
    A_n^2=\operatorname{Var}(S_{n,m})+\operatorname{Var}(R_n).
\)
By Lemma~\ref{lem:higher-rank-variance-gain}, \(\operatorname{Var}(R_n)=o(v_n^2)\).
Therefore \(A_n^2\sim\operatorname{Var}(S_{n,m})\sim\sigma_m^2v_n^2\). Hence
\(A_n\sim\sigma_m v_n\), and \(A_n\) is regularly varying with index \(H\).
\end{proof}

\begin{lemma}[Maximal inequality on dyadic blocks]
\label{lem:dyadic-maximal}
Let $N\ge2$, and let $Z_{N+1},\ldots,Z_{2N}$ be centered square-integrable random variables.
For an interval $I\subset\{N+1,\ldots,2N\}$ write
\(
    Z(I):=\sum_{t\in I}Z_t .
\)
Assume that there exist constants $B_N>0$ and $\gamma\ge1$ such that, for every interval
$I\subset\{N+1,\ldots,2N\}$,
\[
    \mathbb E|Z(I)|^2
    \le
    B_N^2\left(\frac{|I|}{N}\right)^\gamma .
\]
Then there exists a universal constant $C<\infty$ such that
\[
    \mathbb E
    \max_{N\le n\le2N}
    \left|
        \sum_{t=N+1}^{n}Z_t
    \right|^2
    \le
    C(\log N)^2B_N^2 .
\]
\end{lemma}

\begin{proof}
This standard dyadic maximal estimate is included for completeness.

Let $L:=\lfloor\log_2 N\rfloor$.
For each $n\in[N,2N]$, the interval
$\{N+1,\ldots,n\}$ can be decomposed into at most $L+1$ dyadic subintervals. Hence
\[
    \max_{N\le n\le2N}
    \left|
        \sum_{t=N+1}^{n}Z_t
    \right|
    \le
    \sum_{\ell=0}^{L} M_\ell ,
\]
where $M_\ell$ denotes the maximum of $|Z(I)|$ over all dyadic intervals
$I\subset\{N+1,\ldots,2N\}$ of length $2^\ell$.

By Cauchy's inequality,
\(
    \left(\sum_{\ell=0}^{L}M_\ell\right)^2
    \le
    (L+1)\sum_{\ell=0}^{L}M_\ell^2 .
\)
Therefore
\[
    \mathbb E
    \max_{N\le n\le2N}
    \left|
        \sum_{t=N+1}^{n}Z_t
    \right|^2
    \le
    (L+1)\sum_{\ell=0}^{L}\mathbb E M_\ell^2 .
\]
For fixed $\ell$, there are at most $2N/2^\ell$ dyadic intervals of length $2^\ell$. Hence,
using the interval-variance assumption,
\[
    \mathbb E M_\ell^2
    \le
    \sum_{I:\ |I|=2^\ell}\mathbb E|Z(I)|^2
    \le
    \frac{2N}{2^\ell}
    B_N^2
    \left(\frac{2^\ell}{N}\right)^\gamma .
\]
Since $\gamma\ge1$ and $2^\ell\le N$,
\(
    \frac{2N}{2^\ell}
    \left(\frac{2^\ell}{N}\right)^\gamma
    =
    2\left(\frac{2^\ell}{N}\right)^{\gamma-1}
    \le 2 .
\)
Thus
\(
    \mathbb E M_\ell^2\le 2B_N^2
\)
for every $\ell$. Consequently,
\[
    \mathbb E
    \max_{N\le n\le2N}
    \left|
        \sum_{t=N+1}^{n}Z_t
    \right|^2
    \le
    2(L+1)^2B_N^2
    \le
    C(\log N)^2B_N^2 .
\]
This proves the lemma.
\end{proof}

\begin{lemma}[Local variance bound for leading-chaos increments]
\label{lem:leading-chaos-local-variance}
Let
\[
    S_{n,m}:=\frac{J_m}{m!}\sum_{t=1}^n a_tH_m(X_t).
\]
Assume \((\ref{A1})\), \((\ref{A3})\), and \((\ref{A4})\). Let \(m\ge1\) be fixed and suppose that
\(\alpha m<1\) and \(2d+\alpha m<1\).

Then for every sufficiently small \(\eta>0\) there exists \(C_\eta<\infty\) such that, 
for all sufficiently large \(N\) and all intervals \(I\subset\{N+1,\ldots,2N\}\),
\[
    \operatorname{Var}\left(
        \frac{J_m}{m!}\sum_{t\in I}a_tH_m(X_t)
    \right)
    \le
    C_\eta A_N^2
    \left(\frac{|I|}{N}\right)^{2-\alpha m-\eta}.
\]
In particular, \(\eta>0\) may be chosen so small that \(2-\alpha m-\eta>1\).
\end{lemma}

\begin{proof}
Only the fixed-chaos covariance identity
\(
    \operatorname{Cov}(H_m(X_s),H_m(X_t))=m!(\rho(t-s))^m
\)
is used; the Hermite-rank assumption on \(G\) is not needed in this lemma.

Let \(L:=|I|\). Since \(I\subset\{N+1,\ldots,2N\}\), regular variation of the weights gives
\[
    |a_t|\le C N^{-d}\ell_a(N),
    \qquad t\in I.
\]
Therefore
\[
\operatorname{Var}\left(
    \frac{J_m}{m!}\sum_{t\in I}a_tH_m(X_t)
\right)
\le
C N^{-2d}\ell_a(N)^2
\sum_{s,t\in I}|\rho(t-s)|^m.
\]

By the usual diagonal decomposition,
\(
    \sum_{s,t\in I}|\rho(|t-s|)|^m
    \le
    C\sum_{h=0}^{L-1}(L-h)|\rho(h)|^m.
\)
Fix \(\eta>0\) so small that \(\alpha m+\eta<1\). 

By \ref{A1} and Potter's bound, uniformly for \(1\le h\le L\),
\(
    |\rho(h)|^m
    \le
    C_\eta h^{-\alpha m-\eta}L^\eta L_0(L)^m.
\)

Hence
\(
\sum_{s,t\in I}|\rho(|t-s|)|^m
\le
C_\eta L^\eta L_0(L)^m
\sum_{h=1}^{L-1}(L-h)h^{-\alpha m-\eta}
+
CL.
\)
Since \(\alpha m+\eta<1\), Karamata's theorem yields
\[
    \sum_{h=1}^{L-1}(L-h)h^{-\alpha m-\eta}
    \le
    C_\eta L^{2-\alpha m-\eta}.
\]
The diagonal term \(CL\) is absorbed into
\(C_\eta L^{2-\alpha m}L_0(L)^m\), since
\(L^{1-\alpha m}L_0(L)^m\to\infty\) along powers and, by Potter bounds,
a slowly varying factor cannot compensate the positive power \(L^{1-\alpha m}\).
The finitely many small values of \(L\) are absorbed into the constant. Thus
\[
    \sum_{s,t\in I}|\rho(|t-s|)|^m
    \le
    C_\eta L^{2-\alpha m}L_0(L)^m.
\]

Since \(1\le L\le N\), Potter's bound gives
\(
    L_0(L)^m
    \le
    C_\eta\left(\frac{N}{L}\right)^\eta L_0(N)^m.
\)
Consequently,
\[
    \sum_{s,t\in I}|\rho(t-s)|^m
    \le
    C_\eta L^{2-\alpha m-\eta}N^\eta L_0(N)^m.
\]
Therefore
\[
\begin{aligned}
    \operatorname{Var}\left(
        \frac{J_m}{m!}\sum_{t\in I}a_tH_m(X_t)
    \right)
    &\le
    C_\eta N^{-2d}\ell_a(N)^2
    L^{2-\alpha m-\eta}N^\eta L_0(N)^m \\
    &=
    C_\eta
    N^{2-2d-\alpha m}\ell_a(N)^2L_0(N)^m
    \left(\frac{L}{N}\right)^{2-\alpha m-\eta}.
\end{aligned}
\]
By Proposition~\ref{prop:exact-variance},
\(A_N^2\asymp N^{2-2d-\alpha m}\ell_a(N)^2L_0(N)^m\), and the result follows.
\end{proof}

\begin{lemma}[Fixed-chaos maximal tail estimate]
\label{lem:fixed-chaos-maximal-tail}
Let \(m\ge1\)
and  \(Z_1,\ldots,Z_N\) be random variables such that every interval sum
\[
    Z(I):=\sum_{j\in I}Z_j,
    \qquad I\subset\{1,\ldots,N\},
\]
belongs to the fixed \(m\)-th Wiener chaos. Assume that for some \(B>0\) and some \(q>1\),
\[
    \operatorname{Var}(Z(I))
    \le
    B^2\left(\frac{|I|}{N}\right)^q
\]
for every interval \(I\subset\{1,\ldots,N\}\). Then there exist constants
\(C_1,C_2>0\), depending only on \(m\) and \(q\), such that for every \(x>0\),
\[
    \mathbb P\left(
        \max_{1\le r\le N}
        \left|\sum_{j=1}^r Z_j\right|
        >
        xB
    \right)
    \le
    C_1\exp\{-C_2x^{2/m}\}.
\]
\end{lemma}

\begin{proof}
Set
\(
    T_r:=\sum_{j=1}^r Z_j,
    \quad 0\le r\le N,
\)
with \(T_0=0\). For any \(0\le s<r\le N\), the increment \(T_r-T_s\) is an interval sum in the
\(m\)-th Wiener chaos. By the variance assumption,
\[
    \|T_r-T_s\|_2
    \le
    B\left(\frac{r-s}{N}\right)^{q/2}.
\]
Hypercontractivity (see \cite{Janson1997}) in a fixed Wiener chaos gives, for every \(p\ge2\),
\[
    \|T_r-T_s\|_p
    \le
    C_m p^{m/2}B
    \left(\frac{r-s}{N}\right)^{q/2}.
\]

We use a dyadic chaining argument. Let \(L_N:=\lceil \log_2 N\rceil\). For each
\(0\le \ell\le L_N\), let \(\mathcal P_\ell\) be the dyadic partition of \(\{0,\ldots,N\}\)
into intervals of length at most \(N2^{-\ell}\). Every \(r\in\{0,\ldots,N\}\) can be connected
to \(0\) through dyadic increments, and therefore
\[
    \max_{0\le r\le N}|T_r|
    \le
    \sum_{\ell=0}^{L_N}
    \max_{I\in\mathcal D_\ell}|Z(I)|,
\]
where \(\mathcal D_\ell\) is a collection of at most \(C2^\ell\) dyadic intervals of length at
most \(N2^{-\ell}\). Hence, for \(p\ge2\),
$$
    \left\|\max_{0\le r\le N}|T_r|\right\|_p
    \le
    \sum_{\ell=0}^{L_N}
    \left\|\max_{I\in\mathcal D_\ell}|Z(I)|\right\|_p \\
    \le
    \sum_{\ell=0}^{L_N}
    (C2^\ell)^{1/p}\sup_{I\in\mathcal D_\ell}\|Z(I)\|_p \\
    \le
    C_{m,q}B p^{m/2}
    \sum_{\ell=0}^{L_N}2^{\ell/p}2^{-\ell q/2}.
$$
Since \(q>1\) and \(p\ge2\), the last sum is bounded uniformly in \(p\) and \(N\). Thus
\[
    \left\|\max_{0\le r\le N}|T_r|\right\|_p
    \le
    C_{m,q}B p^{m/2},
    \qquad p\ge2.
\]

By Markov's inequality,
\[
    \mathbb P\left(
        \max_{0\le r\le N}|T_r|>xB
    \right)
    \le
    \left(\frac{C_{m,q}p^{m/2}}{x}\right)^p .
\]
If \(x\) is large, choose \(p=\lfloor c x^{2/m}\rfloor\) with \(c>0\) sufficiently small. This gives
\[
    \mathbb P\left(
        \max_{0\le r\le N}|T_r|>xB
    \right)
    \le
    C_1\exp\{-C_2x^{2/m}\}.
\]
For bounded \(x\), the estimate follows by increasing \(C_1\). This proves the lemma.
\end{proof}

\begin{lemma}[Gaussian comparison for shifted balls]
\label{lem:gaussian-shifted-ball-comparison}
Let \((X,Y)\) be a centered Gaussian vector in \(\mathbb R^N\times\mathbb R^N\) such that
\(
    \operatorname{Cov}(X)=I_N,
    \qquad
    \operatorname{Cov}(Y)=I_N.
\)
Let
\(
    R:=\operatorname{Cov}(X,Y).
\)
Assume that
\(
    \|R\|_{\mathrm{op}}\le \frac12 .
\)
Then for every \(\varepsilon>0\) there exists a constant
\(C=C(N,\varepsilon)<\infty\) such that, for all \(u,v\in\mathbb R^N\),

\[
\left|
    \mathbb P\bigl(X\in B(u,\varepsilon),\,Y\in B(v,\varepsilon)\bigr)
    -
    \mathbb P\bigl(X\in B(u,\varepsilon)\bigr)
    \mathbb P\bigl(Y\in B(v,\varepsilon)\bigr)
\right|  
\]   
\[                                 
\le
C\,\|R\|_{\mathrm{op}}
\bigl(1+|u|^2+|v|^2\bigr)
\exp\left\{
    C\|R\|_{\mathrm{op}}\bigl(1+|u|^2+|v|^2\bigr)
\right\}                                      
\times
\mathbb P\bigl(X\in B(u,\varepsilon)\bigr)
\mathbb P\bigl(Y\in B(v,\varepsilon)\bigr).
\]
\end{lemma}

\begin{proof}
Let $\varphi_N$ denote the standard Gaussian density on $\mathbb R^N$,
and let $p_R$ be the joint density of $(X,Y)$; write
$p_R(x,y)=\varphi_N(x)\varphi_N(y)\Lambda_R(x,y)$. The joint covariance
matrix is $\Sigma_R=\begin{pmatrix} I_N & R\\ R^*& I_N\end{pmatrix}$,
and for $\|R\|_{\mathrm{op}}\le\frac12$ the block-inverse formula and
the Neumann series give, in operator norm,
\[
    \Sigma_R^{-1}
    =\begin{pmatrix} I_N & -R\\ -R^* & I_N\end{pmatrix}
     +O(\|R\|_{\mathrm{op}}^2),
    \qquad
    \log\det\Sigma_R=\log\det(I_N-R^*R)=O(\|R\|_{\mathrm{op}}^2).
\]
Substituting into
$\log\Lambda_R(x,y)=-\frac12\log\det\Sigma_R
-\frac12\bigl\langle(\Sigma_R^{-1}-I_{2N})(x,y),(x,y)\bigr\rangle$
yields $|\log\Lambda_R(x,y)|\le
C_N\|R\|_{\mathrm{op}}(1+|x|^2+|y|^2)$ uniformly for
$\|R\|_{\mathrm{op}}\le\frac12$, whence, by $|e^z-1|\le|z|e^{|z|}$,
\[
    |\Lambda_R(x,y)-1|
    \le C_N\|R\|_{\mathrm{op}}(1+|x|^2+|y|^2)
    \exp\bigl\{C_N\|R\|_{\mathrm{op}}(1+|x|^2+|y|^2)\bigr\}.
\]
Integrating $\varphi_N(x)\varphi_N(y)\bigl(\Lambda_R(x,y)-1\bigr)$ over
$B(u,\varepsilon)\times B(v,\varepsilon)$ and using
$1+|x|^2+|y|^2\le C_\varepsilon(1+|u|^2+|v|^2)$ on this set proves the
claim.
\end{proof}

\begin{lemma}[Finite-dimensional sparse recurrence]
\label{lem:finite-dimensional-sparse-recurrence}
Let \(Z_n=(Z_{n,1},\ldots,Z_{n,N})\), \(n\ge3\), be centered Gaussian vectors in
\(\mathbb R^N\) such that
\(
    \operatorname{Cov}(Z_n)=I_N
\)
for every \(n\). Assume that there exist constants \(C<\infty\) and \(\theta>0\) such that, for
all \(3\le n\le m\),
\[
    \max_{1\le i,j\le N}
    \left|
        \operatorname{Cov}(Z_{n,i},Z_{m,j})
    \right|
    \le
    C\left(\frac{n}{m}\right)^\theta .
\]
Then for every \(a\in\mathbb R^N\) with \(|a|<1\) and every \(\varepsilon>0\), there exists
\(\gamma>1\) and a sparse sequence
\(
    n_k:=\lfloor \exp(k^\gamma)\rfloor
\)
such that, with
\(
    r_k:=(2\log\log n_k)^{1/2},
\)
one has
\[
    \mathbb P\left(
        |Z_{n_k}-r_k a|<\varepsilon
        \ \text{\rm for infinitely many } k
    \right)=1.
\]
\end{lemma}

\begin{proof}
Fix \(a\in\mathbb R^N\) with \(|a|<1\), and choose \(\gamma>1\) so close to \(1\) that
\(
    \gamma |a|^2<1.
\)
Let
\[
    n_k:=\lfloor \exp(k^\gamma)\rfloor,
    \qquad
    r_k:=(2\log\log n_k)^{1/2},
\]
and define the events
\[
    E_k:=\{|Z_{n_k}-r_k a|<\varepsilon\}.
\]

We first show that
\(
    \sum_{k=1}^\infty \mathbb P(E_k)=\infty.
\)
Since \(Z_{n_k}\sim N(0,I_N)\), we have
\[
    \mathbb P(E_k)
    =
    \int_{B(r_ka,\varepsilon)}
        (2\pi)^{-N/2}\exp\left\{-\frac{|x|^2}{2}\right\}\,dx .
\]
Choose \(\eta>0\) so small that
\(
    \gamma(|a|+\eta)^2<1.
\)
For all sufficiently large \(k\), every \(x\in B(r_ka,\varepsilon)\) satisfies
\(
    |x|\le r_k(|a|+\eta),
\)
because \(\varepsilon/r_k\to0\). Hence
\[
    \mathbb P(E_k)
    \ge
    c_\varepsilon
    \exp\left\{
        -\frac12 r_k^2(|a|+\eta)^2
    \right\}.
\]
Since
\(
    r_k^2=2\log\log n_k
    \sim 2\gamma\log k,
\)
we obtain, for all sufficiently large \(k\),
\(
    \mathbb P(E_k)
    \ge
    c k^{-\gamma(|a|+\eta)^2}.
\)
By the choice of \(\eta\),
\(
    \gamma(|a|+\eta)^2<1,
\)
and therefore
\(
    \sum_{k=1}^\infty \mathbb P(E_k)=\infty.
\)

Put
\(
    p_k:=\mathbb P(E_k),
    \qquad
    S_M:=\sum_{k=1}^M p_k .
\)
Then \(S_M\to\infty\).

We next estimate the correlations between the events \(E_j\) and \(E_k\). For \(j<k\), set
\[
    R_{j,k}:=\operatorname{Cov}(Z_{n_j},Z_{n_k}).
\]
By the covariance assumption,
\[
    \|R_{j,k}\|_{\mathrm{op}}
    \le
    N
    \max_{1\le p,q\le N}
    |\operatorname{Cov}(Z_{n_j,p},Z_{n_k,q})|
    \le
    C_N\left(\frac{n_j}{n_k}\right)^\theta ,
\]
where \(C_N<\infty\) depends only on the dimension \(N\) and on the constant in the
covariance assumption.

Since
\(
    n_k=\lfloor \exp(k^\gamma)\rfloor,
\)
there are constants \(c,C>0\) such that
\(
    \|R_{j,k}\|_{\mathrm{op}}
    \le
    C\exp\{-c(k^\gamma-j^\gamma)\},
    \qquad j<k .
\)

Hence there exists \(k_0\) such that
\(
    \|R_{j,k}\|_{\mathrm{op}}\le\frac12
    \qquad\text{for all } k>j\ge k_0 .
\)
The finitely many pairs with \(j<k<k_0\) contribute only \(O(1)\) to the double sums below and
therefore do not affect the limiting estimates. Thus, after discarding finitely many initial
indices, we may apply Lemma~\ref{lem:gaussian-shifted-ball-comparison} to all pairs \(j<k\).

Applying Lemma~\ref{lem:gaussian-shifted-ball-comparison} with
\(
    X=Z_{n_j},
    \qquad
    Y=Z_{n_k},
    \qquad
    u=r_j a,
    \qquad
    v=r_k a,
\)
we get, for all sufficiently large \(j<k\),
\[
\left|
    \mathbb P(E_j\cap E_k)-p_jp_k
\right|                                      
\le
C\|R_{j,k}\|_{\mathrm{op}}
\bigl(1+r_j^2+r_k^2\bigr)
\exp\left\{
    C\|R_{j,k}\|_{\mathrm{op}}
    \bigl(1+r_j^2+r_k^2\bigr)
\right\}
p_jp_k .
\]
Since
\(
    r_k^2\asymp \log k,
\)
and
\(
    \|R_{j,k}\|_{\mathrm{op}}
    \le
    C\exp\{-c(k^\gamma-j^\gamma)\},
\)
the exponential factor is uniformly bounded. Thus
\[
    \left|
        \mathbb P(E_j\cap E_k)-p_jp_k
    \right|
    \le
    C
    e^{-c(k^\gamma-j^\gamma)}
    (1+\log k)
    p_jp_k .
\]

We now show that the total covariance error is negligible. Since \(\gamma>1\), there exists
\(c_\gamma>0\) such that, for \(j<k\),
\(
    k^\gamma-j^\gamma
    \ge
    c_\gamma k^{\gamma-1}(k-j).
\)
Therefore
\[
    b_k
   :=
    (1+\log k)
    \sum_{j<k}
    e^{-c(k^\gamma-j^\gamma)}
    \le
    (1+\log k)
    \sum_{\ell=1}^{k-1}
    e^{-c c_\gamma k^{\gamma-1}\ell}
    \le
    C(1+\log k)e^{-c'k^{\gamma-1}} .
\]
In particular,
\(
    b_k\to0.
\)
Hence

\[
\sum_{1\le j<k\le M}
    \left|
        \mathbb P(E_j\cap E_k)-p_jp_k
    \right|                                      
\le
C\sum_{k=1}^M p_k
    (1+\log k)
    \sum_{j<k}
    e^{-c(k^\gamma-j^\gamma)}p_j .
\]
Since \(0\le p_j\le1\),
\[
    \sum_{j<k}
    e^{-c(k^\gamma-j^\gamma)}p_j
    \le
    \sum_{j<k}
    e^{-c(k^\gamma-j^\gamma)}.
\]
Therefore,
\[
\begin{aligned}
&\sum_{1\le j<k\le M}
    \left|
        \mathbb P(E_j\cap E_k)-p_jp_k
    \right|
    \le
    C\sum_{k=1}^M p_k b_k .
\end{aligned}
\]

Since \(b_k\to0\), \(p_k\ge0\), and
\(S_M=\sum_{k=1}^M p_k\to\infty\), the weighted Ces\`aro
lemma\footnote{If \(b_k\to0\), \(p_k\ge0\) and
\(S_M=\sum_{k\le M}p_k\to\infty\), then
\(\sum_{k\le M}p_kb_k=o(S_M)\): given \(\varepsilon>0\), choose
\(k_0\) with \(\sup_{k>k_0}|b_k|\le\varepsilon\); then
\(\bigl|\sum_{k\le M}p_kb_k\bigr|
\le\sum_{k\le k_0}p_k|b_k|+\varepsilon S_M=O(1)+\varepsilon S_M\).}
gives
\(
    \sum_{k=1}^M p_k b_k=o(S_M).
\)

Consequently,
\[
    \sum_{1\le j<k\le M}
    \left|
        \mathbb P(E_j\cap E_k)-p_jp_k
    \right|
    =
    o(S_M).
\]
In particular,
\[
    \sum_{1\le j<k\le M}
    \left|
        \mathbb P(E_j\cap E_k)-p_jp_k
    \right|
    =
    o(S_M^2),
\]
because \(S_M\to\infty\).

It follows that
\[
    \sum_{j,k=1}^M \mathbb P(E_j\cap E_k)
    =
    \sum_{k=1}^M p_k
    +
    2\sum_{1\le j<k\le M}\mathbb P(E_j\cap E_k)        
   =
    S_M
    +
    2\sum_{1\le j<k\le M}p_jp_k
    +
    o(S_M^2)  
 \]
 \[                                            
    \le
    S_M^2+S_M+o(S_M^2).
\]
Since \(S_M\to\infty\), this gives
\[
    \sum_{j,k=1}^M \mathbb P(E_j\cap E_k)
    \le
    (1+o(1))S_M^2 .
\]

By the Kochen--Stone lemma \citep{KochenStone1964},
\[
    \mathbb P(E_k\ \text{\rm i.o.})
    \ge
    \limsup_{M\to\infty}
    \frac{S_M^2}
    {\sum_{j,k=1}^M\mathbb P(E_j\cap E_k)}
    =1.
\]
Hence
\(
    \mathbb P(E_k\ \text{\rm i.o.})=1.
\)
Equivalently,
\[
    \mathbb P\left(
        |Z_{n_k}-r_k a|<\varepsilon
        \ \text{\rm for infinitely many } k
    \right)=1.
\]

This proves the lemma.
\end{proof}

\begin{lemma}[Full-process leading-chaos increment bound]
\label{lem:full-leading-chaos-increment}
Assume \((\ref{A1})\)--\((\ref{A4})\), \(\alpha m<1\), and \(2d+\alpha m<1\). Let
\(H=1-d-\alpha m/2\). For every \(T<\infty\) there exist constants
\(C_T<\infty\) and \(\beta>1\) such that the following estimates hold.

First, uniformly in \(n\ge3\) and in every interval
\(I\subset\{1,\ldots,\lfloor Tn\rfloor+2\}\),
\[
    \operatorname{Var}\left(
        \frac{J_m}{m!}\sum_{u\in I}a_uH_m(X_u)
    \right)
    \le
    C_T A_n^2\left(\frac{|I|}{n}\right)^\beta .
\]
Second, uniformly in \(0\le s<t\le T\),
\(
    \operatorname{Var}\bigl(\mathcal S_{n,m}(t)-\mathcal S_{n,m}(s)\bigr)
    \le
    C_T A_n^2 |t-s|^\beta .
\)
Equivalently,
\(
    \left\|
        \widetilde F^{\mathrm{full}}_{n,t}
        -
        \widetilde F^{\mathrm{full}}_{n,s}
    \right\|_{L^2(\mathbb R^m)}^2
    \le C_T |t-s|^\beta .
\)
\end{lemma}

\begin{proof}
Choose \(\beta>1\) so small that
\(
    \beta<\min\{2H,\ 2-\alpha m\}.
\)
This is possible because \(2d+\alpha m<1\) gives \(2H=2-2d-\alpha m>1\), and
\(\alpha m<1\) gives \(2-\alpha m>1\).

It is enough to prove the estimate for the non-interpolated sums. The interpolation error is
handled at the end. Let \(I\subset\{1,\ldots,\lfloor Tn\rfloor+2\}\) be an interval of length
\(L\ge1\). We prove
\[
    \operatorname{Var}\left(
        \frac{J_m}{m!}\sum_{u\in I}a_uH_m(X_u)
    \right)
    \le
    C_T A_n^2\left(\frac{L}{n}\right)^\beta .
\]
Decompose \(I\) into dyadic annuli \(I_j:=I\cap(2^j,2^{j+1}]\), plus a finite initial part.
For the finite initial part the bound is immediate after increasing \(C_T\), because
\(A_n^2(L/n)^\beta\) is bounded below on \(L\ge1\) by a multiple of
\(n^{2H-\beta}\) up to slowly varying factors, and \(2H>\beta\).

For \(I_j\neq\varnothing\), Lemma~\ref{lem:leading-chaos-local-variance}, regular variation
of \(A_n\), and Potter's bound give, for every small \(\eta>0\),
\[
    \left\|
        \frac{J_m}{m!}\sum_{u\in I_j}a_uH_m(X_u)
    \right\|_2
    \le
    C_T A_n
    \left(\frac{2^j}{n}\right)^{H-\eta}
    \left(\frac{|I_j|}{2^j}\right)^{q/2},
\]
where \(q:=2-\alpha m-\eta\), and \(\eta\) is chosen so small that
\(\beta<\min\{2H-2\eta,q\}\). Since \(|I_j|\le L\), the last display is bounded by
\[
    C_T A_n
    \left(\frac{L}{n}\right)^{\beta/2}
    \left(\frac{2^j}{n}\right)^{H-\eta-\beta/2}
    \left(\frac{L}{2^j}\right)^{q/2-\beta/2}.
\]
If \(2^j\ge L\), both extra factors are bounded by a summable geometric sequence in \(j\).
If \(2^j<L\), we use the simpler bound
\[
    \left\|
        \frac{J_m}{m!}\sum_{u\in I_j}a_uH_m(X_u)
    \right\|_2
    \le C_T A_n\left(\frac{2^j}{n}\right)^{H-\eta},
\]
and summing over \(2^j<L\) gives \(C_T A_n(L/n)^{H-\eta}\), which is bounded by
\(C_TA_n(L/n)^{\beta/2}\), because \(\beta<2H-2\eta\).

By the triangle inequality in \(L^2\), summing all dyadic annuli gives
\[
    \left\|
        \frac{J_m}{m!}\sum_{u\in I}a_uH_m(X_u)
    \right\|_2
    \le
    C_T A_n\left(\frac{L}{n}\right)^{\beta/2}.
\]
Squaring proves the interval estimate.

Now take \(I=(\lfloor ns\rfloor,\lfloor nt\rfloor]\). Then
\(L\le n|t-s|+2\). If \(n|t-s|\ge1\), the preceding estimate gives
\[
    \operatorname{Var}(S_{\lfloor nt\rfloor,m}-S_{\lfloor ns\rfloor,m})
    \le C_TA_n^2|t-s|^\beta .
\]
If \(n|t-s|<1\), the interpolated increment is a linear combination
of at most two one-step terms, with coefficients bounded by
\(n|t-s|\). Applying the preceding interval estimate with \(L=1\)
and using the triangle inequality, we obtain
\[
    \|\mathcal S_{n,m}(t)-\mathcal S_{n,m}(s)\|_2
    \le
    C_T n|t-s|\,A_n n^{-\beta/2}
    \le
    C_T A_n|t-s|^{\beta/2}.
\]
The last inequality follows from
\(\beta<2-\alpha m<2\): if \(x:=n|t-s|<1\), then
\(x\le x^{\beta/2}\). This proves the interpolated estimate and the
lemma.
\end{proof}

\begin{lemma}[Uniform small-kernel fixed-chaos bound]
\label{lem:uniform-small-kernel-chaos}
Let \(K_{n,t}\in\mathfrak H_{\mathbb R}^{\odot m}\), \(0\le t\le T\), be deterministic
kernels. Assume that, for some \(C_T<\infty\), \(\beta>0\), and \(\eta>0\),
\[
    \sup_{0\le t\le T}\|K_{n,t}\|_2\le\eta,
    \qquad
    \|K_{n,t}-K_{n,s}\|_2^2\le C_T|t-s|^\beta
    \qquad 0\le s,t\le T .
\]
Let \(L_n:=2\log\log n\), and let \(n_k\) be an increasing sequence such that
\(L_{n_k}\ge c_0\log k\) for all sufficiently large \(k\), with some \(c_0>0\). Then for every
\(\delta>0\) there exists \(\eta_0=\eta_0(\delta,T,m,\beta,C_T,c_0)>0\) such that, if
\(\eta\le\eta_0\), then
\[
    \sum_{k=1}^\infty
    \mathbb P\left(
        \sup_{0\le t\le T}|I_m(K_{n_k,t};W_{n_k})|>
        \delta L_{n_k}^{m/2}
    \right)<\infty .
\]
Consequently,
\[
    \limsup_{k\to\infty}
    \sup_{0\le t\le T}
    \frac{|I_m(K_{n_k,t};W_{n_k})|}{L_{n_k}^{m/2}}
    \le\delta
    \qquad\text{\rm a.s.}
\]
\end{lemma}

\begin{proof}
Choose \(p\ge2\) so large that \(p\beta>2\). By fixed-chaos
hypercontractivity,
\[
    \mathbb E\left|
        I_m(K_{n,t}-K_{n,s};W_n)
    \right|^p
    \le
    C_{m,p}\|K_{n,t}-K_{n,s}\|_2^{p}
    \le
    C_{m,p}C_T^{p/2}|t-s|^{p\beta/2}.
\]
Hence the Kolmogorov--Chentsov theorem gives, for every \(n\), a
continuous modification of
\(t\mapsto I_m(K_{n,t};W_n)\). We work with these modifications below.
Let \((\delta_r)_{r\ge0}\) be a positive summable sequence with
\(\sum_{r\ge0}\delta_r\le\delta/2\); for instance take
\(\delta_r=c_\ast\delta(r+1)^{-2}\). Let \(\mathcal P_r\) be a dyadic grid of \([0,T]\)
with mesh at most \(T2^{-r}\), and choose the grids nested. Every \(t\in[0,T]\) can be
represented by a chain \(t_0,t_1,\ldots\), with \(t_r\in\mathcal P_r\),
\(|t_r-t_{r-1}|\le C_T'2^{-r}\), and \(t_r\to t\).

We use the fixed-chaos tail estimate in the form
\[
    \mathbb P(|I_m(K)|>x)
    \le
    C_m\exp\left\{
        -c_m\left(\frac{x}{\|K\|_2}\right)^{2/m}
    \right\}.
\]
For the coarse grid \(\mathcal P_0\), \(\|K_{n_k,u}\|_2\le\eta\), and hence
\[
    \sum_k\sum_{u\in\mathcal P_0}
    \mathbb P\left(
        |I_m(K_{n_k,u};W_{n_k})|>\frac{\delta}{2}L_{n_k}^{m/2}
    \right)
    \le
    C\sum_k
    \exp\left\{
        -c\,\delta^{2/m}L_{n_k}/\eta^{2/m}
    \right\}.
\]
Since \(L_{n_k}\ge c_0\log k\), this series is finite if \(\eta\) is small enough.

For the chaining increments, the kernel increment between neighboring grid points at level
\(r\) has \(L^2\)-norm at most
\[
    \rho_r:=\min\{2\eta,\ C2^{-\beta r/2}\}.
\]
There are at most \(C2^r\) such increments. Therefore
\[
\begin{aligned}
&\sum_{k=1}^\infty\sum_{r=1}^\infty
2^r
\mathbb P\left(
    |I_m(K_{n_k,t_r}-K_{n_k,t_{r-1}};W_{n_k})|
    >
    \delta_r L_{n_k}^{m/2}
\right)                                                     \\
&\qquad\le
C\sum_{k=1}^\infty\sum_{r=1}^\infty
2^r
\exp\left\{
    -c\,\delta_r^{2/m}L_{n_k}/\rho_r^{2/m}
\right\}.
\end{aligned}
\]
We prove that the last double series is finite when \(\eta\) is sufficiently small. Split the
levels into \(r\le r_\eta\) and \(r>r_\eta\), where \(r_\eta\) is chosen so that
\(C2^{-\beta r_\eta/2}\le2\eta<C2^{-\beta(r_\eta-1)/2}\). For \(r\le r_\eta\),
\(\rho_r\le2\eta\); choosing \(\eta\) small gives
\(c\,\delta_r^{2/m}c_0/\rho_r^{2/m}>3\) for every \(0\le r\le r_\eta\).
This is possible because \(r_\eta=O(\log(1/\eta))\), while
\(\eta^{-1}r_\eta^{-2}\to\infty\). Hence the sum over \(k\) is bounded by
\(\sum_k k^{-3}\), and the additional finite sum over \(r\le r_\eta\) is harmless.

For \(r>r_\eta\), \(\rho_r\le C2^{-\beta r/2}\). Thus
\[
    2^r
    \exp\left\{
        -c\,\delta_r^{2/m}L_{n_k}2^{\beta r/m}
    \right\}
\]
is summable over \(r\) and \(k\), because \(\delta_r\) decays only polynomially while
\(2^{\beta r/m}\) grows exponentially. Hence the double series is finite.

By Borel--Cantelli, the coarse-grid terms and all chaining increments are eventually bounded
by the assigned thresholds. Summing the chain, letting \(r\to\infty\), and using the path
continuity of the chosen modifications, we get
\[
    \sup_{0\le t\le T}|I_m(K_{n_k,t};W_{n_k})|
    \le
    \delta L_{n_k}^{m/2}
\]
eventually almost surely. This proves the lemma.
\end{proof}

\begin{lemma}[Low-frequency asymptotics of the rescaled spectral factor]\label{lem:spectral-factor-rescaled}
Assume
\[
f(\lambda)\sim c_f |\lambda|^{\alpha-1}L_0(1/|\lambda|),
\qquad
\lambda\to0,
\]
where \(c_f>0\), \(0<\alpha<1\), and \(L_0\) is slowly varying at infinity. Let
\(
q(\lambda):=f(\lambda)^{1/2}.
\)
For \(n\ge1\), define
\[
\psi_n(x)
:=
n^{-(1-\alpha)/2}L_0(n)^{-1/2}\,q(x/n)\,\mathbf 1_{\{|x|\le \pi n\}}.
\]
Then for every \(R>0\) and every \(\eta\in(0,\alpha)\) the following hold:

\smallskip

\noindent
\textup{(i)} For every \(\varepsilon\in(0,R)\),
\[
\sup_{\varepsilon\le |x|\le R}
\left|
\psi_n(x)-c_f^{1/2}|x|^{(\alpha-1)/2}
\right|
\longrightarrow0.
\]

\noindent
\textup{(ii)} There exists \(C_{R,\eta}<\infty\) such that for all sufficiently large \(n\),
\[
|\psi_n(x)|
\le
C_{R,\eta}|x|^{(\alpha-1-\eta)/2},
\qquad
0<|x|\le R.
\]
\end{lemma}

\begin{proof}
For \(x\neq0\),
\[
q(x/n)
=
c_f^{1/2}|x/n|^{(\alpha-1)/2}L_0(n/|x|)^{1/2}(1+r_n(x)),
\]
where \(r_n(x)\to0\) for each fixed \(x\neq0\). Hence
\[
\psi_n(x)
=
c_f^{1/2}|x|^{(\alpha-1)/2}
\left(\frac{L_0(n/|x|)}{L_0(n)}\right)^{1/2}
(1+r_n(x)).
\]
If \(\varepsilon\le |x|\le R\), then \(1/|x|\in[1/R,1/\varepsilon]\), and the uniform convergence theorem for slowly varying functions gives
\[
\sup_{\varepsilon\le |x|\le R}
\left|
\frac{L_0(n/|x|)}{L_0(n)}-1
\right|
\to0.
\]
This proves \textup{(i)}.

For \textup{(ii)}, Potter's bound yields
\[
\frac{L_0(n/|x|)}{L_0(n)}
\le
C_{R,\eta}|x|^{-\eta},
\qquad
0<|x|\le R,
\]
for all sufficiently large \(n\). Substituting into the previous display gives the result.
\end{proof}

\begin{corollary}[Product asymptotics on compact boxes]\label{cor:product-compact-box}
For \(x=(x_1,\dots,x_m)\in\mathbb R^m\), define
\[
\Psi_n(x):=\prod_{j=1}^m \psi_n(x_j),
\qquad
\Psi(x):=c_f^{m/2}\prod_{j=1}^m |x_j|^{(\alpha-1)/2}.
\]
Let
\(
D_R:=[-R,R]^m,
\quad
K_{\varepsilon,R}:=\{x\in D_R:\min_{1\le j\le m}|x_j|\ge\varepsilon\}.
\)
Then for every \(R>0\), every \(\varepsilon\in(0,R)\), and every \(\eta\in(0,\alpha)\),

\begin{itemize}
\item[(a)]
\(
\sup_{x\in K_{\varepsilon,R}}
|\Psi_n(x)-\Psi(x)|
\longrightarrow0;
\)

\item[(b)] for all sufficiently large \(n\),
\(
|\Psi_n(x)|^2
\le
C_{R,\eta}\prod_{j=1}^m |x_j|^{\alpha-1-\eta},
\qquad
x\in D_R;
\)

\item[(c)]
\(
\|\Psi_n-\Psi\|_{L^2(D_R)}
\longrightarrow0.
\)
\end{itemize}
\end{corollary}

\begin{proof}
Part (a) follows from Lemma~\ref{lem:spectral-factor-rescaled}. Part (b) follows by multiplying the one-dimensional bound from Lemma~\ref{lem:spectral-factor-rescaled}\textup{(ii)}. Since \(\alpha-\eta>0\), the product on the right-hand side is integrable on \(D_R\). Hence dominated convergence yields (c).
\end{proof}

\begin{lemma}[Weighted Riemann sums starting at the origin]
\label{lem:weighted-riemann-sum-full}
Assume \((\ref{A4})\) and \(2d+\alpha m<1\). In particular \(d<1\). Let
\(T<\infty\) and \(B<\infty\). Then
\[
    \sup_{0\le t\le T}\sup_{|\theta|\le B}
    \left|
        \frac{1}{n^{1-d}\ell_a(n)}
        \sum_{u=1}^{\lfloor nt\rfloor}
        a_u e^{iu\theta/n}
        -
        \int_0^t y^{-d}e^{iy\theta}\,dy
    \right|
    \longrightarrow0 .
\]
\end{lemma}

\begin{proof}
Fix \(\varepsilon\in(0,1)\). We first control the initial part
\(0\le u\le \varepsilon n\). Since \(d<1\), Karamata's theorem and Potter's
bounds imply, uniformly in \(0\le t\le\varepsilon\),
\(
    \frac{1}{n^{1-d}\ell_a(n)}
    \sum_{u=1}^{\lfloor nt\rfloor} a_u
    \le
    C\varepsilon^{1-d-\eta}
\)
for every sufficiently small \(\eta>0\). Also
\(
    \sup_{0\le t\le\varepsilon}\sup_{|\theta|\le B}
    \left|
        \int_0^t y^{-d}e^{iy\theta}\,dy
    \right|
    \le
    C\varepsilon^{1-d}.
\)
Thus the contribution of the interval \([0,\varepsilon]\) is uniformly small.

It remains to consider \(t\in[\varepsilon,T]\). On the interval
\([\varepsilon,T]\), the uniform convergence theorem for slowly varying
functions gives
\(
    \sup_{\varepsilon\le y\le T}
    \left|
        \frac{\ell_a(ny)}{\ell_a(n)}-1
    \right|
    \longrightarrow0.
\) \\
Therefore
\(
    \frac{1}{n^{1-d}\ell_a(n)}
    \sum_{u=\lfloor\varepsilon n\rfloor+1}^{\lfloor nt\rfloor}
    a_u e^{iu\theta/n}
\)
is a standard Riemann sum for
\(
    \int_\varepsilon^t y^{-d}e^{iy\theta}\,dy,
\)
uniformly in \(t\in[\varepsilon,T]\) and \(|\theta|\le B\), because the function
\((y,\theta)\mapsto y^{-d}e^{iy\theta}\) is continuous on
\([\varepsilon,T]\times[-B,B]\).

Combining the two estimates and then letting \(\varepsilon\downarrow0\) proves
the claim.
\end{proof}

\begin{corollary}[Compact-frequency uniformity for full weighted sums]
\label{cor:compact-frequency-uniformity-full}
For every \(T<\infty\) and \(R<\infty\),
\[
    \sup_{0\le t\le T}\sup_{x\in[-R,R]^m}
    \left|
        \frac{1}{n^{1-d}\ell_a(n)}
        \sum_{u=1}^{\lfloor nt\rfloor}
        a_u e^{iu(x_1+\cdots+x_m)/n}
        -
        \int_0^t y^{-d}e^{iy(x_1+\cdots+x_m)}\,dy
    \right|
    \longrightarrow0 .
\]
\end{corollary}

\begin{proof}
Apply Lemma~\ref{lem:weighted-riemann-sum-full} with
\(\theta=x_1+\cdots+x_m\). If \(x\in[-R,R]^m\), then
\(|\theta|\le mR\).
\end{proof}

\begin{lemma}[The limiting kernels belong to \(L^2\) and have vanishing tails]
\label{lem:limit-kernel-L2}
Let \(T<\infty\), and define
\[
Q_t(x)
=
C_{m,\alpha,d}
\left(
    \int_0^t y^{-d}e^{iy(x_1+\cdots+x_m)}\,dy
\right)
\prod_{j=1}^m |x_j|^{(\alpha-1)/2},
\qquad 0\le t\le T.
\]
Assume
\(
    0<\alpha<1,\qquad \alpha m<1,\qquad 2d+\alpha m<2 .
\)
Then:

\begin{itemize}
\item[(a)] \(Q_t\in L^2(\mathbb R^m)\) for every \(t\in[0,T]\);
\item[(b)] the map \(t\mapsto Q_t\) is continuous from \([0,T]\) into
\(L^2(\mathbb R^m)\);
\item[(c)]
\(
    \lim_{R\to\infty}
    \sup_{0\le t\le T}
    \int_{([-R,R]^m)^c}|Q_t(x)|^2\,dx
    =0 .
\)
Moreover, for \(0\le s\le t\le T\),
\[
    \|Q_t-Q_s\|_{L^2(\mathbb R^m)}^2
    =
    \widetilde C_{m,\alpha,d}
    \int_s^t\int_s^t
    y^{-d}z^{-d}|y-z|^{-\alpha m}\,dy\,dz ,
\]
where \(0<\widetilde C_{m,\alpha,d}<\infty\).

Here
\(
\widetilde C_{m,\alpha,d}
:=
C_{m,\alpha,d}^2 c_\alpha^m,
\quad
c_\alpha=2\Gamma(\alpha)\cos\left(\frac{\pi\alpha}{2}\right).
\)
In particular, \(\widetilde C_{m,\alpha,d}>0\), although
\(C_{m,\alpha,d}\) itself is signed when \(J_m<0\).

\item[(d)] There exist $C_T<\infty$ and
$\beta:=\min\{2H,\,2-\alpha m\}>0$ such that
\[
    \|Q_t-Q_s\|_{L^2(\mathbb R^m)}^2 \le C_T\,|t-s|^{\beta},
    \qquad 0\le s\le t\le T .
\]
Under the stronger condition $2d+\alpha m<1$, one has $\beta>1$.
\end{itemize}

With the normalization chosen above,
\(
\widetilde C_{m,\alpha,d}
=
\frac{1}{m! I_{m,\alpha,d}},
\)
where
\[
I_{m,\alpha,d}
=
\int_0^1\int_0^1
y^{-d}z^{-d}|y-z|^{-\alpha m}\,dy\,dz .
\]

\end{lemma}

\begin{proof}
For \(0\le s\le t\le T\), set
\(
    \Delta_{s,t}(x)
    :=
    Q_t(x)-Q_s(x).
\)
We prove the identity for \(\|\Delta_{s,t}\|_2^2\). Let
\(\chi_\varepsilon(x):=e^{-\varepsilon |x|}\), and define the regularized kernel
\[
    \Delta^{(\varepsilon)}_{s,t}(x)
    :=
    C_{m,\alpha,d}
    \left(
        \int_s^t y^{-d}e^{iy(x_1+\cdots+x_m)}\,dy
    \right)
    \prod_{j=1}^m |x_j|^{(\alpha-1)/2}\chi_\varepsilon(x_j).
\]
For fixed \(\varepsilon>0\), all integrations below are absolutely justified by
Fubini's theorem. Hence
\[
\begin{aligned}
    \|\Delta^{(\varepsilon)}_{s,t}\|_2^2
    &=
    C_{m,\alpha,d}^2
    \int_s^t\int_s^t
    y^{-d}z^{-d}
    \prod_{j=1}^m
    \left[
        \int_{\mathbb R}
        e^{i(y-z)x}
        |x|^{\alpha-1}e^{-2\varepsilon |x|}\,dx
    \right]
    dy\,dz .
\end{aligned}
\]
The one-dimensional Fourier transform of the Riesz kernel
(\cite{SamkoKilbasMarichev}) gives, in the sense of tempered distributions,
\[
    \int_{\mathbb R} e^{iux}|x|^{\alpha-1}\,dx
    =
    c_\alpha |u|^{-\alpha},
    \qquad
    c_\alpha=2\Gamma(\alpha)\cos\left(\frac{\pi\alpha}{2}\right)>0 .
\]
The regularized transforms converge locally uniformly away from \(u=0\) and in
the distributional sense to \(c_\alpha |u|^{-\alpha}\). Since
\(\alpha m<1\), the function \(|y-z|^{-\alpha m}\) is integrable on
\([s,t]^2\). Therefore, by standard dominated convergence for the regularized
Riesz kernels,
\[
    \lim_{\varepsilon\downarrow0}
    \|\Delta^{(\varepsilon)}_{s,t}\|_2^2
    =
    C_{m,\alpha,d}^2 c_\alpha^m
    \int_s^t\int_s^t
    y^{-d}z^{-d}|y-z|^{-\alpha m}\,dy\,dz .
\]
Thus the asserted identity holds with
\(
    \widetilde C_{m,\alpha,d}
    :=
    C_{m,\alpha,d}^2c_\alpha^m .
\)

Taking \(s=0\), we get
\[
    \|Q_t\|_2^2
    =
    \widetilde C_{m,\alpha,d}
    \int_0^t\int_0^t
    y^{-d}z^{-d}|y-z|^{-\alpha m}\,dy\,dz .
\]
The last integral is finite because the diagonal singularity is integrable
under \(\alpha m<1\), while the singularity at the origin is integrable under
\(2d+\alpha m<2\). This proves part \textup{(a)}.

The same identity gives
\[
    \|Q_t-Q_s\|_2^2
    =
    \widetilde C_{m,\alpha,d}
    \int_s^t\int_s^t
    y^{-d}z^{-d}|y-z|^{-\alpha m}\,dy\,dz
    \longrightarrow0
    \qquad (t\to s),
\]
including the case \(s=0\). This proves part \textup{(b)}.

Finally, the set
\(
    \mathcal Q_T:=\{Q_t:0\le t\le T\}
\)
is compact in \(L^2(\mathbb R^m)\), because it is the continuous image of the
compact interval \([0,T]\). For every fixed \(Q_t\), the \(L^2\)-tail over
\(([-R,R]^m)^c\) tends to zero as \(R\to\infty\). Compactness upgrades this
pointwise convergence to uniform convergence in \(t\), proving part
\textup{(c)}.

It remains to prove \textup{(d)}. If $d\ge0$, the map $y\mapsto y^{-d}$
is nonincreasing, so translating the square $[s,t]^2$ to $[0,t-s]^2$
increases the integrand pointwise while leaving $|y-z|^{-\alpha m}$
unchanged; hence, by homogeneity,
\[
    \int_s^t\!\!\int_s^t y^{-d}z^{-d}|y-z|^{-\alpha m}\,dy\,dz
    \le
    \int_0^{t-s}\!\!\int_0^{t-s} y^{-d}z^{-d}|y-z|^{-\alpha m}\,dy\,dz
    =
    I_{m,\alpha,d}\,(t-s)^{2H}.
\]
If $d<0$, then $y^{-d}z^{-d}\le T^{-2d}$ on $[0,T]^2$ and
$\int_s^t\int_s^t|y-z|^{-\alpha m}\,dy\,dz\le C(t-s)^{2-\alpha m}$.
In either case, after enlarging $C_T$ if necessary, the preceding
estimate holds with
\(
    \beta:=\min\{2H,\,2-\alpha m\}>0.
\)
Under the stronger condition $2d+\alpha m<1$, one moreover has
$2H>1$, and hence $\beta>1$.\end{proof}

\begin{lemma}[Compactness of the limiting cluster set]
\label{lem:KT-compact}
The set \(\mathcal K_T\) is compact in \(C[0,T]\). 
\end{lemma}

\begin{proof}
The closed unit ball of
\(\mathfrak H_{\mathbb R}\) is weakly compact. \\ Let
\(
    \Phi(\xi)(t):=\langle Q_t,\xi^{\otimes m}\rangle_{\mathfrak H_{\mathbb R}^{\otimes m}},
    \quad 0\le t\le T .
\)
We claim that \(\Phi\) is weak-to-uniform continuous on the unit ball. Let
\(\xi_n\rightharpoonup\xi\) weakly in \(\mathfrak H_{\mathbb R}\), with
\(\sup_n\|\xi_n\|\le1\). 
The \(L^2\)-continuity of \(t\mapsto Q_t\) follows from the identity
\(
\|Q_t-Q_s\|_2^2
=
C\int_s^t\int_s^t y^{-d}z^{-d}|y-z|^{-\alpha m}\,dy\,dz,
\)
and the integrability of the kernel on \([0,T]^2\) (the same integral as in
Lemma~\ref{lem:two-dimensional-rv-riemann}, with $\beta=\alpha m$).
The \(L^2\)-continuity of \(t\mapsto Q_t\) implies that
\(\mathcal Q_T:=\{Q_t:0\le t\le T\}\) is compact in
\(\mathfrak H_{\mathbb R}^{\odot m}\). We claim that, for every
\(\varepsilon>0\), there is a finite-dimensional subspace
\(E\subset\mathfrak H_{\mathbb R}\) such that
\(
    \sup_{0\le t\le T}
    \|Q_t-P_E^{\odot m}Q_t\|_2<\varepsilon .
\)
Indeed, choose a finite \(\varepsilon/3\)-net
\(Q_{t_1},\ldots,Q_{t_N}\) for \(\mathcal Q_T\). Since
\(
    \overline{\bigcup_{\dim E<\infty}E^{\odot m}}
    =
    \mathfrak H_{\mathbb R}^{\odot m},
\)
for each \(i\) there is a finite-dimensional subspace
\(E_i\subset\mathfrak H_{\mathbb R}\) such that
\(
    \|Q_{t_i}-P_{E_i}^{\odot m}Q_{t_i}\|_2<\varepsilon/3.
\)
Let \(E:=E_1+\cdots+E_N\). Since \(E_i^{\odot m}\subset E^{\odot m}\) and
\(P_E^{\odot m}\) is an orthogonal projection, we have
\(
    \|Q_{t_i}-P_E^{\odot m}Q_{t_i}\|_2<\varepsilon/3,
    \qquad i=1,\ldots,N.
\)
Therefore, for every \(t\in[0,T]\), choosing \(i\) with
\(\|Q_t-Q_{t_i}\|_2<\varepsilon/3\), we obtain
\[
    \|Q_t-P_E^{\odot m}Q_t\|_2
    \le
    \|Q_t-Q_{t_i}\|_2
    +
    \|Q_{t_i}-P_E^{\odot m}Q_{t_i}\|_2
    +
    \|P_E^{\odot m}(Q_{t_i}-Q_t)\|_2
    <\varepsilon .
\]
For the finite-rank family \(P_E^{\odot m}Q_t\), weak convergence of \(\xi_n\) implies
uniform convergence of
\(t\mapsto\langle P_E^{\odot m}Q_t,\xi_n^{\otimes m}\rangle\) to
\(t\mapsto\langle P_E^{\odot m}Q_t,\xi^{\otimes m}\rangle\). The remaining error is bounded
by \(C_m\varepsilon\), uniformly in \(t\). Hence \(\Phi(\xi_n)\to\Phi(\xi)\) in \(C[0,T]\).
Thus \(\Phi\) maps the weakly compact unit ball continuously into \(C[0,T]\), and
\(\mathcal K_T=\Phi(\{\|\xi\|\le1\})\) is compact.
\end{proof}

\begin{lemma}[An integrable low-frequency majorant]\label{lem:integrable-majorant}
Assume \(m\ge2\), and let \(\eta\in(0,\alpha)\) be such that
\(m(\alpha+\eta)<1\). Define
\[
\theta(x):=x_1+\cdots+x_m,
\qquad
G_\eta(x)
:=
\frac{1}{1+|\theta(x)|^2}
\prod_{j=1}^m
\left(
|x_j|^{\alpha-1-\eta}
+
|x_j|^{\alpha-1+\eta}
\right).
\]
Then \(G_\eta\in L^1(\mathbb R^m)\).
\end{lemma}

\begin{proof}
Expand the product into \(2^m\) terms. For each such term, write
\(\beta_j=\alpha+\sigma_j\eta\), where
\(\sigma_j\in\{-1,1\}\), and put
\(B:=\beta_1+\cdots+\beta_m\). Then
\(0<\beta_j<1\) and
\(
0<B\le m(\alpha+\eta)<1.
\)
By the change of variables \(s=x_1+\cdots+x_m\) and the
one-dimensional Riesz-kernel convolution formula
\cite{SamkoKilbasMarichev},
\[
\int_{\mathbb R^m}
\frac{\prod_{j=1}^m |x_j|^{\beta_j-1}}
     {1+|\theta(x)|^2}\,dx
=
c_{\beta_1,\ldots,\beta_m}
\int_{\mathbb R}
\frac{|s|^{B-1}}{1+|s|^2}\,ds
<\infty .
\]
Indeed, \(B>0\) gives integrability at zero, while \(B<1\) gives
integrability at infinity. Hence all \(2^m\) terms are integrable.
\end{proof}

\begin{lemma}[Uniform Abel bound for weighted exponential sums]
\label{lem:uniform-abel-weighted-sum}
Assume \((\ref{A4})\)--\((\ref{A5})\), and let \(d<1\). Fix \(0<\varepsilon<T<\infty\).
For
\[
    R_n^{(\varepsilon)}(t,\theta)
    :=
    \frac{1}{n^{1-d}\ell_a(n)}
    \sum_{u=\lfloor\varepsilon n\rfloor+1}^{\lfloor nt\rfloor}
    a_u e^{iu\theta/n},
    \qquad \varepsilon\le t\le T,
\]
the following holds: for every \(c_0\in(0,2\pi)\), there is \(C_{\varepsilon,T,c_0}<\infty\)
such that, for all large \(n\),
\[
    \sup_{\varepsilon\le t\le T}
    |R_n^{(\varepsilon)}(t,\theta)|
    \le
    C_{\varepsilon,T,c_0}\min\{1,|\theta|^{-1}\},
    \qquad |\theta|\le c_0n .
\]
\end{lemma}

\begin{proof}
Write \(a(x):=x^{-d}\ell_a(x)\). By \((\ref{A5})\),
\[
    a'(x)=x^{-d-1}\ell_a(x)
    \left(-d+\frac{x\ell_a'(x)}{\ell_a(x)}\right),
\]
so, uniformly for \(x\in[\varepsilon n,Tn]\), one has
\[
    |a'(x)|\le C_{\varepsilon,T}n^{-d-1}\ell_a(n),
    \qquad
    |a(x)|\le C_{\varepsilon,T}n^{-d}\ell_a(n).
\]
Hence, by the mean-value theorem,
\[
    \sum_{u=\lfloor\varepsilon n\rfloor}^{\lfloor Tn\rfloor}|a_{u+1}-a_u|
    \le C_{\varepsilon,T}n^{-d}\ell_a(n).
\]

If \(|\theta|\le1\), then
\[
    \sup_{\varepsilon\le t\le T}
    \left|
    \sum_{u=\lfloor\varepsilon n\rfloor+1}^{\lfloor nt\rfloor}
    a_u e^{iu\theta/n}
    \right|
    \le
    C_{\varepsilon,T}n^{1-d}\ell_a(n),
\]
because \(d<1\). This gives the bound with the factor \(1\).

Assume now \(1<|\theta|\le c_0n\). Put \(M:=\lfloor\varepsilon n\rfloor\),
\(N:=\lfloor nt\rfloor\), and \(\Sigma_u(\theta):=\sum_{r=M+1}^u e^{ir\theta/n}\). Since
\(c_0<2\pi\), the geometric-series formula gives
\[
    |\Sigma_u(\theta)|\le C_{c_0}n|\theta|^{-1},
    \qquad M<u\le\lfloor Tn\rfloor .
\]
By Abel summation,
\[
    \sum_{u=M+1}^{N}a_u e^{iu\theta/n}
    =
    a_N\Sigma_N(\theta)
    -
    \sum_{u=M+1}^{N-1}(a_{u+1}-a_u)\Sigma_u(\theta).
\]
Using the bounds above,
\[
    \left|
    \sum_{u=M+1}^{N}a_u e^{iu\theta/n}
    \right|
    \le
    C_{\varepsilon,T,c_0}n^{1-d}\ell_a(n)|\theta|^{-1}.
\]
Dividing by \(n^{1-d}\ell_a(n)\) proves the lemma.
\end{proof}

\begin{proposition}[Uniform full prelimit tail estimate]
\label{prop:full-prelimit-tail}
Assume \((\ref{A1})\)--\((\ref{A6})\), \(m\ge2\), \(\alpha m<1\), and \(2d+\alpha m<1\).
Let \(T<\infty\). Then
\[
    \lim_{R\to\infty}\limsup_{n\to\infty}
    \sup_{0\le t\le T}
    \int_{([-R,R]^m)^c}
    |\widetilde F^{\mathrm{full}}_{n,t}(x)|^2\,dx
    =0 .
\]
\end{proposition}

\begin{proof}
Put \(D_R:=[-R,R]^m\), choose \(\delta\in(0,\pi/m)\), and set
\(B_{\delta,n}:=[-\delta n,\delta n]^m\). We split the tail into the low- and high-frequency parts,
\[
    \int_{D_R^c}|\widetilde F^{\mathrm{full}}_{n,t}|^2
    =
    \int_{D_R^c\cap B_{\delta,n}}|\widetilde F^{\mathrm{full}}_{n,t}|^2
    +
    \int_{D_R^c\cap B_{\delta,n}^c}|\widetilde F^{\mathrm{full}}_{n,t}|^2 .
\]

We first handle the low-frequency part. Fix \(\varepsilon\in(0,T)\). Decompose
\(\widetilde F^{\mathrm{full}}_{n,t}=E_{n,t}^{(\varepsilon)}+L_{n,t}^{(\varepsilon)}\), where
\(E_{n,t}^{(\varepsilon)}\) contains the summands \(1\le u\le\lfloor n(t\wedge\varepsilon)\rfloor\), and
\(L_{n,t}^{(\varepsilon)}\) contains the summands \(\lfloor \varepsilon n\rfloor<u\le\lfloor nt\rfloor\).

For the initial part, by the variance estimate in Proposition~\ref{prop:exact-variance} and regular variation of \(A_n\),
\[
    \sup_{0\le t\le T}\|E_{n,t}^{(\varepsilon)}\|_{L^2(\mathbb R^m)}^2
    \le
    C\frac{A_{\lfloor\varepsilon n\rfloor}^2}{A_n^2}
    \le
    C_\eta \varepsilon^{2H-\eta}
\]
for every small \(\eta>0\), uniformly in all large \(n\). Thus the initial part can be made arbitrarily small by letting \(\varepsilon\downarrow0\).

It remains to estimate \(L_{n,t}^{(\varepsilon)}\). For \(t\le\varepsilon\) this term is zero. For \(t>\varepsilon\), write
\[
    L_{n,t}^{(\varepsilon)}(x)
    =
    \gamma_n R^{(\varepsilon)}_n(t,\theta(x))\Psi_n(x),
    \qquad
    \theta(x):=x_1+\cdots+x_m,
\]
where
\[
    R^{(\varepsilon)}_n(t,\theta)
    :=
    \frac{1}{n^{1-d}\ell_a(n)}
    \sum_{u=\lfloor\varepsilon n\rfloor+1}^{\lfloor nt\rfloor}
    a_u e^{iu\theta/n}.
\]
By Lemma~\ref{lem:uniform-abel-weighted-sum}, for every \(c_0\in(0,2\pi)\),
\[
    \sup_{\varepsilon\le t\le T}|R^{(\varepsilon)}_n(t,\theta)|
    \le C_{\varepsilon,T,c_0}\min\{1,|\theta|^{-1}\},
    \qquad |\theta|\le c_0n .
\]

Choose \(\eta\in(0,\alpha)\) so small that
\(m(\alpha+\eta)<1\). Decreasing \(\delta>0\) if necessary, assume
that \(m\delta<c_0\), that \(1/\delta\) lies beyond the threshold in
Potter's bound, and, by \((\ref{A2})\), that
\[
f(\lambda)
\le
C|\lambda|^{\alpha-1}L_0(1/|\lambda|),
\qquad
0<|\lambda|\le\delta .
\]
Then Potter's bound gives, uniformly for
\(0<|x|\le\delta n\) and all sufficiently large \(n\),
\[
|\psi_n(x)|^2
\le
C|x|^{\alpha-1}
\frac{L_0(n/|x|)}{L_0(n)}
\le
C_\eta
\left(
|x|^{\alpha-1-\eta}
+
|x|^{\alpha-1+\eta}
\right).
\]
Consequently,
\[
|\Psi_n(x)|^2
\le
C_\eta
\prod_{j=1}^m
\left(
|x_j|^{\alpha-1-\eta}
+
|x_j|^{\alpha-1+\eta}
\right),
\qquad
x\in B_{\delta,n},
\]
almost everywhere. Since \(\gamma_n\) is bounded and
Lemma~\ref{lem:uniform-abel-weighted-sum} applies on
\(B_{\delta,n}\), for all sufficiently large \(n\),
\[
|L_{n,t}^{(\varepsilon)}(x)|^2
\le
C_{\varepsilon,T,\eta}
\frac{1}{1+|\theta(x)|^2}
\prod_{j=1}^m
\left(
|x_j|^{\alpha-1-\eta}
+
|x_j|^{\alpha-1+\eta}
\right),
\qquad
x\in B_{\delta,n},\quad 0\le t\le T .
\]
By Lemma~\ref{lem:integrable-majorant}, the right-hand side is
integrable on \(\mathbb R^m\). 
Hence
\[
    \lim_{R\to\infty}\limsup_{n\to\infty}
    \sup_{0\le t\le T}
    \int_{D_R^c\cap B_{\delta,n}}
    |L_{n,t}^{(\varepsilon)}(x)|^2\,dx
    =0 .
\]
Combining this with the initial-block estimate gives
\[
    \lim_{R\to\infty}\limsup_{n\to\infty}
    \sup_{0\le t\le T}
    \int_{D_R^c\cap B_{\delta,n}}
    |\widetilde F^{\mathrm{full}}_{n,t}(x)|^2\,dx
    \le
    C_\eta\varepsilon^{2H-\eta}.
\]
Letting \(\varepsilon\downarrow0\) proves the low-frequency tail estimate.

We now handle the high-frequency part. Since \(D_R^c\cap B_{\delta,n}^c\subset B_{\delta,n}^c\), the change of variables \(x=n\lambda\) gives
\[
    \int_{D_R^c\cap B_{\delta,n}^c}
    |\widetilde F^{\mathrm{full}}_{n,t}(x)|^2\,dx
    \le
    \frac{1}{A_n^2}
    \int_{[-\pi,\pi]^m\setminus[-\delta,\delta]^m}
    |F^{\mathrm{lin}}_{n,t}(\lambda)|^2\,d\lambda .
\]
The linear interpolation changes this expression only by a term bounded by
\(C A_n^{-2}\sup_{u\le Tn}a_u^2\), which tends to zero. It is therefore enough to consider \(F_{\lfloor nt\rfloor}\).

Using \(1_{[-\pi,\pi]^m\setminus[-\delta,\delta]^m}\le\sum_{j=1}^m1_{\{|\lambda_j|>\delta\}}\), fix \(j\). For \(N:=\lfloor nt\rfloor\),
\[
    F_N(\lambda)
    =
    \frac{J_m}{m!}
    \sum_{u=1}^{N}a_u e^{iu(\lambda_1+\cdots+\lambda_m)}
    \prod_{\ell=1}^m q(\lambda_\ell).
\]

By Fubini's theorem and by integrating first over the remaining \(m-1\) spectral variables,
\[
    \frac{1}{A_n^2}
    \int_{\{|\lambda_j|>\delta\}}|F_N(\lambda)|^2\,d\lambda
    \le
    C_m A_n^{-2}
    \sum_{u=1}^{N}\sum_{v=1}^{N}|a_u a_v|\,b_\delta(u-v),
\]
where \(b_\delta(h):=|\rho(h)|^{m-1}|r_\delta(h)|\) and
\(r_\delta(h):=\int_{|\lambda|>\delta}e^{ih\lambda}f(\lambda)\,d\lambda\).
Since \(f\in W^{1,1}([-\pi,-\delta]\cup[\delta,\pi])\), integration by parts gives
\(|r_\delta(h)|\le C_\delta(1+|h|)^{-1}\). By Potter's bound,
\(|\rho(h)|^{m-1}\le C_\eta(1+|h|)^{-\alpha(m-1)+\eta}\) for every small \(\eta>0\).
Choosing \(\eta<\alpha(m-1)\), we get
\(b_\delta(h)\le C(1+|h|)^{-1-\alpha(m-1)+\eta}\), hence \(b_\delta\in\ell^1(\mathbb Z)\).

Let \(c_u^{(n,t)}:=|a_u|1_{\{1\le u\le\lfloor nt\rfloor\}}\). By Young's inequality on \(\ell^2\),
\[
    \sum_{u=1}^{N}\sum_{v=1}^{N}|a_u a_v|\,b_\delta(u-v)
    \le
    \|b_\delta\|_{\ell^1}\sum_{u=1}^{\lfloor Tn\rfloor}a_u^2 .
\]
Since \(2d<1\), Karamata's theorem yields
\(\sum_{u=1}^{\lfloor Tn\rfloor}a_u^2\le C_T n^{1-2d}\ell_a(n)^2\). Hence
\[
    A_n^{-2}\sum_{u=1}^{\lfloor Tn\rfloor}a_u^2
    \le
    C_T n^{-1+\alpha m}L_0(n)^{-m}\longrightarrow0,
\]
because \(\alpha m<1\). Therefore
\[
    \lim_{n\to\infty}
    \sup_{0\le t\le T}
    \int_{D_R^c\cap B_{\delta,n}^c}
    |\widetilde F^{\mathrm{full}}_{n,t}(x)|^2\,dx
    =0
\]
for every fixed \(R>0\). Combining the low- and high-frequency estimates proves the proposition.
\end{proof}




\begin{thebibliography}{99}

\bibitem[Adler and Taylor(2007)]{AdlerTaylor2007}
Adler, R. J. and Taylor, J. E. (2007).
\emph{Random Fields and Geometry}.
Springer, New York.

\bibitem[Arcones(1999)]{Arcones1999}
Arcones, M. A. (1999).
The law of the iterated logarithm over a stationary Gaussian sequence of random vectors.
\emph{Journal of Theoretical Probability}, 12, 615--641.

\bibitem[Azmoodeh, Peccati and Poly(2016)]{AzmoodehPeccatiPoly2016}
Azmoodeh, E., Peccati, G. and Poly, G. (2016).
The law of iterated logarithm for subordinated Gaussian sequences:
uniform Wasserstein bounds.
\emph{ALEA Lat. Am. J. Probab. Math. Stat.} \textbf{13}, 659--686.

\bibitem[Beran et al.(2013)]{BFGK2013}
Beran, J., Feng, Y., Ghosh, S. and Kulik, R. (2013).
\emph{Long-Memory Processes: Probabilistic Properties and Statistical
Methods}.
Springer, Heidelberg.

\bibitem[Bingham, Goldie and Teugels(1987)]{BGT1987}
Bingham, N. H., Goldie, C. M. and Teugels, J. L. (1987).
\emph{Regular Variation}. Encyclopedia of Mathematics and its Applications 27.
Cambridge University Press, Cambridge.

\bibitem[Dehling and Taqqu(1988)]{DehlingTaqqu1988}
Dehling, H. and Taqqu, M. S. (1988).
The functional law of the iterated logarithm for the
empirical process of some long-range dependent sequences.
\emph{Statistics \& Probability Letters} \textbf{7}, 81--85.

\bibitem[Janson(1997)]{Janson1997}
Janson, S. (1997).
\emph{Gaussian Hilbert Spaces}.
Cambridge Tracts in Mathematics 129.
Cambridge University Press, Cambridge.

\bibitem[Kochen and Stone(1964)]{KochenStone1964}
Kochen, S. and Stone, C. (1964).
A note on the Borel--Cantelli lemma.
\emph{Illinois Journal of Mathematics}, 8, 248--251.

\bibitem[Lai and Stout(1978)]{LaiStout1978}
Lai, T. L. and Stout, W. F. (1978).
The law of the iterated logarithm and upper-lower class tests for partial sums of stationary
Gaussian sequences.
\emph{The Annals of Probability}, 6, 731--750.

\bibitem[Li and Tomkins(1996)]{LiTomkins1996}
Li, D. and Tomkins, R. J. (1996).
Laws of the iterated logarithm for weighted sums of independent random variables.
\emph{Statistics \& Probability Letters} \textbf{27}, 247--254.

\bibitem[Moldavskaya(2009)]{Moldavskaya2009Lviv}
Moldavskaya, E. (2009).
Approximate representation for the estimators of parameters in the regression
models with LRD and restrictions.
In \emph{Stochastic Analysis and Random Dynamical Systems. International
Conference. Abstracts}, Lviv, Ukraine, June 14--20, 2009, p.~171.
Institute of Mathematics, National Academy of Sciences of Ukraine, Kyiv.

\bibitem[Mori and Oodaira(1986)]{MoriOodaira1986}
Mori, T. and Oodaira, H. (1986).
The law of the iterated logarithm for self-similar processes represented by multiple Wiener
integrals.
\emph{Probability Theory and Related Fields}, 71, 367--391.

\bibitem[Mori and Oodaira(1987)]{MoriOodaira1987}
Mori, T. and Oodaira, H. (1987).
The functional iterated logarithm law for stochastic processes represented by multiple Wiener
integrals.
\emph{Probability Theory and Related Fields}, 76, 299--310.

\bibitem[Samko, Kilbas and Marichev(1993)]{SamkoKilbasMarichev}
Samko, S. G., Kilbas, A. A. and Marichev, O. I. (1993).
\emph{Fractional Integrals and Derivatives: Theory and Applications}.
Gordon and Breach Science Publishers, Yverdon.

\bibitem[Taqqu(1977)]{Taqqu1977}
Taqqu, M. S. (1977).
Law of the iterated logarithm for sums of non-linear functions of Gaussian variables that
exhibit a long range dependence.
\emph{Zeitschrift f\"ur Wahrscheinlichkeitstheorie und Verwandte Gebiete}, 40, 203--238.

\end{thebibliography}
\end{document}